\documentclass[reqno,english,11pt]{amsart}

\usepackage[utf8]{inputenc}       % encodage
\usepackage[T1]{fontenc}
\usepackage{microtype}            % petites améliorations typographiques

\usepackage{geometry}
\usepackage{amsmath,amssymb,amsthm,mathtools}
\usepackage{mathrsfs}             % \mathscr
\usepackage{bm}                   % vecteurs gras mathématiques
\numberwithin{equation}{section}  % équations numérotées par section

\usepackage[notref,notcite,final]{showkeys}
\mathtoolsset{showonlyrefs}

\theoremstyle{plain}
\newtheorem{theorem}{Theorem}[section]
\newtheorem{lemma}[theorem]{Lemma}
\newtheorem{proposition}[theorem]{Proposition}
\newtheorem{corollary}[theorem]{Corollary}

\theoremstyle{plain}
\newtheorem{definition}[theorem]{Definition}

\newtheorem{remark}[theorem]{Remark}

\usepackage{graphicx}
\usepackage{caption}
\usepackage{subcaption}
\usepackage{tikz}      % dessins vectoriels
\usepackage{booktabs}  % tables soignées

\usepackage{enumitem}
\setlist[itemize]{leftmargin=*, itemsep=2pt}
\setlist[enumerate]{leftmargin=*, itemsep=3pt}

\usepackage[colorlinks=true, linkcolor=blue, citecolor=blue, urlcolor=blue]{hyperref}

\def\whF{\widehat{{\mathcal{F}}}}
\def\jx{\langle x \rangle}
\def\jy{\langle y \rangle}
\def\jk{\langle k \rangle}

\def\#{\sharp}

\def\im{\text{Im }}

\def\R{\mathbb{R}}
\def\C{\mathbb{C}}

\def\K{\mathcal{K}}
\def\e{\epsilon}

\def\jt{\langle t \rangle}
\def\js{\langle s \rangle}

\def\pv{\mathrm{p.v.}}

\def\sgn{\mathrm{sgn}}
\def \ran {\text{Ran}}

\def\what{\widehat}
\def\Fou{\mathcal{F}}

\usepackage{mathtools}
\title[NLS with a non-generic potential in 1D]{Long time behavior of small solutions of NLS \\ with non-generic potentials in one dimension 
 }

\author{Neba Polneau}
\address{Institut de Mathématiques de Jussieu–Paris Rive Gauche,
CNRS, Sorbonne Université, Université Paris Cité,
4 place Jussieu, 75005 Paris, France}
\email{polneau@imj-prg.fr}
\date{\today} % ou fixer : 12 septembre 2025
\keywords{Nonlinear Schr\"odinger equation, distorted Fourier transform, resonance}
\thanks{This research was supported by the ERC project INSOLIT (No. 101117126), by NSF Grants DMS-2350301 and CAREER-DMS-2540992, by the Simons Foundation MP-TSM-00002258. The author would like to thank Jacek Jendrej and Gong Chen for helpful discussions, and Andrew Lawrie for his encouragement and support.}

\begin{document}

\begin{abstract}
We consider the one-dimensional cubic nonlinear Schr\"odinger equation with a non-generic real-valued external potential $V$. We prove almost global-in-time quantitative bounds for small solutions. More precisely, small initial data of size $\varepsilon$ in a weighted Sobolev space give rise to solutions with the sharp decay rate $t^{-1/2}$ in $L^{\infty}_x$ up to time $\exp(\frac{1}{c\varepsilon^{2}})$. The main novelty of our result is that no additional symmetry assumption is imposed on $V$.

First, we use a modification of the standard distorted Fourier transform basis to resolve the possible discontinuity at zero energy due to the presence of a resonance. Then, following the work of Chen and Pusateri \cite{CPNLS2}, we use smoothing estimates in the setting of non-generic potentials to analyze the low frequency structure of the (modified) nonlinear spectral distribution. A key novel ingredient is a Fourier restriction type inequality that handles low frequency contributions not amenable to the approach of \cite{CPNLS2}, and which is central to establishing the quantitative bounds. 

%In addition, we use a "Fourier restriction" type inequality to resolve some issues in the low frequency structure we cannot handle with the approach in \cite{CPNLS2}. That was the key trick in establishing the quantitative bounds.
\end{abstract}

\maketitle

\setcounter{tocdepth}{1}
\tableofcontents

%-----------------------------------------------

\section{Introduction}
We consider the one-dimensional cubic nonlinear Schrödinger equation (NLS)
\begin{align} \label{NLS}
    \begin{cases}
        i\partial_t u - \partial_{xx} u + V(x)u \pm |u|^2u = 0  \\
        u(0,x) = u_0(x)
    \end{cases} 
\end{align}
for an unknown function $u : \R_t \times \R_x \to \C$, with initial data $u_0$ in $H^1(\R) \cap L^2(\R, \jx^2 \, dx)$ and a real-valued potential $V \in L^1(\R, \jx^\gamma \, dx)$ for some $\gamma$. We assume that $V$ is \emph{non-generic}, in the sense that there exists a non-trivial $L^\infty$ bounded function $\varphi = \varphi(x)$ satisfying $H\varphi = 0$, where $H := -\partial_x^2 + V(x)$ is the associated Schr\"odinger operator. Our goal is to understand the long-time behavior of small solutions to this equation. 

\smallskip
\subsection{Known facts}
First of all, for small data in $H^1(\R)$ the equation \eqref{NLS} is well-posed in $H^1(\R)$ and has a global solution regardless of the sign $\pm$ in front of the nonlinear term. This is a consequence of the semigroup theory and conservation of  the mass 
\begin{align}\label{mass}
    M(u) := \int\left|u\right|^{2}\,dx 
\end{align}
and the total energy (Hamiltonian)
\begin{align}\label{Ham}
H(u) := \int\frac{1}{2}|\partial_xu|^{2}
  +\frac{1}{2}V|u|^{2}\pm\frac{1}{4}|u|^{4} \, dx.
\end{align}
%\col{The question is now to give a more precise description of the solutions when $|t|\to +\infty$, especially quantitative bounds and asymptotics. Due to some dispersive features of the equation, the solutions tend to exhibit  scattering behavior at large times. On this question of scattering, the cubic nonlinearity is critical and we do not have standard scattering compared to higher order of nonlinearity, and the one dimensional case makes things harder in that situation because wave packets tend to spread more slowly in time (heuristically one can say they have less space to spread out). \\
%In the "free" case ($V = 0$), Hayashi and Naumkin (see \cite{HN}) give a precise description of the asymptotics behavior with a $\log t$ correction in the phase of the asymptotic "profile". There exist several proofs of this modified scattering result. We can consult the expository paper of Murphy (\cite{Mur}). The more recent proof due to Kato and Pusateri (see \cite{KP}) uses the method of space-time resonances introduced by Germain, Masmoudi and Shatah in their paper \cite{GMS}. This method was also used to study the behavior in the presence of an external potential $V(x)$. In the case where $V$ is generic, Germain, Rousset, Pusateri give a proof of modified scattering (see \cite{GPR}). The non-generic case was treated by Chen and Pusateri (see \cite{CPNLS2}) under some symmetry assumptions on the zero-energy resonance. This work considers the non-generic case, regardless of symmetry assumptions.}

We now turn to a more precise description of the behavior of solutions as $|t|\to +\infty$, seeking both quantitative bounds and sharp asymptotics. Due to the dispersive nature of the equation, solutions are expected to exhibit scattering behavior at large times. However, the cubic nonlinearity is critical in this regard: unlike higher-order nonlinearities, it does not admit standard scattering, and the one-dimensional setting further complicates the analysis, as wave packets spread more slowly in time. Heuristically, they have less "room" to disperse.

In the flat case ($V = 0$), Hayashi and Naumkin \cite{HN} established a precise description of the asymptotic behavior, revealing a logarithmic correction in the phase of the asymptotic profile, a phenomenon known as modified scattering. Several alternative proofs of this result have since appeared; we refer the reader to the expository paper of Murphy \cite{Mur}. A more recent proof, due to Kato and Pusateri \cite{KP}, relies on the method of space-time resonances introduced by Germain, Masmoudi, and Shatah \cite{GMS}. This method was adapted to the case of an external potential by Germain, Pusateri, and Rousset in \cite{GPR}. In the non-generic case, Chen and Pusateri \cite{CPNLS2} obtained global-in-time results under additional assumptions on the zero-energy resonance (see Definition \ref{0reson}). They established modified scattering if the zero-energy resonance is either odd or even. In particular, this includes the case of even potentials. In fact, the arguments in \cite{CPNLS2} work under a slightly weaker assumption on the zero-energy resonance, namely its limit at $\pm\infty$ are either equal or opposite in sign (see Remark \ref{remarkdecompmu}). The present work removes these additional symmetry-type restrictions. We obtain almost global-in-time quantitative bounds similar to the result of Murphy and Pusateri in \cite{MurPus}.

\smallskip
\subsection{Main result}\label{THM}
The goal of this paper is to prove the following result.
\begin{theorem}\label{mainthm}
Consider the nonlinear Schrödinger equation  \eqref{NLS}  with a real-valued potential $V$ satisfying the following assumptions:
\begin{enumerate} [label = (H\arabic*)]
\item  $V$ is non-generic
\item  $-\partial_{xx} + V$ has no eigenvalues 
\item \label{VassumeWei}
$V \in L^1(\R, \jx^\gamma dx)$ for some $\gamma > \frac 52$.
\end{enumerate}
Then there exists $\varepsilon_{0} > 0$ %only depending on prescribed constants
such that for all $\varepsilon\leq\varepsilon_{0}$ and $u_0$ with
\begin{align}\label{smalldata}
\|u_{0}\|_{H^{1}} +\|xu_{0}\|_{L^{2}} = \varepsilon
\end{align}
the equation \eqref{NLS} has a unique global solution $u\in C(\R,H^1(\R))$, with $u(0,x)=u_{0}(x)$ and this solution satisfies the linear decay rate
\begin{align}\label{main1fdecay}
\|u(t)\|_{L^\infty_x} \lesssim \varepsilon \jt^{-\frac{1}{2}} 
\end{align}
for all $t \in [-T_\varepsilon , T_\varepsilon]$ where $T_\varepsilon = \exp(\frac{1}{c\varepsilon^2})$ with $c > 0$ depending only on the potential $V$.\\
Moreover, if we define the profile of the solution $u$ as
\begin{align}\label{main1prof}
f(t,x) := e^{-itH}u(t,x), \quad H := -\partial_{xx}+V,
\end{align} 
then for all $t \in [-T_\varepsilon , T_\varepsilon]$, the following bound holds:
:
\begin{align}\label{main1fbounds2}
{\big\| (\mathcal{F}^\sharp f) (t) \big\|}_{L_k^\infty} + \jt^{-1/4} 
  {\big\| \partial_{k} (\mathcal{F}^\sharp f) (t) \big\|}_{L_k^2} \lesssim \varepsilon,
\end{align}
where $\mathcal{F}^\sharp$ is the modified distorted Fourier transform (see \eqref{matK} and \eqref{sharpF}).
\end{theorem}

\begin{remark}[Quintic case] Using the same strategy, one can prove with the same hypothesis on $V$ and $u_0$, that the solution of the equation
    \begin{align} \label{NLS5}
    \begin{cases}
        i\partial_t u - \partial_{xx} u + V(x)u \pm |u|^4u = 0  \\
        u(0,x) = u_0(x)
    \end{cases} 
\end{align}
exhibits global bounds for the profile in distorted Fourier setting and  standard scattering of the solution in $H^1(\R)$.
    
\end{remark}

\subsection{Ideas of the proof}
\subsubsection{Bootstrap space}
We perform a bootstrap argument to prove theorem \ref{mainthm}. In fact, we work under the hypothesis that  
\begin{align} \label{boothyp}
    \|u\|_{T} : =\underset{t \in [-T,T]}{\sup}\Big({\big\| (\mathcal{F}^\sharp f) (t) \big\|}_{L_k^\infty} + \jt^{-1/4} 
  {\big\| \partial_{k} (\mathcal{F}^\sharp f) (t) \big\|}_{L_k^2} \Big) = 2C_0\varepsilon 
\end{align}
with $C_0$ a positive constant depending only on $V$ to be determined and $T \in (0 , T_\varepsilon]$. Our analysis will show that $\|u\|_{T} \le C_0 \varepsilon$, which by standard arguments yields the desired bounds on the time interval $[-T_\varepsilon,T_\varepsilon]$, see the proof of the main theorem in Section \ref{propgoalproof} for details. The choice of this bootstrap space comes naturally from the linear dispersion estimate 
\begin{align}
    \|u(t)\|_{L^\infty_x} \lesssim \jt^{-1/2}\Big({\big\| (\mathcal{F}^\sharp f) (t) \big\|}_{L_k^\infty} + \jt^{-1/4} 
  {\big\| \partial_{k} (\mathcal{F}^\sharp f) (t) \big\|}_{L_k^2} \Big)
\end{align} 
which will be proved in Lemma \ref{lem:pointwiseH}.

\subsubsection{Linear theory}
The analysis of this problem highlights the importance of having a distorted Fourier transform that is continuous at zero frequency, particularly when handling nonlinear interactions in the low-frequency regime. First of all, the distorted Fourier transforms serve to diagonalize self-adjoint Schrödinger operators and define unitary isomorphisms from $L^2(\R)$ to $L^2(\R)$. We know that such operators admit Schwartz kernels, so one can formally write 
$$ \Fou{(\varphi)}(k) = \int_{\R} \overline{\K(x,k)}\varphi(x) dx $$
Requiring that this transform diagonalizes $H$, the function $ x \mapsto \overline{\K(x,k)}$ must be a generalized eigenfunction of $H$, associated with the spectral value $k^2$. Thus, constructing a distorted Fourier transform reduces to selecting an appropriate basis of generalized eigenfunctions.
In the non-generic case, the kernel of the standard distorted Fourier transform may exhibit a discontinuity at zero frequency (see Proposition \ref{dftdiscont}). To resolve this, we introduce a modification of the distorted Fourier basis, consisting of a well-chosen unitary transformation thereof, see \eqref{matK} and \eqref{matKunitary}, so that Plancherel-type identities are preserved.

We then study the dynamics of the profile $f = e^{-itH}u$, for which we establish in Section \ref{seclinest} the following linear estimates: 
\begin{itemize}
    \item dispersive decay as \eqref{eq:linearpoinwiseH}
    \item smoothing estimates in Lemma \ref{lem:smoothingsim};
    \item local decay (Lemma \ref{lem:lowlocaldecay}) and $L^2$ improved local decay (Corollary \ref{cor:locdecdiffflow}). 
\end{itemize}

\subsubsection{Nonlinear spectral distribution}
We begin the nonlinear analysis by writing the Duhamel formulation for the profile $f$ in the distorted frequency space. For any function $\varphi$, we use the convention $\varphi^\sharp : = \Fou^\sharp(\varphi)$. Since $f^\sharp (t,k) = e^{-itk^2}u^\sharp(t,k)$, this  gives us the  nonlinear equation 
\begin{align}\label{introD}
f^{\#}(t,k) & = f^{\#}(0,k) + i \int_{0}^{t} \mathcal{N}_{\mu^\sharp}[f^\sharp,f^\sharp,f^\sharp](s,k) \, ds,
\end{align}
with
\begin{align}
\mathcal{N}_{\mu^\sharp}[f^\sharp,f^\sharp,f^\sharp](s,k) & := \!\! \underset{\R^3}{\iiint} e^{is(-k^2+\ell^2-m^2+n^2)}
f^{\#}(s,\ell)\overline{f^{\#}(s,m)}f^{\#}(s,n) \, \mu^\sharp(k,\ell,m,n) \,\frac{dn}{2\pi}\frac{dm}{2\pi}\frac{d\ell}{2\pi}
\end{align}
and
%the profile $\wt{f}(t,k) = e^{it|k|^2} \wt{u}(t,k)$ for a solution $u$ of \eqref{eq:NLSE}, and 
\begin{align}\label{intromu}
\mu^{\#}(k,\ell,m,n):=\int_\R\overline{\mathcal{K}^{\#}(x,k)}\mathcal{K}^{\#}(x,\ell)
  \overline{\mathcal{K}^{\#}(x,m)}\mathcal{K}^\sharp(x,n)\,dx,
\end{align}
which is a tempered distribution on $\R^4$. We call this distribution the {\it nonlinear spectral distribution} (NSD).
We decompose it into 4 pieces  as follows
\begin{align}\label{intromudec}
\mu^{\#}(k,\ell,m,n) = \mu^\#_0(k,\ell,m,n) + \mu^{\#}_L(k,\ell,m,n) +  \mu^{\#}_{\pv} (k,\ell,m,n)  + \mu^{\#}_R(k,\ell,m,n).
\end{align}
%where  $\mu^{\#}_L$ vanishes if one of $k,\ell,m,n$ is $0$ and  $\mu^{\#}_R$ is a regular term.
The first part is a $\C-$linear combination of zero order distributions involving \emph{Dirac delta} terms. 
The term $\mu^{\#}_L$ is referred to as the "improved low-frequency" part of $\mu^{\#}$, as the  singularity at $k =\ell = m = n= 0$ vanishes. We call $\mu^{\#}_{\pv}$ the "dangerous" $\pv$ part, because of a lack of improvement of the low frequency in this part, and we call $\mu^{\#}_R$ the "regular" part of $\mu^{\#}$ since it is a function rather than a distribution. This decomposition generalizes the one carried out in \cite{CPNLS2}; notably, the term $\mu^\#_{\pv}$, which is absent in the setting of \cite{CPNLS2}, appears here as an additional contribution. This decomposition follows from a refined analysis of the modified kernel
$$\mathcal{K}^\sharp(x,k) = \mathcal{K}^\sharp_S(x,k) + \mathcal{K}^\sharp_R(x,k). $$
The term $\K_S^\sharp$ splits into two contributions: one capturing the main effect of the zero-energy resonance, which does not vanish at $k=0$, and one consisting of a linear combination of exponentials $e^{\pm i xk}$ with coefficients that are smooth enough
functions depending only on $k$ and vanish at $k=0$. $\K_R^\sharp$ is the part that arises from the interaction with the potential, it has strong localization in $x$ and is regular in $k$ (see \eqref{Ksharpdecomp} and \eqref{KsharpSdecomp}-\eqref{kR}).

The decomposition \eqref{intromudec} of the NSD induces a corresponding decomposition of the nonlinear interactions $\mathcal{N}_{\mu^\sharp}$. i.e. 
\begin{align}
   \mathcal{N}_{\mu^\sharp}[f^\sharp,f^\sharp,f^\sharp] =  \sum_{\ast\in\{0; \pv; L; R \}}\mathcal{N}_{\ast}
\end{align}
where, for $\ast\in\{0; \pv; L; R \}$,  
\begin{align}\label{Nastintro}
	\begin{split} 
	\mathcal{N}_{\ast}(s,k) := \underset{\R^3}{\iiint} e^{is(-k^2+\ell^2-m^2+n^2)}
	f^{\#}(s,\ell)\overline{f^{\#}(s,m)}
	f^{\#}(s,n) \ \mu^{\#}_{\ast}(k,\ell,m,n)\,\frac{dn}{2\pi}\frac{dm}{2\pi}\frac{d\ell}{2\pi}
	\end{split}
	\end{align} 

\subsubsection{Nonlinear estimates}

 We start with the $L^2$ estimates.  We do not  have  much trouble with the nonlinear part $\mathcal{N}_{0} $ since it corresponds, in some sense, to the flat case $(V = 0)$. The part involving the most singular features at low frequency are  $\mathcal{N}_{\pv} $ and $\mathcal{N}_{L}$. To treat this part, we use the strategy of Chen and Pusateri, based on a commutation identity between $\partial_k$ and the trilinear form that appear in the nonlinear interactions $\mathcal{N}_{\ast}$. By performing the flat (that is, non-distorted) Fourier transform, we are led to treat terms of the form $a(x)s|u|^2u$ with a  localized function $a(x)$. But, the $t^{-1/2}$ sharp  $L^\infty$ dispersive decay of $u$ is not sufficient to perform a direct time integration and hope to close the bootstrap scheme. This issue is resolved in \cite{CPNLS2} using smoothing estimates obtained by exploiting the time oscillation of the Schr\"odinger flow (see Lemma \ref{lem:smoothingsim})
%The other type of estimate that we use %in combination with the improved low frequency behavior of $\mu^\sharp_L$
and enhanced local decay for low frequency improved Schrödinger flows (see Lemma \ref{lem:lowlocaldecay}).
The main smoothing estimate used can be stated as follows:
assuming $\mathcal{Q}(x,k)$ is a bounded function, and letting $\phi=\phi(k)$ be a  
function such that $|\phi(k)|\lesssim \sqrt{|k|}$, we have
\begin{align}\label{eq:smoothingintro}
\left\Vert \int_0^t \left[ \int_\R e^{-ik^{2}s} \phi(k) \mathcal{Q}(y,k)F(s,y)\,dy\right]
	  \,ds\right\Vert _{L_{k}^{2}}\lesssim & \left\Vert F\right\Vert _{L_{y}^{1}L_{s}^{2}([0,t])}.
\end{align}

This estimate is established via the (flat) Fourier transform in time. However, this strategy breaks down for $\mathcal{N}_{\pv}$, due to the absence of low-frequency improvement, which is essential in applying the smoothing estimates above. To handle $\mathcal{N}_\pv$, we also exploit the commutation identity, as in the treatment of $\mathcal{N}_L$ in \cite{CPNLS2}, but the estimation techniques differ. We resolve the difficulty through explicit refined (flat) Fourier-in-time computations involving the Schrödinger multiplier. The key ingredient and main novelty of this analysis is the following inequality (see Lemma \ref{k^2lemma}) stated as follows: for a Schwartz function $F$ on the real line, \begin{align}
        \Big( \int_\R |\widehat{F}(k^2)|^2 \ dk \Big)^{1/2} \lesssim ||F||_{L^{4/3}} \ .
    \end{align}

For the regular part $\mathcal{N}_R$, the analysis is essentially to estimate the $L^2_k$ norm of the term of the form
\begin{align*}
    \int_0^t e^{-isk^2}sk \int_\R
 \mathcal{Q}(x,k)u_1(s,x)\overline{u_2(s,x)}u_3(s,x) \ dx ds \ .  \end{align*}
 The kernel $\mathcal{Q}(x,k)$ is either well localized in $x$ or just bounded. In the latter case, one of the $u_j$ is well localized. Moreover, the different $u_j$ decay as Schrödinger free flow. The $k$ growth is also a challenge. We treat it by performing a low$\backslash$high frequency analysis in order to use the smoothing estimate \eqref{eq:smoothingintro}. The low frequency regime is direct. But, for the high frequencies, we need integration by parts to absorb the growth in $k$. After that, we use the smoothing estimate combined with the improved local decay of the differentiated Schr\"odinger flow (see Lemma~\ref{lem:localderivative}).

% For the regular part $\mathcal{N}_R$, we observe that it can be viewed broadly as a nonlinear expression of the form $\widehat{\Fou}\big(e^{isH} a(x)\, |u(s)|^{2}u(s)\big)$, where $a(x)$ is a localized function. Applying $\partial_k$ to $\mathcal{N}_R[f^\sharp,f^\sharp,f^\sharp]$ makes a factor of $s k$ appear in the computations. We proceed in a low/high frequency analysis. In fact, when $|k|$ is small, this factor allows us to invoke the smoothing estimate 
% \eqref{eq:smoothingintro} directly to obtain the desired bound. For frequencies $|k|\gg 1$, the estimate \eqref{eq:smoothingintro} is no longer immediately applicable. However, by performing an integration by parts, we can absorb the growing factor of $k$, after which we use the smoothing estimate \eqref{eq:smoothingintro}, in the form given in \eqref{eq:smoothingQimh} (corollary \ref{cor:smoothing}) together with the enhanced local decay available for the differentiated Schr\"odinger flow (see Lemma~\ref{lem:localderivative}).\\ \\ \\

\smallskip
For the $L^\infty$ estimates,  the regular part $\mathcal{N}_R$ contributes to remainder in the asymptotics. The bound is obtained with the dispersive estimates we establish in Subsection \ref{ssecdisp}. Therefore, the asymptotic dynamics is encoded in the singular parts $\mathcal{N}_0$, $\mathcal{N}_\pv$ and $\mathcal{N}_L$. We handle it via stationary phase lemmas for oscillatory integrals with a quadratic phase and a distributional symbol. We mean quantities as 
\begin{align*}
\begin{split}
 & \underset{\R^2}{\iint} e^{it \Phi(k,p,m,n)}
	g_1(\epsilon_1(\epsilon_0 k  + \epsilon_2 m - \epsilon_3 n)) \overline{g_2(m)} 
	g_3(n) \, dm dn  \ ,\\
& \underset{\R^3}{\iiint} e^{it \Phi(k,p,m,n)}
	g_1(\epsilon_1(\epsilon_0 k - p + \epsilon_2 m - \epsilon_3 n)) \overline{g_2(m)} 
	g_3(n) \pv \frac{\widehat{\phi}(p)}{p} \, dm dn dp \ . 
\end{split}
\end{align*}
where $\e_j \in \{+,-\}$ and $g_j$ are smooth enough functions and
$$ \Phi(k,p,m,n) = -k^2 + (\epsilon_0 k - p + \epsilon_2 m - \epsilon_3 n)^2 - m^2 + n^2 \ . $$
The analysis leads to a pointwise growth of $\log \jt$. A more satisfactory pointwise asymptotic description remains to be established. For instance, in previous work such as \cite{GPR}, \cite{CPNLS2}, the authors got an asymptotic ODE for their respective models. It takes the form
\begin{align}\label{asODE}
i\partial_t \mathcal{F}f(t,k) =  - \frac{1}{2t}\big|\mathcal{F}f(t,k)\big|^2 \mathcal{F}f(t,k)
   + \mathcal{O}(\varepsilon^3 \jt^{-1-\rho})
\end{align}
where $\rho$ is a small positive constant and $\mathcal{F}$ denotes the respective (distorted) Fourier transforms in each setting. These precise asymptotics lead to modified scattering in the models studied in \cite{GPR} or \cite{CPNLS2}. In our case, this strategy fails. This is due to the $t^{1/4}$ growth of $\| \partial_kf^\sharp\|_{L^2}$ established by the bootstrap argument, rather than  $t^{\alpha}$ with $\alpha < 1/4$ as in the papers  \cite{GPR} or \cite{CPNLS2}. We discuss these issues at the end of the paper. \\
 \smallskip
\subsubsection*{Organization of the paper}
Section \ref{seclinscatt} presents basic results on linear scattering theory and introduces the 
distorted Fourier transform (dFT), the modified/sharp transform needed in the non-generic case,
and our decomposition of the (modified) generalized eigenfunctions. Section \ref{seclinest} provides the linear estimates that we will use for the nonlinear estimates. Section \ref{secmu} exposes the decomposition of the nonlinear spectral distribution (NSD) which encodes the nonlinear interactions. In Sections \ref{secnonlinl2sing},  \ref{secnonlinl2reg} and \ref{secnonlinlinf}, we prove respectively the $L^2$ and $L^\infty$ nonlinear estimates. Finally, in Section \ref{propgoalproof}, we give the proof of the main theorem \ref{mainthm} with details of the bootstrap scheme. We also discuss some issues about the choice of the bootstrap space. 

\smallskip
\subsection*{Notation}
We use the notation $\langle x\rangle := (1+|x|^{2})^{\frac{1}{2}}$.

\noindent
For positive quantities $a$ and $b$, we write $a\lesssim b$ for
	$a\leq Cb$ where $C$ is a universal constant,
	and $a\simeq b$ when $a\lesssim b$ and $b\lesssim a$. For $a,b \in \C$, we write $a = \mathcal{O}(b)$ if $|a| \lesssim |b|$. 
	
\noindent
We let $\mathbf{1}_A$ denote the characteristic function of the set $A$,
	and let $\mathbf{1}_+ := \mathbf{1}_{[0,\infty)}$, $\mathbf{1}_- := \mathbf{1}_{(-\infty,0)}$.

\noindent
The (flat) Fourier transform is defined as
	\begin{equation}\label{eq:FT}
	\hat{h}(\xi) = \whF [h](\xi)
	=\int_\R e^{-ix\xi}h(x)\,dx,
	\qquad \whF^{-1} [h](x)
	=\int_\R e^{ix\xi}h(\xi)\,\frac{d\xi}{2\pi}. 
	\end{equation}

\noindent
We use notation $\mathcal{S}(\R)$ to denote the space of Schwartz functions. \\
We use the standard notation $L^p$ for Lebesgue spaces, and $H^s$ for Sobolev spaces of order $s$. i.e. for a measurable function $f$
$$  \|f\|_{L^p} = \Big( \int | f(x)|^p \ dx \Big)^{1/p}  \ \ , \ \qquad \|f\|_{H^s} = \big\|\jk^s \widehat{f}\big\|_{L^2}.$$
We will sometimes specify the domain and the variable, e.g. $L^p_x(\R)$ or $L^2_s([0,t])$,
but often omit these when there is no risk of confusion.

Given $p,q \in[1,\infty]$, the mixed norms for a space-time function $F(t,x)$ are given by
$$	\big\| F \big\|_{L_{x}^{p}L_t^q}=	\big\| x \mapsto \big\| F(\cdot,x)\big\|_{L^q_t} \big\|_{L_{x}^{p}}\quad\text{and}\quad\big\| F \big\|_{L_{t}^{q}L_x^p}=	\big\| t \mapsto \big\| F(t, \cdot)\big\|_{L^p_x} \big\|_{L_{t}^{q}}.$$
For two differentiable functions $f , g : \R \to \C$, their Wronskian is defined as the function
$$ W(f,g) = f'g - fg' .$$

%---------------------------------------------
\bigskip
\section{Linear theory and modified distorted Fourier transforms}\label{seclinscatt}
\subsection{Linear scattering}
\subsubsection{Jost functions} \label{sec:jostf}  
% A more detailed exposition can be found in, for example, Deift-Trubowitz \cite{DT}, Germain-Pusateri-Rousset \cite{GPR}, Lindblad-L\"uhrmann-Schlag-Soffer \cite{LLSS} and the monograph \cite{Yaf} of Yafaev.

\begin{lemma}\label{existJost}
    For each $k \in \C$ such that $\im k \ge 0$, there exists functions \( \psi_+(x,k) \) and \( \psi_-(x,k) \) which are solutions to the eigenvalue problem
    \begin{align}\label{eigenprob}
		 ( -\partial_{xx} + V ) \psi_{\pm}(x,k) = k^2 \psi_{\pm}(x,k)
		\end{align}
    with the following asymptotic behavior:
    $$
    \lim \limits_{x \to +\infty}  \psi_+(x,k) - e^{ikx} \ = 0 \quad , \quad \lim \limits_{x \to -\infty}  \psi_-(x,k) - e^{-ikx}  = 0,
    $$
    $$
    \lim \limits_{x \to +\infty}  \partial_x\psi_+(x,k) - ike^{ikx}  = 0 \quad , \quad \lim \limits_{x \to -\infty}  \partial_x\psi_-(x,k) + ike^{-ikx} = 0.
    $$
    The functions $\psi_{\pm}$ are called Jost functions.
    
    Moreover, for any fixed $x \in \R$, $\psi_{\pm}$ is analytic in $k$ for $\im k > 0$, continuous up to $\im k \ge 0$ and they satisfy  
$$ \psi_\pm(x,-\bar{k}) = \overline{\psi_\pm(x,k)}.$$
\end{lemma}
\begin{proof} See (\cite{DT}, Lemma 1).
\end{proof}

% \begin{definition}
%     The \textit{Jost functions} are the functions \( \psi_+(x,k) \) and \( \psi_-(x,k) \) which are solutions to the eigenvalue problem
%     \begin{align}\label{eigenprob}
% 		H\psi_{\pm}(x,k) = ( -\partial_{xx} + V ) \psi_{\pm}(x,k) = k^2 \psi_{\pm}(x,k)
% 		\end{align}
%     with the following asymptotic behavior:
%     $$
%     \lim \limits_{x \to +\infty}  \psi_+(x,k) - e^{ikx} \ = 0 \quad ; \quad \lim \limits_{x \to -\infty}  \psi_-(x,k) - e^{-ikx}  = 0.
%     $$
%     $$
%     \lim \limits_{x \to +\infty}  \partial_x\psi_+(x,k) - ike^{ikx}  = 0 \quad ; \quad \lim \limits_{x \to -\infty}  \partial_x\psi_-(x,k) + ike^{-ikx} = 0.
%     $$
% \end{definition}

We normalize the Jost functions by defining the functions
$$
m_{\pm}(x,k) = e^{\mp ikx} \psi_{\pm}(x,k).
$$
They satisfy the following asymptotic behavior 
$$
\lim \limits_{x \to \pm \infty}  m_{\pm}(x,k) = 1\quad, \quad \quad    \lim \limits_{x \to \pm \infty} \partial_x m_{\pm}(x,k)  = 0,
$$
and solve the Volterra equations (see Lemma 1 in \cite{DT} ) 
% \begin{align}\label{neigprob}
%     \partial_{xx} m_{\pm} \pm 2ik \partial_x m_{\pm} = V m_{\pm}.
% \end{align}
\begin{align} \label{volterra}
 m_{\pm}(x,k) = 1 \pm \int_x^{\pm \infty} D_k(\pm(y-x)) V(y) m_{\pm}(y,k) \, dy \quad \text{with} \quad D_k(x) = \frac{e^{2ikx} - 1}{2ik}.
  \end{align}

%\begin{lemma}\label{Jostfunc}
    %Let $V$ be a real-valued function.
    %\begin{enumerate}[label = (\roman*)]
    %\item If $V \in L^1(\R)$, then for any $k \in \C$ such that $\im k \ge 0$ and $k \neq 0$, the Jost functions associated to $V$ exist. Moreover, these functions are analytic in $k$ on the upper half plane $\im k > 0$ and are continuous in $k$ up to the real axis, with a possible discontinuity at $k=0$.
    
    %\item If $V \in L^1(\R,\jx dx)$, then for any $x \in \R$,  the Jost functions $k \mapsto \psi_{\pm}(x, k)$  are continuous in $k$ up to the whole real line. Moreover, for $k \in \R$ the Jost functions satisfy the following
    %$$ \psi_\pm(x,-k) = \overline{\psi_\pm(x,k)}$$
   %\end{enumerate} 
%\end{lemma}
%\begin{proof} For (i), see \cite{Yaf}, Chapter 4, Lemma 1.4. For a proof of (ii), see \cite{GPR}, Appendix A.1. 
%\end{proof}
We define the following quantities:
$$
\mathcal{W}_+^s(x) = \int_x^{\infty} \langle y \rangle^s |V(y)| \, dy \quad ; \quad \mathcal{W}_-^s(x) = \int_{-\infty}^{x} \langle y \rangle^s |V(y)| \, dy.
$$
\begin{lemma}\label{estim m} For every $s\ge 0$, assuming that $V \in L^1(\R , \jx^{s+1} dx)$, we have the following uniform estimates in $x$ and $k$
\begin{align} \label{mestim}
 |\partial_k^s (m_\pm(x,k) - 1)| & \lesssim \frac{1}{\jk}\mathcal{W}_\pm^{s+1}(x) \quad,\quad \text{ for } \pm x \ge -1 \\
  |\partial_k^s (m_\pm(x,k) - 1)| & \lesssim \frac{1}{\jk}\jx^{s+1} \quad\quad, \quad \text{ for } \pm x \le 1. 
\end{align}
Also,
\begin{align}\label{Mestimates1'}
			\begin{split}
				& \left|\partial_{k} \left(m_{\pm}(x,k)-1\right)\right|\lesssim\frac{1}{|k|}\mathcal{W}_{\pm}^{1}(x)\quad, \quad  \pm x\geq-1.
				%\\
				%& \left|\partial_{k}^{s}\left(m_{\pm}(x,k)-1\right)\right|\lesssim\frac{1}{\jk}\jx^{s+1}, \qquad \pm x\leq1.
			\end{split}
		\end{align}
For the derivatives in $x$, we have 
\begin{align}
    |\partial_x\partial_k^s m_\pm(x,k)| & \lesssim \mathcal{W}_\pm^{s}(x) \quad\quad \text{ for } \pm x \ge -1 \\
  |\partial_x\partial_k^s m_\pm(x,k) | & \lesssim \jx^{s} \quad\quad \quad \text{ for } \pm x \le 1.
\end{align}
\end{lemma}
 \begin{proof} %These are consequences of the fact that the functions \( m_{\pm}(x,k) \) satisfy the following Volterra equation:
 %  \begin{align} \label{volterra}
 % m_{\pm}(x,k) = 1 \pm \int_x^{\pm \infty} D_k(\pm(y-x)) V(y) m_{\pm}(y,k) \, dy \quad \text{with} \quad D_k(x) = \frac{e^{2ikx} - 1}{2ik}.
 %  \end{align}
see (\cite{GPR}, Appendix A.1).
\end{proof}

\smallskip
\subsubsection{Transmission and reflection coefficients, scattering matrix}
\begin{lemma}\label{TRcoeff}
   For all $k \in \R\setminus \{0\}$, there exist $T(k)\in \C$ and $R_\pm(k) \in \C$ such that for all $x \in \R$, the following relations hold
    \begin{eqnarray}
        T(k) \psi_+(x,k) &  = & R_-(k)\psi_-(x,k) + \psi_-(x,-k) \ , 
        \label{transref+} \\
         T(k) \psi_-(x,k) &  = & R_+(k)\psi_+(x,k) + \psi_+(x,-k)\ .
         \label{transref-}
    \end{eqnarray}
    $T(k)$ is called the transmission coefficient and $R_\pm(k)$ are the reflection coefficients.
    
    Moreover, the following relations hold:
    \begin{enumerate} [label = (\roman*)]
    \item \begin{align} \label{wronsk}
        T(k) W(k) = 2ik 
    \end{align} 
    where $W(k) := W(\psi_+(\cdot , k) , \psi_-(\cdot,k)) $.
         \item  
         \begin{align}
          T(-k) = \overline{T(k)} \qquad  , \qquad R_{\pm}(-k) = \overline{R_{\pm}(k)}. \nonumber 
          \end{align}
          \item 
          \begin{align}
               |R_{\pm}(k)|^2 + |T(k)|^2 =1 \qquad  , \qquad  T(k)\overline{R_-(k)} + \overline{T(k)}R_+(k) = 0.\nonumber 
          \end{align}
          \item 
    \begin{eqnarray}
    \frac{1}{T(k)} & = & 1 - \frac{1}{2ik}\int_{\R} V(x)m_{\pm}(x,k) dx,  \label{formulaT} \\
    \frac{R_{\pm}(k)}{T(k)} & = & \frac{1}{2ik} \int_{\R} e^{\mp 2ikx} V(x)m_{\pm}(x,k) dx. \label{formulaR}
\end{eqnarray} 
     \end{enumerate}
     \
\end{lemma}
\begin{proof}
See the discussion in the Subsection $3$ of Section 2 in \cite{DT}.
    % For any real $k \neq 0$, the Wronskian of solutions of \eqref{eigenprob} does not depend on $x$ and the asymptotics behavior of the Jost functions give respectively
    % $$ W(\psi_+(\cdot , k) , \psi_+(\cdot,-k)) = -2ik \quad \quad ; \quad  \quad W(\psi_-(\cdot , k) , \psi_-(\cdot,-k)) = 2ik. $$
    % Therefore, $\{\psi_+(\cdot , k) , \psi_+(\cdot,-k)\}$ and $\{\psi_-(\cdot , k) , \psi_-(\cdot,-k)\}$ are couples of linear independent solutions of \eqref{eigenprob}. Thus the ODE theory says 
    % \begin{align}
    %     \label{sol1}
    %     \psi_+(x,k) & = A_+(k) \psi_-(x,k) + B_+(k)\psi_-(x,-k) \\
    %     \label{sol2}
    %      \psi_-(x,k) & = A_-(k) \psi_+(x,k) + B_-(k)\psi_+(x,-k)
    % \end{align}
    % and we can compute exactly $A_\pm$ and $B_\pm$ in terms of Wronskian
    % \begin{align}
    % A_+(k) = \frac{W_0(k)}{2ik} \quad \quad ; \quad \quad B_+(k) = -\frac{W(k)}{2ik} \\
    %  A_-(k) = \frac{\overline{W_0(k)}}{2ik} \quad \quad ; \quad \quad B_-(k) = -\frac{W(k)}{2ik}
    % \end{align} 
    % where $W_0(k) : =W(\psi_+(\cdot , k) , \psi_-(\cdot,-k))$
    % then plugging \eqref{sol2} in \eqref{sol1} gives the relation
    % \begin{align} \label{rel}
    %     |W(k)|^2 = 4k^2 +|W_0(k)|^2
    % \end{align}
    % Therefore, $W(k) \neq 0$ for real $k \neq 0$ and we get the lemma by setting 
    % $$ T(k) = \frac{1}{B_\pm(k)} =  - \frac{2ik}{W(k)} \quad  ; \quad R_-(k) = \frac{A_+(k)}{B_+(k)} = -\frac{W_0(k)}{W(k)}\quad  ;  \quad \frac{R_+(k)}{T(k)} = \frac{A_-(k)}{B_-(k)} = -\frac{\overline{W_0(k)}}{W(k)} $$
    \end{proof}

We now give a decay property of the derivatives of the transmission and reflection coefficients.
\begin{lemma}\label{estim TR} The following uniform estimate holds
\begin{align}
 \langle k \rangle (|\partial_k T(k)| + |\partial_k R_{\pm}(k)| ) \lesssim 1 \ .
 \end{align}
\end{lemma}
\begin{proof}  
The boundedness at low frequency follows from  (\cite{EHT}, Theorem 2.2). The high frequency estimates is provided in (\cite{We2}, Theorem 2.3).
% It suffices to look at the case $|k| \ge 1$ since the small value case is handled by lemma \ref{explicit TR}. First, we estimate $\partial_k T(k)$. By differentiating the formula (\ref{formulaT}), we get 
% \begin{align}
%     \partial_k T(k) & = -T(k)^2 \Big[\frac{1}{2ik}\int_\R V(y)\partial_k m_\pm(y,k) \ dy -\frac{1}{2ik^2}  \int_\R V(y)m_\pm(y,k) \ dy \Big] 
% \end{align}
% then for $|k| \ge 1$, we get the desired estimation by invoking the estimations in lemma \ref{mestim} and the fact that $|T(k)| \le 1$.\\
%
% Now for the estimation of $\partial_kR_\pm(k)$,
% we start by differentiating the formula (\ref{formulaR}) and we get 
% \begin{align}
%     \partial_k R_\pm(k)  = \Big[ \frac{\partial_kT(k)}{2ik} - \frac{T(k)}{2ik^2}\Big] \int_{\R} e^{\mp 2iky} V(y)m_{\pm}(y,k) \ dy + \frac{T(k)}{2ik}\int_{\R} V(y) \partial_k \Big(e^{\mp 2iky} m_{\pm}(y,k)\Big) \ dy 
%  \end{align}
% \begin{align}
%     \partial_k R_\pm(k)  = \Big[ \frac{\partial_kT(k)}{2ik} - \frac{T(k)}{2ik^2}\Big] \int_{\R} e^{\mp 2iky} V(y)m_{\pm}(y,k) \ dy &+ \frac{T(k)}{2ik}\int_{\R} \mp 2iy e^{\mp 2iky} V(y)m_{\pm}(y,k) \ dy \\
%         & + \frac{T(k)}{2ik}\int_{\R} e^{\mp 2iky} V(y)\partial_k m_{\pm}(y,k) \ dy
%  \end{align}
% then for $|k| \ge 1$, we have the right estimates, since the estimation of $\partial_k T(k)$, the fact that $|T(k)| \le 1$ and the estimations of lemma \ref{mestim}.
\end{proof}  

\smallskip
\subsubsection{Generic and non-generic potentials}

\begin{definition}\label{def:generic}
		$V$ is said to be a generic potential if
		\begin{equation}\label{eq:genericcond}
		\int_\R V(x)m_{\pm}\left(x,0\right)\,dx \neq 0,
	\end{equation}
		and it is non-generic otherwise.
\end{definition}

\begin{definition}\label{0reson}
     $V$ is said to have a zero-energy resonance if there exist a $L^\infty$ bounded function $\varphi$ solution of $-\partial_{xx}\varphi + V(x)\varphi =0$. In that case, the normalized solution $\varphi_+(x) = \psi_+(x,0)$ which has $1$ has limit at $+\infty$ is called the zero-energy resonance of the potential $V$.
\end{definition}

\begin{remark} \label{resonlim}
    Assume that there exist a $L^\infty$ bounded function $\varphi$ solution of $-\partial_{xx}\varphi + V(x)\varphi =0$. Then, $\varphi$ has finite limits at $\pm \infty$. This is a consequence of the decay of the potential $V$. 
    \end{remark}

\begin{lemma} \label{moreongen}We have equivalence between the following assertions :
\begin{enumerate} [label = (\roman*)]
    \item $V$ is generic.
    \item $T(0) = 0$ ; $R_\pm(0) = -1 .$
    \item $W(0) \neq 0$.
    \item The potential $V$ does not have a zero-energy resonance.
\end{enumerate}
\end{lemma}
\begin{proof}
   $(i) \iff  (ii)$. Indeed, from (\ref{formulaT}), we know that 
\begin{align}\label{formulaT(k)}
T(k) = \frac{2ik}{2ik - \int_\R V(x)m_\pm(x,k) \ dx }
\end{align}
and the equivalence follows from evaluation $k =0$. If $(i)$ is true, we get $T(0) = 0$. $R_\pm(0) = -1$ follows from the relations (\ref{transref+}) and (\ref{transref-}). And if $(ii)$ is true, then $\int_\R V(x)m_{\pm}\left(x,0\right)\,dx \neq 0$ necessarily.\\

$(ii) \iff (iii)$. By differentiating the wronskian relation (\ref{wronsk}), we get 
\begin{align}
    \partial_k T(k) W(k) + T(k) \partial_k W(k) = 2i 
\end{align}
If $(ii)$ is true, by evaluating at $k = 0$, we get $\partial_k T(0)W(0) = 2i$; in particular, $W(0) \neq 0$. And if $(iii)$ is true, we deduce $T(0) =0$ from the relation \eqref{wronsk} evaluating at $k=0$. $R_\pm(0) = -1$ follows from the relations (\ref{transref+}) and (\ref{transref-}).\\

$(iii) \iff (iv)$. Let us assume $(iii)$ is true. Suppose by contradiction that there exist a non-trivial $L^\infty$ bounded function $\varphi$ such that $- \partial_{xx}\varphi + V\varphi =0$. The function $\varphi$ has finite limits at both $\pm \infty$ (see Remark \ref{resonlim}).  Since $W(0) \neq 0$, then the functions $\psi_+(x,0)$ and $\psi_-(x,0)$ are linearly independent solutions of the ODE $-y''+ Vy=0$. So, they form a basis of the solutions. Thus, there exist $\alpha, \beta \in \C$ such that 
\begin{align}
    \varphi(x) = \alpha\psi_+(x,0) +\beta\psi_-(x,0)
\end{align}
Without loss of generality, assume $\beta \neq 0$. Since $\varphi$ is bounded and $\psi_+(x,0) \to 1$ as $x \to +\infty$, the function $\psi_-(x,0)$ must also have a finite limit as $x \to +\infty$. After normalization, $\psi_-(\cdot,0)$ satisfies the same zero-energy Jost condition at $+\infty$ as $\psi_+(\cdot,0)$. By uniqueness of the Volterra construction at zero energy, the two solutions must coincide, contradicting the linear independence of $\psi_+(\cdot,0)$ and $\psi_-(\cdot,0)$. Therefore $(iv)$ is true. 

Now, let us assume $(iv)$ is true and $(iii)$ is not true. Then $W(0) =0$. It implies that $\psi_+(x,0)$ and $\psi_-(x,0)$ are linearly dependent. Thus, $\psi_+$ has finite limit at both infinity which contradict $(iv)$.
% $(iv) \Rightarrow (i)$. If $V$ were not generic, then by \eqref{formulaT(k)} the denominator in \eqref{formulaT(k)} would not vanish at $k=0$, so $T(0)$ would be finite and nonzero. The Wronskian identity \eqref{wronsk} would then force $W(0)=0$, hence $\psi_+(\cdot,0)$ and $\psi_-(\cdot,0)$ would be linearly dependent and yield a non-trivial bounded solution of $H\varphi=0$, contradicting (iv). Therefore $V$ is generic
\end{proof}

\subsection{Some bounds on Pseudo-Differential Operators}
We state some bounds on Pseudo-Differential Operators (PDOs) 
that are specifically tailored to the Jost functions $m_\pm(x,k)$.

	\begin{lemma}\label{lem:m-1}
		Suppose $\jx^{\gamma}V\in L^1$.
		Then, for $\gamma>3/2$, we have
		\begin{align}\label{eq:pseeasy}
		\left\Vert \mathbf{1}_{\{\pm x\geq-1\}}\int_\R e^{ik x}m_\pm(x,k)g(k)\,dk
		\right\Vert_{L^2_x}\lesssim{\| g \|}_{L^2},
		\\
		\label{eq:pseeasy-1}
		\left\Vert \int_\R e^{ik x} m_\pm(x,k)
		\mathbf{1}_{\{\pm x\geq-1\}}h(x)\,dx\right\Vert_{L_{k}^{2}}
		\lesssim\left\Vert h\right\Vert_{L^2},
		\end{align}
		%Similar estimates hold for $m_{-}$ restricted on $x\leq1$.	
		and, for $\gamma>\beta+3/2$,
		\begin{align}\label{eq:mPDO+}
		\left\Vert \mathbf{1}_{\{\pm x\geq-1\}}\jx^{\beta}
		\int_\R e^{ik x}\left(m_\pm(x,k)-1\right)g(k)\,dk\right\Vert_{L^2_x}
		\lesssim \left\Vert g\right\Vert_{L^{2}},
		\\
		\label{eq:mPDO+-1}
		\left\Vert \int_\R e^{ik x}\left(m_\pm(x,k)-1\right)
		\mathbf{1}_{\{\pm x\geq-1\}}h(x)\,dx\right\Vert_{L_{k}^{2}}
		\lesssim\left\Vert \jx^{-\beta}h\right\Vert_{L^2_x}.
		\end{align}
		%Similar estimates hold for $m_{-}-1$ restricted on $x\leq1$.
		%
		%\label{lem:m_x}
		Moreover, with $\gamma>\max(\beta/2+3/4,\beta)$, we have
		\begin{align}\label{eq:psehard}
		& \left\Vert \jx^{\beta} \mathbf{1}_{\{\pm x\geq-1\}}
		\int_\R e^{ik x}\partial_xm_\pm(x,k)g(k)\,dk\right\Vert_{L^2_x}\lesssim{\| g \|}_{L^2},
		\\
		\label{eq:mPDOhard+2}
		& \left\Vert \int_\R e^{ik x}\partial_xm_\pm (x,k)
		\mathbf{1}_{\{\pm x\geq-1\}}h(x)\,dx\right\Vert_{L_{k}^{2}}
		\lesssim\left\Vert \jx^{-\beta}h\right\Vert_{L^2}.
		\end{align}
		%Similar estimates hold for $\partial_xm_{-}$ restricted on $x\leq1$.
		%
		%\label{lem:m_k}
		%Suppose $\jx^{\gamma}V\in L^1$ with $\gamma \geq 2$.
		Furthermore, if $\gamma > \beta + 5/2$, then
		\begin{align}
		\label{eq:pdomk}
		& \left\Vert \mathbf{1}_{\{\pm x\geq-1\}} \jx^\beta 
		\int_\R e^{ik x}\partial_{k}m_\pm(x,k)g(k)\,dk\right\Vert_{L^2_x}\lesssim{\| g \|}_{L^2},
		\\
		\label{eq:pdomkd}
		& \left\Vert \int_\R e^{ik x}\partial_{k}m_\pm(x,k)
		\mathbf{1}_{\{\pm x\geq-1\}}h(x)\,dx\right\Vert_{L^2_k}
		\lesssim\left\Vert \jx^{-\beta}h\right\Vert_{L^2_x}.
		\end{align}
		Finally if $\gamma > \max(\beta/2 + 3/4, \beta) + 1 $, then
		\begin{align}
		\label{eq:pdomxk}
		& \left\Vert \mathbf{1}_{\{\pm x\geq-1\}} \jx^\beta 
		\int_\R e^{ik x} \partial_x\partial_k m_\pm(x,k)g(k)\,dk\right\Vert_{L^2_x}\lesssim{\| g \|}_{L^2},
		\\
		\label{eq:pdomxkd}
		& \left\Vert \int_\R e^{ik x} \partial_x\partial_k m_\pm(x,k)
		\mathbf{1}_{\{\pm x\geq-1\}}h(x)\,dx\right\Vert_{L^2_k}
		\lesssim\left\Vert \jx^{-\beta}h\right\Vert_{L^2_x}.
		\end{align}

	\end{lemma}
\begin{proof}
See (\cite{NLSV}, Lemmas 2.4-2.9).
\end{proof}

\medskip
\subsection{Spectral theory: Resolvent and Spectral Projectors}
The Schrödinger operator $H = -\partial_{xx} + V$ is defined as the Friedrichs realization associated to the quadratic form on $H^1(\R)$
\begin{align}
    q(f) = \int_\R  |f'(x)|^2 \ dx + \int_\R V(x)|f(x)|^2 \ dx , 
    \end{align}
with domain
 \begin{align}
    D(H) = \big\{ f \in H^1(\R) \quad : \quad -\partial_{xx} f + Vf \in L^2(\R) \big\}
\end{align}
where every term is understood in the sense of distributions. Moreover, since $V$ is a real-valued $L^1$ function on $\R$, Weyl's theorem on the invariance of the essential spectrum yields
$$ \sigma(H) = \sigma_{ess}(H) = \sigma_{ess} (-\partial_{xx}) = [0,+\infty) $$
Under our standing assumption that $H$ has no eigenvalues, it follows that $\sigma(H) = [0,+\infty)$.
Those standard facts can be found in the spectral theory book of Lewin \cite{Lew}.

\begin{lemma}[Resolvent of $H$]
For all $z \in \C$ such that $\im z \neq 0$ and $W(\lambda) \neq 0$ where $\lambda= z^{1/2}$ is such that $\im \lambda > 0$, the resolvent $R(z)$ of the Schrödinger operator is an integral operator given by the kernel 
\label{resol}
\begin{align}R(z)(x,y) = \frac{1}{W(\lambda)}\, \psi_+(\max(x,y), \lambda)\, \psi_-(\min(x,y), \lambda).
\end{align}
The kernel $R(z)(x,y)$ is an analytic function of $z$ on $\C \setminus [0,+\infty)$ and continuous up to the cut along $[0, +\infty)$, with exception, possibly at $z = 0$. 
\end{lemma}
\begin{proof}See (\cite{Yaf} chapter 5, proposition 1.4).
% We need to prove that for all $f \in L^2(\R)$, we have:
% $$R(z)f \in H^1(\R) \quad \text{ and } \quad  (-\partial_{xx} + V)R(z)f = f $$
% First of all, for $f \in L^2$ and $z, \lambda$ as in the lemma, we have :
% \begin{align}
%     (R(z)f) (x) = \frac{1}{W(\lambda)}\Big\{\int_{-\infty}^x \psi_+(x,\lambda)\psi_-(y,\lambda)f(y) \ dy  + \int_x ^{+\infty} \psi_+(y,\lambda)\psi_-(x,\lambda)f(y) \ dy \Big \} \nonumber
% \end{align}
% Using the fact that Jost functions satisfy the  asymptotics (with $C$ a constant that depends only on $V$ and $\lambda$)
% $$ |\psi_+(x,\lambda)| \le C e^{-x\im\lambda } \qquad ; \qquad |\psi_-(x,\lambda)| \le C e^{x\im\lambda } $$
% we have 
% \begin{align}
%     |W(\lambda)|^2\|R(z)f\|^2 & \le C^2 \int_\R \int_{-\infty}^x e^{-2(x-y)\im \lambda} |f(y)|^2 dy + C^2 \int_\R \int_x^{+\infty} e^{-2(y-x)\im \lambda} |f(y)|^2 dy \nonumber  \\
%     & \le \frac{C^2}{\im \lambda} \int_{\R} |f(y)|^2 \ dy \nonumber
% \end{align}
% therefore $R(z)f \in L^2(\R)$. For the derivative, we get by direct computation (in the sense of distribution) 
% \begin{align}
%     W(\lambda)\partial_x (R(z)f)(x) = \partial_x \psi_+(x,\lambda)\int_{-\infty}^x \psi_-(y,\lambda)f(y) \ dy + \partial_x \psi_-(x,\lambda)\int_x^{+\infty} \psi_-(y,\lambda)f(y) \ dy \nonumber
% \end{align}
% We get the $L^2$ boundedness by using the asymptotics (through lemma \ref{mestim})
% $$ |\partial_x\psi_\pm(x,\lambda)| \le C \langle \lambda \rangle e^{\mp x\im \lambda} $$
% and $(-\partial_{xx} + V)R(z)f = f$ follows straightforwardly by computations.
\end{proof} 
\smallskip
\begin{lemma}[Spectral Projectors of $H$]
Let $\lambda > 0$. The spectral projector $\mathbf{1}_{[0, \lambda]}(H)$ is defined as an integral operator with kernel given by:
\begin{align}\label{projspec}
(x, y) \mapsto \frac{1}{4\pi} \int_{0}^{\lambda} \frac{|T(\sqrt{s})|^2}{\sqrt{s}} \left[ \psi_+(x,\sqrt{s})\, \psi_+(y, -\sqrt{s}) + \psi_-(x,\sqrt{s})\, \psi_-(y, -\sqrt{s}) \right] ds.
\end{align}

\end{lemma}
\begin{proof} See (\cite{Yaf} chapter 5, proposition 1.5).
\end{proof}

\subsection{(Standard) distorted Fourier Transform}\label{ssecDFT0}
 Given the Jost functions $\psi_{\pm}$, we define the following kernel:
\begin{equation}
    \K(x,k) = \begin{cases}
    T(k)\psi_+(x,k)  \quad \quad k \ge 0 ,\\
    T(-k)\psi_-(x,-k) \quad k < 0.
\end{cases}
\label{kernel}
\end{equation} 
\begin{definition}\label{dft}
    For any function $f$ in the Schwartz class $\mathcal{S}(\R)$, we define the distorted Fourier transform (dFT) by
    $$ \widetilde{\Fou}{(f)}(k) = \int_{\R} \overline{\K(x,k)}f(x) dx .$$
\end{definition}
\begin{proposition} 
\label{plancherel}\
    \begin{enumerate} [label = (\roman*)]
        \item  The dFT extends to an unitary isomorphism from $L^2(\R,dx)$ to $L^2(\R,\frac{dk}{2\pi})$, that is:
        $$ ||\widetilde{\Fou}{(f)}||_{L^2} = ||f||_{L^2} \quad, \quad \forall f \in L^2(\R) .$$
        \item If $\phi \in L^1$, the inverse of $\widetilde{\Fou}$ is given by the formula: 
        $$ \widetilde{\Fou}^{-1}{(\phi)} (x) = \int_{\R} \K(x,k)\phi(k) \frac{dk}{2\pi}. $$
        \item If $f \in L^1$, then the distorted Fourier transform $\widetilde{\mathcal F}(f)$ is bounded. In fact, the following holds:
$$ \|\widetilde{\mathcal F}(f)\|_{L^{\infty}} \lesssim \|f\|_{L^1}. $$
Moreover, $\widetilde{\mathcal F}(f)$ is continuous everywhere except possibly at $0$, where a discontinuity may occur.
    \end{enumerate}
\end{proposition}
\begin{proof} See (\cite{KGV}, Proposition 3.6).

\end{proof}

The eventual discontinuity at 0 in the standard dFT can be analyzed based on the genericity of the potential $V$. In fact, when the potential is generic, the standard dFT is continuous since it vanishes at \(0\). This follows from the following lemma and the dominated convergence theorem.

\begin{lemma}
For a generic potential \(V\), we have the following Taylor expansion:
$$ T(k) = \alpha k + \mathcal{O}(k^2) \quad \text{ as } k \to 0 \quad , \alpha \in i\R$$
\end{lemma}

\begin{proof}
See \cite{KGV}, proposition 3.4
\end{proof} 

For the non-generic case, we collect the impact of the zero-energy resonance on the problem at $k=0$. First of all, one should note that the zero-energy resonance is a real-valued function since its conjugate satisfies the same eigenvalue equation with the same boundary value conditions at $+\infty$. 
\begin{lemma}\label{lemngTR}
		If $V$ is a non-generic potential, we set
		\begin{align}\label{a}
		a := \psi_+(-\infty,0) \in \R \setminus \{ 0 \}. 
		\end{align}
		Then, 
		\begin{align}\label{LESTR0}
		T(0) = \frac{2a}{1+a^2}, \qquad R_+(0) = \frac{1 - a^2}{1+ a^2}, \qquad \mbox{and} \qquad R_-(0) = \frac{a^2-1}{1+ a^2}.
		\end{align}
	\end{lemma}
\begin{proof}
	See \cite{KGV}, proposition 3.4 
\end{proof}
This lemma, along with the fact that \(\psi_+(x,0) = a \psi_-(x,0)\), allows us to deduce the following relation for the kernel of the dFT:
$$ \K(x,0+) = \frac{2a}{1+a^2}\psi_+(x,0) \quad \text{ and } \quad \K(x,0-) = \frac{1}{a}\K(x,0+) $$
We can therefore state the following proposition:

\begin{proposition}[Discontinuity of the standard dFT] \label{dftdiscont}
Assume $V$ is non-generic. If \(f \in L^1\), then:
$$ \widetilde{f}(0+) = \frac{2a}{1+a^2}\int_\R f(x)\psi_+(x,0)\,dx \quad \text{ and } \quad \widetilde{f}(0-) = \frac{1}{a}\widetilde{f}(0+)$$
In particular, \(\widetilde{f}\) is continuous if and only if \(a = 1\) or \(\widetilde{f}(0) = 0\).
\label{continu}
\end{proposition}

\smallskip
	\subsection{Modified kernel}\label{subsec:JCP}
   Given the Jost functions $\psi_{\pm}$ from \eqref{eigenprob}, assume that $V$ is non-generic. Set $T := T(0)$ and $R := R_-(0)$, so that $R_+(0) = -R$, and define
	\begin{align}\label{matK}
	\mathcal{K}^{\#}(x,k):=
	\begin{cases}
	T(k) \hspace{0.3cm} [ A_+\psi_+(x,k) + A_-\psi_-(x,k)] & k \geq 0
	\\
	T(-k) [ B_+\psi_+(x,-k) + B_-\psi_-(x,-k)] & k < 0,
	\end{cases}
	\end{align}
    where 
    \begin{eqnarray}
        A_+ & = & \frac{1}{2}\Big(\sqrt{1 + R} + \sqrt{1 - R} \Big) \nonumber \\
         B_+ & = & \frac{1}{2}\Big(\sqrt{1 + R} - \sqrt{1 - R} \Big) \nonumber \\
         A_- & = & \frac{T}{2}\Big(\frac{1}{\sqrt{1 + R}} - \frac{1}{\sqrt{1 - R}}  \Big) \nonumber \\
          B_- & = & \frac{T}{2}\Big(\frac{1}{\sqrt{1 + R}} + \frac{1}{\sqrt{1 - R}}  \Big) \nonumber
    \end{eqnarray}
 We note the following algebraic relations  
 \begin{equation}\label{eq:ABrela1}
     A_+A_-+B_+B_-=0,\,\,A_+B_++A_-B_-=0
 \end{equation}
and
\begin{equation}\label{eq:ABrela2}
    A_+^2+B_+^2=A_-^2+B_-^2=A_+^2+A_-^2=B_+^2+B_-^2=1.
\end{equation}
From \eqref{eq:ABrela2}, one notes that
\begin{equation}\label{eq:ABrela3}
 A^2_+=B_-^2,\,\,A_-^2=B_+^2.   
\end{equation}
  
    One can check that the vectors $(A_+ , B_+)$ and $(A_- , B_-)$ are orthonormal. This information will lead to  Plancherel theorem for the forthcoming modified transform. Furthermore, those coefficients were chosen so that the modified kernel satisfies continuity at $k =0$.

We can rewrite \eqref{matK} as follows
\begin{align} \label{matKunitary}
     \begin{bmatrix}
\mathcal{K}^\#(x,k) \mathbf{1}_{k\geq 0}  \\
\mathcal{K}^\#(x,k) \mathbf{1}_{k\leq 0}\\
   \mathcal{K}^\#(x,-k) \mathbf{1}_{k\leq 0}\\
    \mathcal{K}^\#(x,-k) \mathbf{1}_{k\geq 0}\\
	\end{bmatrix}=\begin{bmatrix}
	A_+ & 0 & 0 & A_- \\
	0& B_-& B_+ & 0\\
    0& A_-& A_+ & 0\\
    B_+& 0& 0 & B_-\\
	\end{bmatrix} \begin{bmatrix}
	\mathcal{K}(x,k) \mathbf{1}_{k\geq 0}  \\
	\mathcal{K}(x,k) \mathbf{1}_{k\leq 0}\\
   \mathcal{K}(x,-k) \mathbf{1}_{k\leq 0}\\
   \mathcal{K}(x,-k) \mathbf{1}_{k\geq 0}\\
	\end{bmatrix}.
\end{align}
So $\mathcal{K}^\#(x,k)$ can be thought of as a unitary transformation of $\mathcal{K}(x,k)$.

\begin{remark}[Special cases] Certain special values of the coefficients recover previously known distorted transforms. Consider the case $R=0$, which is equivalent to $a^2=1$.
\begin{itemize}
    \item In the case $a = 1$, which is equivalent to $T = 1$, we get the coefficients $(A_+ , B_+ , A_- , B_-) = (1,0,0,1)$. That corresponds to the standard dFT kernel (\ref{kernel}) we defined above. 
    \item In the case $a = -1$, which is equivalent to $T = -1$, we get the coefficients $(A_+ , B_+ , A_- , B_-) = (1,0,0,-1)$. That corresponds to the Chen-Pusateri dFT kernel introduced in \cite{CPNLS2} to solve the modified scattering problem in the case of an odd zero-energy resonance (see remark \ref{remarkdecompmu}). 
\end{itemize}
    
\end{remark}
	
\smallskip
	\subsection{Decomposition of the modified kernel}\label{subsec:DecK}
We define the following coefficients
        \begin{eqnarray} \label{coefformula+}
        \gamma_+(k) := A_+T(k) + A_-R_+(k) \quad \quad & , & \quad \quad \lambda_+(k) := A_+R_-(k) + A_-T(k) \ ,  \\
        \label{coefformula-}
        \gamma_-(k) := B_-T(-k) + B_+R_-(-k) \quad \quad & , & \quad \quad \lambda_-(k) := B_+T(-k) + B_-R_+(-k). 
    \end{eqnarray}
    One can check that 
    \begin{align} \label{coeff0cont}
    \gamma_+(0) = B_- \quad , \quad \gamma_-(0) = A_+ \quad , \quad \lambda_+(0) = B_+ \quad , \quad \lambda_-(0) = A_- .
     \end{align} 
    In fact, the coefficients $A_\pm$ and $B_\pm$ were determined to satisfy \eqref{coeff0cont}, that ensures continuity at $k =0$.
	
	Consider a smooth positive function $\phi$ such that $\operatorname{supp} \phi \subset [-1,1]$ and $\int_{\R} \phi \, dx = 1$.  
We define the cut-off functions $\chi_+$ and $\chi_-$ by:
\begin{align} \label{defchi+-}
\chi_+(x) = \int_{-\infty}^x \phi(y)\,dy \quad \quad \text{and} \quad \quad  \chi_+(x) + \chi_-(x) = 1 
\end{align} 
	
	% Let us then analyze the structure of $\mathcal{K}^\#$ more closely.
 %    First of all, using the cut-off functions $\chi_+$ and $\chi_-$ as above we can have a first decomposition:
 %    $$ \mathcal{K}^{\#}(x,k) = \Gamma (x,k)e^{ikx} + \Lambda(x,k)e^{-ikx} $$
 %    where 
    
 %    \begin{align}\label{matKGam}
	% \Gamma(x,k) = 
	% \begin{cases}
	% \chi_+(x) \gamma_+(k)m_+(x,k) + \chi_-(x) A_+m_-(x,-k) & k \geq 0
	% \\
	% \chi_-(x) \gamma_-(k) m_-(x,-k) + \chi_+(x) B_-m_+(x,k) & k < 0,
	% \end{cases}
	% \end{align}

 %    \begin{align}\label{matKLam}
	% \Lambda(x,k) = 
	% \begin{cases}
	% \chi_-(x) \lambda_+(k)m_-(x,k) + \chi_+(x) A_-m_+(x,-k) & k \geq 0
	% \\
	% \chi_+(x) \lambda_-(k)m_+(x,-k) + \chi_-(x) B_+m_-(x,k) & k < 0,
	% \end{cases}
	% \end{align} 
 
    %The above decomposition allows to have some information on the linear Schrödinger flow but even if it has some aspects of regularity, we need to refine the decomposition to have the right clues for the nonlinear analysis.\\\\

	For notational convenience we let $\chi_0(x) \equiv 1$.
	The kernel $\mathcal{K}^\#$ can be decomposed 
	 as the sum of a ``singular'' and ``regular'' part
	\begin{align}\label{Ksharpdecomp}
	\begin{split}
	& \mathcal{K}^{\#}(x,k) =  \mathcal{K}^{\#}_S(x,k) + \mathcal{K}^{\#}_R(x,k),
	\end{split}
	\end{align}
	with the following definitions:
	
	\setlength{\leftmargini}{1.5em}
	\begin{itemize}
		
		\item The singular part is given by
	\begin{align}\label{KsharpSdecomp}
	\begin{split}
	&  \mathcal{K}^{\#}_S(x,k) := \chi_0(x)\K^{\#}_0(x,k) + \chi_+(x)\K^{\#}_+(x,k) + \chi_-(x)\K^{\#}_-(x,k)
	\end{split}
	\end{align}
    where 
    \begin{align}\label{KsharpS0}
	\begin{split}
	 \K^{\#}_0(x,k)  & := (\chi_+(x)aA_- + \chi_-(x)B_+)m_-(x,0)e^{-ikx}  \\
      & \qquad + (\chi_-(x)\frac{1}{a}A_+ + \chi_+(x)B_-)m_+(x,0)e^{ikx}
	\end{split}
	\end{align}
    and
     \begin{align}\label{KsharpSpm}
	\begin{split}
	  \K^{\#}_+(x,k)  &:= (\gamma_+(k) - \gamma_+(0)) \mathbf{1}_+(k)e^{ikx}  + (\lambda_-(k) - \lambda_-(0) )\mathbf{1}_-(k)e^{-ikx} \\
      & =: a_+^+(k)e^{ikx} + a_+^-(k)e^{-ikx} \ ,\\
       \K^{\#}_-(x,k) &:= (\gamma_-(k) - \gamma_-(0)) \mathbf{1}_-(k)e^{ikx}  + (\lambda_+(k) - \lambda_+(0) )\mathbf{1}_+(k)e^{-ikx} \\
       & =: a_-^+(k)e^{ikx} + a_-^-(k)e^{-ikx} \ .
 	\end{split}
	\end{align}
		\smallskip
		\item the regular part is given as:

        For $k\geq 0$, 
        \begin{align}\label{KRk>0}
	 \mathcal{K}^\#_R(x,k) & = \Big\{\chi_+(x)\Big[ \gamma_+(k)(m_+(x,k) - m_+(x,0))  +(\gamma_+(k) - \gamma_+(0))(m_+(x,0) - 1) \Big]
	\nonumber\\
	& +\chi_-(x) A_+(m_-(x,-k) - m_-(x,0)) \Big\}e^{ikx}  \nonumber\\
       & + \Big\{\chi_-(x)\Big[ \lambda_+(k)(m_-(x,k) - m_-(x,0)) + (\lambda_+(k) - \lambda_+(0))(m_-(x,0) - 1) \Big]
	\nonumber\\
	& +\chi_+(x) A_-(m_+(x,-k) - m_+(x,0)) \Big\}e^{-ikx} \ . 
	\end{align}
	For $k<0$,
	\begin{align}\label{KRk<0}
	 \qquad \mathcal{K}^\#_R(x,k) & = \Big\{\chi_-(x)\Big[ \gamma_-(k)(m_-(x,-k) - m_-(x,0)) + (\gamma_-(k) - \gamma_-(0))(m_-(x,0) - 1) \Big]  \nonumber \\
	& +\chi_+(x) B_-(m_+(x,k) - m_+(x,0)) \Big\}e^{ikx}  \nonumber\\
       & + \Big\{\chi_+(x)\Big[ \lambda_-(k)(m_+(x,-k) - m_+(x,0)) + (\lambda_-(k) - \lambda_-(0))(m_+(x,0) - 1) \Big]
	 \nonumber\\
	& +\chi_-(x) B_+(m_-(x,k) - m_-(x,0)) \Big\}e^{-ikx} \ . 
	\end{align}
	For convenience in the forthcoming computations we write 
    \begin{align}\begin{split}
    \label{kR}
    \mathcal{K}_R^\sharp(x,k) = :\begin{cases}
    \mathcal{G}_+^+(x,k)e^{ikx} + \mathcal{G}_+^-(x,k)e^{-ikx} \quad \quad \text{ for } k\ge 0 \\
    \mathcal{G}_-^+(x,k)e^{ikx} + \mathcal{G}_-^-(x,k)e^{-ikx} \quad \quad \text{ for } k<0
    \end{cases}
     \end{split}
    \end{align}
    
	\end{itemize}

	%Also notice that $\mathcal{K}_{0}^{\#}(x,k)$ is the component which could
	%correspond to the discontinuous part of $\mathcal{K}$.

Let us record here the fact that $m_+(x,0)$ converges rapidly to $ 1$ and $a$ as $x \rightarrow \infty$ and $x\rightarrow-\infty$ respectively.  On the other hand, $m_-(x,0)$ converges rapidly to $ 1$ and $a^{-1}$ as $x \rightarrow - \infty$ and $x\rightarrow \infty$ respectively. In light of this, we decompose $\K^{\#}_0(x,k)$ from \eqref{KsharpS0} more closely as the following    
\begin{align}
	\begin{split}
    \K^{\#}_0(x,k) &= \chi_+(x)A_- (a m_-(x,0)-1)e^{-ikx} + \chi_-(x)B_+ (m_-(x,0)-1)e^{-ikx} \\
    &  + \chi_-(x)A_+ (a^{-1} m_+(x,0)-1)e^{ikx}  + \chi_+(x)B_-(m_+(x,0)-1)e^{ikx}\\
    & + \chi_+(x)A_- e^{-ikx} + \chi_-(x)B_+ e^{-ikx} +  \chi_-(x)A_+ e^{ikx}  + \chi_+(x)B_-e^{ikx}
    \label{K0closedecomp}
	\end{split}
	\end{align}

	\smallskip
	\begin{lemma} \label{m0asymp}
    For $\alpha =0,1$, the following estimates hold :
    \begin{align} \label{estimR+}
        \chi_{\pm}(x) \big| \partial_x^\alpha \big( m_\pm(x,0) - 1 \big) \big| 
		& \lesssim \jx^{-\gamma+1} \ , \\
        \label{estimR-}
        \chi_{\pm}(x) \big| \partial_x^\alpha \big( a^{\pm 1}m_\mp(x,0) - 1 \big) \big| 
		& \lesssim \jx^{-\gamma+1} \ .
    \end{align}
		% For $\epsilon \in \{+,-\}$, $r \in \N $ and for all $s \geq 0$, the following estimates hold
		% \begin{align}\label{m+0}
		% & \chi_{\pm}(x) \big| \partial_x^\alpha \big( m_\epsilon(x,0)^r - m_\epsilon(\pm \infty , 0)^r \big) \big| 
		% \lesssim \jx^{-s} \mathcal{W}_{\pm}^{s+1}(x), \qquad \alpha = 0,1,
		% \end{align}
		% In particular, under our assumption \ref{VassumeWei}
		% we have
		% \begin{align}\label{m+0'}
		% \chi_{\pm}(x) \big| \partial_x^\alpha \big( m_\epsilon(x,0)^r - m_\epsilon(\pm \infty , 0)^r \big) \big| 
		% \lesssim \jx^{-\gamma+1}, 
  %       % \qquad r=1,2,3,4, 
  %       \qquad \alpha = 0,1.
		% \end{align}
	\end{lemma}
	
	\begin{proof}
     Lemma \ref{estim m} implies that 
     \begin{align}
         \chi_{\pm}(x) \big| \partial_x^\alpha \big( m_\pm(x,0) - 1 \big) \big| 
		& \lesssim \chi_{\pm}(x)\mathcal{W}_{\pm}^{1}(x) \ . \nonumber
     \end{align}
     \eqref{estimR+} follows from the assumption \ref{VassumeWei} and the observation  $\chi_{\pm}(x)\mathcal{W}_{\pm}^{1}(x) \le \jx^{-\gamma+1}\mathcal{W}_{\pm}^{\gamma}(x)$.
      \eqref{estimR-} follows from \eqref{estimR+} and
     the fact that $m_{+}(x,0) = am_{-}(x,0)$.
%      the following Volterra's equation 
% \begin{align} \label{volterra0neg}
%     m_{+}(x,0) = a + \int_{- \infty}^x (x-y) V(y) m_{+}(y,0) \, dy \ .
% \end{align}
% In fact, since $m_{+}(x,0) = am_{-}(x,0)$. It suffices to argue on $m_{+}(x,0)$. \eqref{volterra0neg} implies for $x \le -1$, the inequality
% \begin{align}
%      |m_{+}(x,0) - a| \lesssim |a|\mathcal{W}_{-}^{1}(x) +\int_{- \infty}^x \jy |V(y)| |m_{+}(y,0) - a| \, dy \ .
% \end{align}
% By Gronwall's inequlity, we get 
% \begin{align}
%     |m_{+}(x,0) - a| \lesssim \mathcal{W}_{-}^{1}(x)
% \end{align}
% and the estimate \eqref{estimR-} is obtained through in the lemma \ref{mestim}.
	\end{proof}
    
	\begin{lemma}\label{lem:Ksing}
		The function $k \mapsto \mathcal{K}_{\pm}^{\#}(x,k)$ is Lipschitz continuous 
			with $\mathcal{K}_{\pm}^{\#}(x,0) = 0$.
			More precisely, $\mathcal{K}_{\pm}^{\#}(x,k) = a_\pm^+(k) e^{ikx} + a_\pm^-(k) e^{-ikx}$
			with bounded Lipschitz coefficients $a_{\epsilon_1}^{\epsilon_2}$, $\epsilon_1,\epsilon_2 \in \{+,-\}$, 
			with $a_{\epsilon_1}^{\epsilon_2}(0)=0$ and bounded derivatives. %in the set $A \smallsetminus \{1,-1\}$. 
\end{lemma}

\begin{proof} To establish $(i)$, it suffices to treat of $a_+^+$, since the analysis of the other coefficients is similar. We have $a_+^+(k) = (\gamma_+(k) - \gamma_+(0))\mathbf{1}_+(k)$ which is continuous and bounded due to the boundedness of $T(k)$ and $R_\pm(k)$. We have in the sense of distributions  
\begin{align*}
    \partial_ka_+^+(k) = \partial_k\gamma_+(k)\mathbf{1}_+(k) + (\gamma_+(k) - \gamma_+(0))\delta_0(k) = \partial_k\gamma_+(k)\mathbf{1}_+(k).
\end{align*}
The $L^\infty$ boundedness of $\partial_ka_+^+(k)$
 then follows from the boundedness of $T(k)$ and $R_\pm(k)$ and the bounds in Lemma \ref{estim TR}. Thus, $a_+^+(k)$ is Lipschitz continuous.
\end{proof}

\begin{lemma}
    \label{estimG+-}
    For all $\e_1, \e_2 \in \{+,-\}$, the following estimates hold 
    \begin{align} \label{decayG+-1}
        |\mathcal{G}_{\e_1}^{\e_2}(x,k)| &\lesssim \jx^{-\gamma +1} \ ,\\
         \label{decayG+-2}
         |\partial_x\mathcal{G}_{\e_1}^{\e_2}(x,k)| &\lesssim \jx^{-\gamma } \ , \\
           \label{decayG+-3}
           |k\partial_k\mathcal{G}_{\e_1}^{\e_2}(x,k)| &\lesssim \jx^{-\gamma +1} \ ,
    \end{align}
    uniformly in $x$ and $k$.
\end{lemma}
\begin{proof}
    We focus on the case $\e_1=\e_2=+$, since the other cases are similar. We have 
    \begin{align}
        \mathcal{G}_+^+(x,k) &= \Big\{ \chi_+(x)\Big[ \gamma_+(k)(m_+(x,k) - m_+(x,0))  +(\gamma_+(k) - \gamma_+(0))(m_+(x,0) - 1) \Big] \\
	 & +\chi_-(x) A_+(m_-(x,-k) - m_-(x,0))\Big\} \mathbf{1}_+(k) \ .
    \end{align}
The function $k \mapsto \gamma_+(k)$ is bounded due to boundedness of $T(k)$ and $R_\pm(k)$. In addition, Lemma \ref{estim m} give the estimate
$$ \chi_\pm(x)|m_\pm(x,k) - m_\pm(x,0)| + \chi_\pm(x)|m_\pm(x,0) - 1| \lesssim \chi_{\pm}(x)\mathcal{W}_{\pm}^{1}(x) \le \jx^{-\gamma+1}\mathcal{W}_{\pm}^{\gamma}(x)$$.
Then, we get \eqref{decayG+-1} since $\jx^\gamma V \in L^1$. 

Next, we prove \eqref{decayG+-2}. We can write 
\begin{align}
    \partial_x\mathcal{G}_+^+(x,k) &= \Big\{ \partial_x\chi_+(x)\Big[ \gamma_+(k)(m_+(x,k) - m_+(x,0))  +(\gamma_+(k) - \gamma_+(0))(m_+(x,0) - 1) \Big] \\
	 & +\partial_x\chi_-(x) A_+(m_-(x,-k) - m_-(x,0))\Big\} \mathbf{1}_+(k) \ \\
     &+ \Big\{ \chi_+(x)\Big[ \gamma_+(k)\partial_x(m_+(x,k) - m_+(x,0))  +(\gamma_+(k) - \gamma_+(0))\partial_x(m_+(x,0) - 1) \Big] \\
	 & +\chi_-(x) A_+\partial_x(m_-(x,-k) - m_-(x,0))\Big\} \mathbf{1}_+(k) \ .
\end{align}
Lemma \eqref{estim m} gives $\chi_\pm(x)|\partial_xm_\pm(x,k)| \lesssim \chi_{\pm}(x)\mathcal{W}_{\pm}^{0}(x) \lesssim \jx^{-\gamma}\mathcal{W}_{\pm}^{\gamma}(x)\lesssim \jx^{-\gamma}$. That yields the desired inequality.

 For \eqref{decayG+-3}, since $\mathcal{G}_+^+(x,0) =0$, we can write  
\begin{align}
    \partial_k\mathcal{G}_+^+(x,k) &= \Big\{ \chi_+(x)\Big[ \partial_k\gamma_+(k)(m_+(x,k) - 1)  +\gamma_+(k) \partial_k m_+(x,k)  \Big] \\
	 & -\chi_-(x) A_+\partial_k m_-(x,-k) \Big\} \mathbf{1}_+(k) \ .
\end{align}
Lemma \ref{estim TR} gives $|k\gamma_+(k)| \lesssim 1 $. Lemma \ref{estim m} gives $\chi_\pm(x)|k\partial_k m_\pm(x,k)| \lesssim \chi_{\pm}(x)\mathcal{W}_{\pm}^{1}(x) \le \jx^{-\gamma+1}$. This yields the desired estimate.
\end{proof}

\begin{lemma}
    \label{estimkR}
    The following estimate holds
    \begin{align}
        \label{decaykR1}
    |\K_R^\sharp(x,k)| &\lesssim \jx^{-\gamma +1} \ ,
    \end{align}  
    uniformly in $x$ and $k$. Moreover, the function $k \mapsto \mathcal{K}_R^{\#}(x,k)$  is Lipschitz continuous with $\mathcal{K}_{R}^\#(x,0)=0$. 
\end{lemma}
\begin{proof}
The estimate \eqref{decaykR1} is a direct consequence of \eqref{decayG+-1}. To establish The Lipschitz continuity, it suffices to treat \begin{align*}
\mathcal{G}_+^+(x,k) &= \Big\{ \chi_+(x)\Big[ \gamma_+(k)(m_+(x,k) - m_+(x,0))  +(\gamma_+(k) - \gamma_+(0))(m_+(x,0) - 1) \Big] \\
	 & +\chi_-(x) A_+(m_-(x,-k) - m_-(x,0))\Big\} \mathbf{1}_+(k) \ 
\end{align*}
which is continuous in $k$. We get that $\partial_k\mathcal{G}_+^+(x,k)$ is  bounded by invoking the estimations in Lemma \ref{estim m}. We have $\mathcal{K}_{R}^\#(x,0)=0$, because $\mathcal{G}_{\e_1}^{\e_2}(x,0) =0$ for any $\e_1, \e_2 \in \{+,-\}$. 
\end{proof}

\medskip
	\subsection{Modified distorted Fourier Transform}
	In analogy with \eqref{dft}, for $f \in \mathcal{S}(\R)$, we define the following modified distorted Fourier transform:
	\begin{align}\label{sharpF}
	\mathcal{F}^\# [f](k) = f^\#(k) = \int_\R \overline{\mathcal{K}^\#(x,k)} f(x)\,dx.
	\end{align}
	Here is the analogue of Proposition \ref{plancherel}.

	\begin{proposition}\label{propsharpF}\ 
        \begin{enumerate} [label = (\roman*)]
        \item  The modified dFT \eqref{sharpF} extends to an unitary isomorphism from $L^2(\R,dx)$ to $L^2(\R,\frac{dk}{2\pi})$, that is:
        \begin{align}\label{FsharpL2}
		{\big\| f^\# \big\|}_{L^2} = {\| f \|}_{L^2}, \quad \forall f\in L^{2} .
		\end{align}
        \item If $\phi \in L^1$, the inverse of $\mathcal{F}^\#$ is given by the formula:
        \begin{align}\label{Fsharp-1}
		\big(\mathcal{F}^\#\big)^{-1}[\phi](x) = \int \mathcal{K}^\#(x,k) \phi(k)\,\frac{dk}{2\pi}.
		\end{align}
        \item The following diagonalization holds
        \begin{align}
            \label{diagosharp}
        \left(-\partial_{xx}+V\right) 
		= \big(\mathcal{F}^\#\big)^{-1} k^{2} \mathcal{F}^\#.
        \end{align}
        More generally, if we denote $D:=\sqrt{-\partial_{xx}+V}$, then for any measurable function $m$, the following functional calculus hold 
		%, if $P_c$ denotes the projection onto the continuous spectrum, for any 
		\begin{align}\label{Fsharpm}
		m(D) = \big(\mathcal{F}^\#\big)^{-1} {m(|k|)} \mathcal{F}^\# \ .
		\end{align}
		
			\item If $ f \in L^1$, then $f^\#$ is a continuous bounded function
			that decays to zero at infinity. Moreover, 
            \begin{align} \label{fsharp l1}
                \|f^\sharp\|_{L^{\infty}} \lesssim \|f\|_{L^1}.
            \end{align}
			
			\smallskip
			\item There exists $C>0$, depending on ${\| V \|}_{L^1}$,
			such that the following bounds hold
			\begin{equation}\label{kusharp}
			\frac{1}{C} {\| u \|}_{H^1} \leq {\big\| \jk u^\# \big\|}_{L^2} \leq C %\big( 1 + \Vert V \Vert_{L^1}^{\frac{1}{2}} \big)
			{\| u \|}_{H^1}.
			\end{equation}
			
			\smallskip
          
			\item The following bound holds
			\begin{align}\label{eq:weiF}
			{\| \partial_{k} f^\# \big\|}_{L^2} \lesssim {\| \jx f \|}_{L^2}.
			\end{align}
			Moreover, for any $\beta < \gamma - 5/2$,
			\begin{align}\label{eq:weiFR}
			{\Big\| \partial_{k} \int \mathcal{K}^\#_R(x,k) f(x) \, dx \Big\|}_{L^2}
			\lesssim {\big\| \jx^{-\beta} f \big\|}_{L^2}.
			\end{align}
		
	     \end{enumerate}
	\end{proposition}

    \begin{proof}$(i)$
		To establish \eqref{FsharpL2}, we first observe that 
        \begin{align} \label{fsharpftil}
            |f^{\sharp}(k)|^2 + |f^{\sharp}(-k)|^2 = |\widetilde{f}(k)|^2 + |\widetilde{f}(-k)|^2 \ .
        \end{align}
        Indeed, we have for $k \ge 0$
        \begin{align}
             |f^{\sharp}(k)|^2 + |f^{\sharp}(-k)|^2 = \int_{\R \times \R} \big(\overline{\mathcal{K}^\#(x,k)}\mathcal{K}^\#(y,k) + \overline{\mathcal{K}^\#(x,-k)}\mathcal{K}^\#(y,-k)\big) f(x)\overline{f(y)} \ dxdy \ .
        \end{align}
        We also have
        \begin{align}
            \overline{\mathcal{K}^\#(x,k)}\mathcal{K}^\#(y,k) & = |T(k)|^2 [ A_+\overline{\psi_+(x,k)} + A_-\overline{\psi_-(x,k)}] [ A_+\psi_+(y,k) + A_-\psi_-(y,k)] \\
            & =  |T(k)|^2 \Big[ A_+^2\overline{\psi_+(x,k)}\psi_+(y,k) + A_-^2\overline{\psi_-(x,k)}\psi_-(y,k) \\
            & + A_+A_-\big(\overline{\psi_+(x,k)}\psi_-(y,k) + \overline{\psi_-(x,k)}\psi_+(y,k)\big)\Big]
        \end{align}
        and 
        \begin{align}
            \overline{\mathcal{K}^\#(x,-k)}\mathcal{K}^\#(y,-k) & = |T(k)|^2 \Big[ B_+^2\overline{\psi_+(x,k)}\psi_+(y,k) + B_-^2\overline{\psi_-(x,k)}\psi_-(y,k) \\
            & + B_+B_-\big(\overline{\psi_+(x,k)}\psi_-(y,k) + \overline{\psi_-(x,k)}\psi_+(y,k)\big)\Big] \ .
        \end{align}
        Now using the relations \eqref{eq:ABrela2} on the coefficients $A_\pm$ and $B_\pm$, we get 
        \begin{align}
            \overline{\mathcal{K}^\#(x,k)}\mathcal{K}^\#(y,k) + \overline{\mathcal{K}^\#(x,-k)}\mathcal{K}^\#(y,-k) = \overline{\mathcal{K}(x,k)}\mathcal{K}(y,k) + \overline{\mathcal{K}(x,-k)}\mathcal{K}(y,-k) \ ,
        \end{align}
        which then leads us to (\ref{fsharpftil}). Thus \eqref{FsharpL2} follows directly from the relation (\ref{fsharpftil}) and the proposition \ref{plancherel}.

        To show that $\Fou^\sharp$ is an isomorphism, it suffices to prove that its range is $L^2$. For that sake, we consider $g \in (\ran \ \Fou^\sharp)^\perp$. That is for any $f \in L^2$, 
        $$\int_\R f^\sharp(k)g(k) \ dk = 0 \ .$$  
        Now from the relationship \eqref{matKunitary}, we have 
        \begin{align*}
            \begin{cases}
                f^\sharp(k) = A_+\widetilde{f}(k)  + A_-\widetilde{f}(-k) \qquad \text{ for } k\ge 0,\\
                f^\sharp(k) = B_-\widetilde{f}(k)  + B_+\widetilde{f}(-k) \qquad \text{ for } k < 0.
            \end{cases}
        \end{align*}
        That leads to 
        \begin{align}
       \int_\R  \widetilde{f}(k) \Big[ A_+g(k)\mathbf{1}_+(k) + A_-g(-k)\mathbf{1}_-(k) + B_-g(k)\mathbf{1}_-(k) + B_+g(-k)\mathbf{1}_+(k)\Big] \ dk = 0 \ .
            \label{nulliso}
        \end{align}
        From Proposition \ref{plancherel}, we know that $\widetilde{\Fou}$ is an isomorphism from $L^2$ to $L^2$. Then \eqref{nulliso} implies    
        $$A_+g(k)\mathbf{1}_+(k) + A_-g(-k)\mathbf{1}_-(k) + B_-g(k)\mathbf{1}_-(k) + B_+g(-k)\mathbf{1}_+(k) = 0 \ .$$
        For $k \ge 0$, we can write the equation as 
        $$ \begin{bmatrix}
            A_+ & B_+ \\
            A_- & B_- 
        \end{bmatrix} \begin{bmatrix}
            g(k) \\
            g(-k)
        \end{bmatrix} = 0 .$$
        Since $\begin{bmatrix}
            A_+ & B_+ \\
            A_- & B_- 
        \end{bmatrix}$ is unitary, we get $g \equiv 0$.
        \medskip

        $(ii)$ The inversion formula \eqref{Fsharp-1} follows by standard argument analogous to Proposition \ref{plancherel}.

\medskip
		$(iii)$ The diagonalization \eqref{diagosharp} follows from the fact that the modified kernel $\K^\sharp(x,k)$ is made of generalized eigenfunctions of $-\partial_{xx} + V$ associated to the spectral value $k^2$. The functional calculus \eqref{Fsharpm} follows directly from the diagonalization \eqref{diagosharp}.

    \medskip
		
		$(iv)$ The fact that $\mathcal{F}^\#$ maps $L^1$ to bounded functions that decay at infinity
		is a consequence of the Riemann-Lebesgue lemma and decay properties of the normalized Jost functions $m_\pm(x,k)$. Continuity follows from dominated convergence, since $k \mapsto \K^\#(x,k)$ is continuous.
		
\medskip
$(v)$ For \eqref{kusharp} one can first use \eqref{FsharpL2} and \eqref{Fsharpm} to see that
\begin{align*}
{\| u \|}_{H^1}^2 & = - \int_\R u u_{xx} \, dx + \int_\R u^2 = 
  \int_\R u Hu \, dx - \int_\R u Vu \, dx + \int_\R u^2 \, dx
\\
& = %\lesssim 
{\big\| |k| u^\sharp \big\|}_{L^2}^2 - \int_\R u Vu \, dx %+ {\| V \|}_{L^1}{\| u \|}_{L^\infty}^2 
+ {\| u^\sharp \|}_{L^2}^2.
\end{align*}
Then, using \eqref{Fsharp-1} and $|\mathcal{K}^\sharp| \lesssim 1$,
one can estimate ${\| u \|}_{L^\infty} \lesssim {\| u^\sharp \|}_{L^1} \lesssim {\| \jk u^\sharp \|}_{L^2}$
and
\begin{align*}
\Big| \int_\R u Vu \, dx \Big| \lesssim {\| V \|}_{L^1}  {\| \jk u^\sharp \|}_{L^2}^2
\end{align*}
to obtain the first inequality in \eqref{kusharp}. The second inequality follows similarly using ${\| u \|}_{L^\infty} \lesssim {\| u \|}_{H^1}$.

\medskip
$(vi)$ For \eqref{eq:weiF}, we first use \eqref{Ksharpdecomp}-\eqref{KsharpSdecomp} to write 
		\begin{align}\label{dkfsharp}
		\begin{split}
		\partial_k f^\# (k) & = \int_\R \partial_k \overline{\mathcal{K}^\#_0(x,k)} f(x) \, dx
		+ \int_\R \chi_+(x) \partial_k \overline{\mathcal{K}^\#_+(x,k)} f(x) \, dx
		\\
		& + \int_\R \chi_-(x) \partial_k \overline{\mathcal{K}^\#_-(x,k)} f(x) \, dx
		+ \int_\R \partial_k \overline{\mathcal{K}^\#_R(x,k)} f(x) \, dx.
		\end{split}
		\end{align}
		For the first term on the right-hand side above we see from \eqref{KsharpS0} that
		\begin{align*}
		\int_\R \partial_k \overline{\mathcal{K}^\#_0(x,k)}  f(x) \, dx
		 &= \int_\R (\chi_+(x)aA_- + \chi_-(x)B_+)m_-(x,0)ixe^{ikx}f(x) \ dx  \\
         & + \int_\R(\chi_-(x)\frac{1}{a}A_+ + \chi_+(x)B_-)m_+(x,0)(-ix)e^{-ikx}f(x) \ dx
         \end{align*}
		so that taking $L^2$-norms and using (flat) Plancherel we get a bound by ${\| x f \|}_{L^2}$.
		For the second and third term on the right-hand side of \eqref{dkfsharp}, 
		we recall the definitions \eqref{KsharpSpm}
		and use again (flat) Plancherel together with the fact that $|\partial_k a_{\epsilon_1}^{\epsilon_2}(k)|\lesssim 1$,
		$\epsilon_1,\epsilon_2 \in \{+,-\}$.
		
		We are left with estimating the last term  on the right-hand side of \eqref{dkfsharp},
		for which we prove directly the stronger estimate \eqref{eq:weiFR}.
		We start from the formula \eqref{KRk>0},\eqref{KRk<0} to write
		%(notice that we drop the $k$ independent terms)
		\begin{align}\label{dkKR}
		\begin{split}
		& \int_\R \partial_k \mathcal{K}^\#_R(x,k) f(x) \, dx
		\\
		& = \int_\R \chi_+(x)\partial_k\Bigg\{\Big[ \mathbf{1}_+(k)\gamma_+(k)(m_+(x,k) - m_+(x,0)) + \mathbf{1}_+(k)(\gamma_+(k) - \gamma_+(0))(m_+(x,0) - 1) 
	\\
	&  + B_-\mathbf{1}_-(k)(m_+(x,k) - m_+(x,0)) \Big]e^{ikx} \Bigg\}f(x) \ dx\\
    & + \int_\R \chi_+(x) \partial_k \Bigg\{\Big[ \mathbf{1}_-(k)\lambda_-(k)(m_+(x,-k) - m_+(x,0)) + \mathbf{1}_-(k)(\lambda_-(k) - \lambda_-(0))(m_+(x,0) - 1) \\
    &  + A_-\mathbf{1}_+(k)(m_+(x,-k) - m_+(x,0)) \Big] e^{-ikx} \Bigg\}f(x) \ dx \\
       & + \int_\R \chi_-(x) \partial_k \Bigg\{\Big[ \mathbf{1}_-(k)\gamma_-(k)(m_-(x,-k) - m_-(x,0)) + \mathbf{1}_-(k)(\gamma_-(k) - \gamma_-(0))(m_-(x,0) - 1)  \\
       & + A_+\mathbf{1}_+(k)(m_-(x,-k) - m_-(x,0))\Big]e^{ikx} \Bigg\} f(x) \ dx \\
       & + \int_\R \chi_-(x) \partial_k \Bigg\{\Big[ \mathbf{1}_+(k)\lambda_+(k)(m_-(x,k) - m_-(x,0)) + \mathbf{1}_+(k)(\lambda_+(k) - \lambda_+(0))(m_-(x,0) - 1)  \\
       & +  B_+\mathbf{1}_-(k)(m_-(x,k) - m_-(x,0)) \Big]e^{-ikx}\Bigg\} f(x) \ dx
		\end{split}
		\end{align}
		Note that the symbol in curly brackets is Lipschitz in $k$.
		Since the contributions from all four integrals on the right-hand side of \eqref{dkKR} can be treated in the same way, 
		we only look at the first one.
		
		We can write it as
		\begin{align}
		\nonumber
		& \int_\R \chi_+(x) \, 
		\Big\{ \partial_k a_+^+(k) (m_{+}(x,0)-1) 
		\\ 
		\label{dkKR1}
		& \qquad + \partial_k \gamma_+(k) \mathbf{1}_+(k) 
		(m_{+}(x,k)-m_{+}(x,0)) \big\} e^{ixk} f(x) \, dx
		\\
		\label{dkKR2}
		+ & \int_\R \chi_+(x) \, \gamma_+(k) \partial_k \big[\mathbf{1}_+(k) (m_{+}(x,k)-m_{+}(x,0)) \big] \, e^{ixk} 
		f(x) \, dx
		\\
		\nonumber
		+ & \int_\R \chi_+(x) \,\big\{ a_+^+(k) (m_{+}(x,0)-1) 
		\\
		\label{dkKR3}
		& \qquad + \gamma_+(k) \mathbf{1}_+(k) (m_{+}(x,k)-m_{+}(x,0)) \big\}
		(ix) e^{ixk} f(x) \, dx \\
        + & \int_\R B_-\chi_+(x)\ \Big\{ \partial_k\big[\mathbf{1}_-(k)(m_+(x,k) - m_+(x,0))\big]  \nonumber \\
        \label{dkKR4}
        &\qquad + \mathbf{1}_-(k)(m_+(x,k) - m_+(x,0))(ix) \Big\}e^{ikx}f(x) \ dx.
		\end{align}
		For \eqref{dkKR1}, using $|\partial_k a_+^+(k)|, |\partial_k\gamma_+(k)| \leq 1$, we have
		\begin{align*}%\label{dkKR1'}
		\big| \eqref{dkKR1} \big| & \lesssim \Big| \int_\R \chi_+(x) \,(m_{+}(x,0)-1) e^{ixk} f(x) \, dx \Big|
		\\
		& + \Big| \int_\R \chi_+(x) (m_{+}(x,k)- 1 %m_{+}(x,0)
		) e^{ixk} f(x) \, dx \Big|.
		\end{align*}
		We can then use (flat) Plancherel with \eqref{estimR+} for the first line and \eqref{eq:mPDO+-1} ($\gamma > 3/2 + \beta$) for the second line, to estimate the $L^2$ norm of \eqref{dkKR1} by ${\| \jx^{-\beta}f \|}_{L^2}$ as desired.
		
		For \eqref{dkKR2}, 
			we first notice that
			$\partial_k [ \mathbf{1}_+(k) (m_{+}(x,k)-m_{+}(x,0)) ] = \mathbf{1}_+(k) \partial_k m_{+}(x,k)$,
			and then use  $|\gamma_+(k)| \lesssim 1$ and \eqref{eq:pdomkd} ($\gamma > 5/2 + \beta$)
			to estimate
			\begin{align*}
			{\big\| \eqref{dkKR2} \big\|}_{L^2} & \lesssim 
			{\Big\| \int_\R \chi_+(x)  \partial_k m_+(x,k) e^{ixk} f(x) \, dx \Big\|}_{L^2}
			\lesssim {\| \jx^{-\beta} f \|}_{L^2}.
			\end{align*}
		
		%From Lemma \ref{lemmstrong} and the definition \eqref{mdiffk0},
		%we can see that both $a(x,k)$ and $ka(x,k)$ satisfy the assumption of Lemma \ref{Hwang}
		%(provided we assume \eqref{Vassume} with $N\geq 3$ say):
		%for $b,c,d \in\{0,1\}$
		%\begin{align*}
		%\Big| (k\partial_k)^d \partial_x^b \partial_k^c \int_0^1 (\partial_2 m_+)(x, tk) \, dt \Big| 
		%  & \lesssim 
		%  \Big| k^d \int_0^1 t^{c+d} (\partial_x^b \partial_2^{c+d+1} m_+)(x, tk) \, dt \Big| \\
		%& \lesssim |k|^d \int_0^1 t^{c+d} \langle kt \rangle^{-(2+c+d)} \, dt \lesssim 1;
		%\end{align*}
		%we thus get a bound for \eqref{dkKR2} by the right-hand side of \eqref{eq:weiF}.
		
		For \eqref{dkKR3}, we can estimate
		\begin{align*}
		\big| \eqref{dkKR3} \big| & \lesssim \Big| \int_\R \chi_+(x) \, (m_{+}(x,0)-1) e^{ixk} (ix) f(x) \, dx \Big|
		\\
		& + \Big| \int_\R \chi_+(x) (m_{+}(x,k)-m_{+}(x,0)) e^{ixk} (ix) f(x) \, dx \Big|.
		\end{align*}
		Using again (flat) Plancherel with \eqref{estimR+} for the first line,
		and the $L^2$-bound for PDOs \eqref{eq:mPDO+-1} for the second line (with $\gamma \ge \beta  + 3/2$), we obtain a bound by ${\| \jx^{-\beta} f\|}_{L^2}$ as desired. \\
        Finally, for \eqref{dkKR4} we get the inequality with the same trick and this concludes the proof of \eqref{eq:weiFR} and of the proposition.
	\end{proof}

	\medskip
	\begin{definition}[Singular and Regular `Projections'] \label{singregproj}
		According to the decomposition \eqref{Ksharpdecomp}-\eqref{KsharpSdecomp}, 
		given $\phi\in \mathcal{S}$ we can write
		\begin{align}\label{decompphi}
		\begin{split}
		& \phi = \phi_S + \phi_R, \qquad \phi_S = \phi_0 + \phi_+ + \phi_-,
		\\
		& \phi_\ast(x) := \chi_\ast (x)  \int_\R \mathcal{K}_\ast^\#(x,k) \, 
		(\mathcal{F}^\# \phi) (k) \, \frac{dk}{2\pi}, \qquad \ast \in \{0,+,-\}
		\\
		& \phi_R(x) := \int_\R \mathcal{K}_R^\#(x,k) \, 
		(\mathcal{F}^\# \phi) (k) \, \frac{dk}{2\pi}.
		\end{split}
		\end{align}
		We can then extend these definitions to $L^2$ analogously to proposition \ref{propsharpF}.
	\end{definition}
	
	%Moreover $\phi_S$ and $\phi_R$ are projection in all the relevant spaces that we consider.

%-------------------------------------------
\bigskip
\section{Linear estimates: decay and smoothing}\label{seclinest}
In this section, we closely follow  (\cite{CPNLS2}, Section 3) and (\cite{GPR}, Section 3). For the reader's convenience, we restate the relevant results from (\cite{CPNLS2},\cite{GPR}). First, we establish, through the modified distorted Fourier transform, the $1/\sqrt{t}$ decay in the linear regime. That emphasizing the dispersive nature of the problem (see Corollaries \ref{cor:pointwisesingular}, \ref{cor:regularLinfty} and Lemma \ref{lem:pointwiseH}). Then, we provide smoothing estimates for the Schrödinger-type flows (see corollary \ref{cor:smoothing}). We conclude this section with some local decay bounds: one tailored for "improved low-frequency" i.e. pseudo linear solutions with Lipschitz coefficients vanishing at 0 (lemma \ref{lem:lowlocaldecay}), and another $L^2$ local decay  for the differentiated flows (lemma \ref{lem:localderivative} and corollary \ref{cor:locdecdiffflow}).

\medskip
	\subsection{Dispersive decay}\label{ssecdisp} 
	
	\begin{lemma}
		\label{lemstat}
		Consider a function $a(x,k)$ defined on $B \times \mathbb{R}_{+}$, for some $B \subset \R$, and such that
		\begin{equation}
		\label{hypest}
		| a(x, k)| +|k| |\partial_{k}a(x,k)| \lesssim 1, \quad \forall x \in B, \, \forall k \in \mathbb{R}_{+}
		\end{equation}
		and for every  $X \in \mathbb{R}$, consider the oscillatory  integral
		$$
		I(t,X,x) = \int_{0}^{\infty} e^{it(k-X)^2} a(x,k)  h(k) \, dk, \quad t>0, \, x \in B.
		$$
		Then, the following estimate holds
		\begin{equation}
		\label{decaylem} | I(t,X,x)|  \lesssim \frac{1}{\sqrt t} \| h  \|_{L^\infty} + \frac{1}{t^{\frac{3}{4}}} \| \partial_k h \|_{L^2},
		\end{equation}
		 uniformly for $X\in \mathbb{R}$, $t >0$ and $x\in B$.
	\end{lemma}
    \begin{proof}
        See (\cite{GPR}, Lemma 3.3).
    \end{proof}
	
	Given the decomposition \eqref{Ksharpdecomp}, 
	we can establish pointwise decay estimates associated to each piece of the decomposition. 
	
	\begin{corollary}\label{cor:pointwisesingular}
		Using the notation in \eqref{KsharpSdecomp}, the following estimates hold
		\begin{align}\label{eq:singulardecaysep}
		\left\Vert\chi_\ast(x)\int_\R \mathcal{K}_{\ast}^{\#}(x,k)e^{ik^{2}t}h(k)\,dk\right\Vert_{L^{\infty}_x}
		\lesssim\frac{1}{\sqrt{t}}{\big\| h \big\|}_{L^\infty_k}
		+\frac{1}{t^{\frac{3}{4}}}{\big\| \partial_k h \big\|} _{L^2_k},
		\end{align}
		for $\ast\in\{0,+,-\}$. 
		In particular, we get
		\begin{align}\label{eq:singulardecay}
		\left\Vert  \int_\R \mathcal{K}^{\#}_S(x,k)e^{ik^{2}t}h(k)\,dk\right\Vert_{L^{\infty}_x}
		\lesssim\frac{1}{\sqrt{t}}{\big\| h \big\|}_{L^\infty_k}
		+\frac{1}{t^{\frac{3}{4}}}{\big\| \partial_k h \big\|} _{L^2_k}.
		\end{align}
	\end{corollary}

	\begin{proof}
		We only consider the pieces with $k\geq0$, since $k<0$ is similar.  
		%The estimates above follow from Lemma \ref{lemstat}.  
		We define
		\begin{equation}
	J_{\ast}(t,x):=\chi_{\ast}(x)\int_{k\geq0} \mathcal{K}_{\ast}^{\#}(x,k)e^{ik^{2}t}h(k)\,dk, \qquad \text{ for } \ast \in \{0,\pm\}.
		\end{equation}
        First, we note $X=  -x/(2 t)$.
        
        For $\ast = 0$, from the expression \eqref{KsharpS0} of $\K^\sharp_0(x,k)$, we can write 
		\begin{align*}
        J_{0}(t,x) & = e^{ -i X^2}(\chi_+(x)aA_- + \chi_-(x)B_+)m_-(x,0)\int_{0}^{+\infty} e^{it (k-X)^2}   h(k)\, dk \\
        & \qquad + e^{ -i X^2}(\chi_-(x)\frac{1}{a}A_+ + \chi_+(x)B_-)m_+(x,0)\int_{0}^{+\infty} e^{it (k+ X)^2} h(k)\, dk.
        \end{align*}
        Then, the result follows from Lemma \ref{lemstat} by taking $a(x,k) =1$.

        For $\ast =\pm$, we will consider the case $\ast = +$, since the other case is similar. From the expression \eqref{KsharpSpm} of $\K^\sharp_+(x,k)$, we can write 
        \begin{align*}
            J_+(t,x) &=  e ^{-iX^2}\chi_+(x)\int_0^{+\infty} e^{i(t+X)^2} (\gamma_+(k)-\gamma(0))h(k) \ dk.
        \end{align*}
        The function $k \mapsto \gamma_+(k)-\gamma(0)$ satisfies \eqref{hypest}. Indeed, from \eqref{coefformula+}, we know that $\gamma_+(k) =  A_+T(k) + A_-R_+(k)$ which is bounded since $T(k)$ and $R_+(k)$ are bounded. In addition, $k\partial_k\gamma_+(k)$ is also bounded due to Lemma \ref{estim TR}. Hence, the estimate follows by Lemma \ref{lemstat}.  
       
	\end{proof}

	% Recall the definition of the regular part from \eqref{KRk>0} and \eqref{KRk<0}.
	% We have the following $L^\infty$ decay for the regular part of  the Schr\"odinger flow:
	% % again as a corollary of Lemma \ref{lemstat}.
	
	\begin{corollary}\label{cor:regularLinfty}
		If $\jx^\gamma V(x)\in L^1$ with $\gamma\geq\beta+1$, then
		\begin{align}\label{regularinfty}
		\left\Vert \jx^\beta \int_\R \mathcal{K}^{\#}_R(x,k)e^{ik^{2}t}h(k)\,dk\right\Vert_{L^{\infty}_x}
		\lesssim\frac{1}{\sqrt{t}}{\big\| h \big\|}_{L^\infty_k}
		+\frac{1}{t^{\frac{3}{4}}}{\big\| \partial_k h \big\|} _{L^2_k}.
		\end{align}
	\end{corollary}
	
	\begin{proof} We only analyze the case $k\geq0$, since the case $k<0$ is similar.
		From the definition of the regular part \eqref{KRk>0}, we focus on the contribution of $e^{ikx}$, since the case of $e^{-ikx}$ is similar. 
        Then, recalling the definition of $\mathcal{G}_{+}^{+}(x,k)$ in \eqref{kR}, we are led to estimate:
\begin{align*}
 &  e^{-iX^2}\int_{k \ge 0} e^{i(t-X)^2}\jx^{\beta} \mathcal{G}_{+}^{+}(x,k)h(k) \ dk , 
\end{align*}
where $X=  -x/(2 t)$. We consider $a(x,k)  :=  \jx^{\beta} \mathcal{G}_{+}^{+}(x,k)$. 
Lemma \ref{estimG+-} gives 
$$ |a(x,k)| + |k\partial_ka(x,k)| \lesssim \jx^{\beta -\gamma +1} .$$
Then, the functions $a(x,k)$ satisfies the assumptions \eqref{hypest}. Lemma \ref{lemstat} yields the desired estimate.
% First, the function $a(x,k)$ is bounded. Indeed, since $\gamma_+(k)$ is bounded, it suffices to show that $\chi_\pm(x)\jx^\beta(m_\pm(x,k) - 1 )$ is also bounded. This follows from Lemma \ref{mestim} that gives 
% \begin{align*}
%     \chi_\pm(x)\jx^\beta|m_\pm(x,k) - 1 | & \lesssim \jx^\beta\mathcal{W}^1_\pm(x).
% \end{align*}
% We get the boundedness, using the fact that $\beta +1 \le  \gamma$ and $\jx^\gamma V(x)\in L^1$.

% Second, the function $k\partial_ka(x,k)$ is bounded. Indeed, it suffices to prove that $k\partial_k \gamma_+(k)$  and $\chi_\pm(x)\jx^\beta k\partial_{k}m_\pm(x,k) $ are bounded. The first boundedness follows from Lemma \ref{estim TR}. For the second one, Lemma \ref{estim m} tells us 
% \begin{align*}
%     \chi_\pm(x)\jx^\beta| \partial_{k}m_\pm(x,k)| 
% 		& \lesssim \frac{1}{| k|}\jx^\beta\mathcal{W}^{1}_\pm(x).
% \end{align*}
% Thus, we get the boundedness using the fact that $\beta +1 \le  \gamma$ and $\jx^\gamma V(x)\in L^1$. Hence, we get the estimate by invoking Lemma \ref{lemstat}.
	\end{proof}

	\smallskip
	Putting together the results from Corollary \ref{cor:pointwisesingular}
	and Corollary \ref{cor:regularLinfty} we can obtain the following pointwise decay estimate:
	
	\begin{lemma}\label{lem:pointwiseH}
		Suppose $\jx^\gamma V(x)\in L^1$ with $\gamma\ge 5/2$ and assume the potential $V$ is non-generic.
		Then, the Schr\"odinger flow has the following dispersive estimate: for all $t \ge 0$,
		\begin{equation}\label{eq:linearpoinwiseH}
		{\big\| e^{itH}h \big\|}_{L^\infty_x}  \lesssim \frac{1}{\sqrt{t}} {\big\| h^\# \big\|}_{L^\infty_k}
		+\frac{1}{t^{\frac{3}{4}}}{\big\| \partial_k h^\# \big\|} _{L^2_k}.
		\end{equation}
	\end{lemma}

	\begin{proof}
		The functional calculus of Proposition \ref{propsharpF} gives
		\begin{align}
		e^{itH}h 
		= \int_\R \mathcal{K}^{\#}(x,k) \mathcal{F}^{\#}(e^{itH} h) \,\frac{dk}{2\pi}
		& = \int_\R \mathcal{K}^{\#}(x,k) e^{itk^2} h^{\#}(k)\,\frac{dk}{2\pi}.
		\end{align}
		By the decomposition \eqref{Ksharpdecomp}, 
		the desired decay estimate is a direct consequence of 
		Corollary \ref{cor:pointwisesingular} and Corollary \ref{cor:regularLinfty} with $\beta=0$.
	\end{proof}

\subsection{Smoothing estimates}\label{ssecsmooth}
	
\begin{lemma}\label{lem:smoothingsim}
		Let  $\mathcal{Q}: \R_x\times \R_k \mapsto \C$ be such that, for some $\beta \in \R$,
		\begin{equation}\label{eq:Qbound}
		\sup_{x\in\mathbb{R},\,k\in\mathbb{R}} \big| \left\langle x\right\rangle^{\beta}\mathcal{Q}(x,k) \big| < \infty,
		\end{equation}
		and let $\phi: \R \mapsto \C$ be such that $|\phi(k)|\lesssim \sqrt{|k|}$. Then, for all $t\geq0$,
		\begin{equation}\label{eq:moregeneralsmoothing}
		\left\Vert \left \langle x\right\rangle^{\beta} \int_\R \mathcal{Q}(x,k) \phi(k) e^{ik^{2}s}h(k)\,dk\right\Vert _{L_{x}^{\infty}L_{s}^{2}([0,t])}
		\lesssim \left\Vert h\right\Vert _{L^{2}}.
		\end{equation}
		
		Under the same assumptions, the following inhomogeneous estimate holds:
		\begin{align}\label{eq:generalimhomsmoothing}
		\left\Vert \int_0^t \left[ \int_\R e^{-ik^{2}s} \overline{\phi(k)} \overline{\mathcal{Q}(y,k)}F(s,y)\,dy \right] \,ds\right\Vert_{L_{k}^{2}}
		\lesssim & \left\Vert \left\langle x\right\rangle^{-\beta}F\right\Vert_{L_{x}^{1}L_{s}^{2}([0,t])}.
		\end{align}
		
	\end{lemma}
    \begin{proof}
        See (\cite{CPNLS2}, Lemma 3.5).
    \end{proof}
    Note that \eqref{eq:generalimhomsmoothing} follows from \eqref{eq:moregeneralsmoothing} by duality.\\

   Since some parts of the nonlinear analysis require treating small and large frequencies separately, we establish the following tailored corollary.
	
	\begin{corollary}\label{cor:smoothing}
		Let $\mathcal{Q}:\R_x\times \R_k \mapsto \C$ be such that for some $\beta\in\mathbb{R}$
		\begin{equation}\label{eq:Qbound'}
		\sup_{x\in\mathbb{R},\,k\in\mathbb{R}}\left|\left\langle x\right\rangle ^{\beta}\mathcal{Q}(x,k)\right|
		<\infty,
		\end{equation}
        and let $\phi_1$ and $\phi_2$ be non-negative cut-off functions such that $\phi_1$ is supported in $[-1,1]$ and $\phi_1 +\phi_2 =1$.
		Then, the following estimates hold, for all $t>0$:
		\begin{align}\label{eq:smoothingQ}
		\left\Vert \left\langle x\right\rangle ^{\beta}\int_\R \phi_1(k)\,k\,\mathcal{Q}(x,k) e^{ik^{2}s}h(k)\,dk
		\right\Vert_{L_{x}^{\infty}L_{s}^{2}([0,t])}
		&\lesssim\left\Vert h\right\Vert _{L^{2}}, \\
		% By duality, the following inhomogeneous estimate holds:
		\label{eq:smoothingQim}
		\left\Vert \int_0^t \left[\phi_1(k) \int_\R  e^{-ik^{2}s} \,k\, \overline{\mathcal{Q}(y,k)} F(s,y)\,dy\right]\,ds\right\Vert _{L_{k}^{2}}
		& \lesssim {\big\| \left\langle x\right\rangle^{-\beta}F \big\|}_{L_{x}^{1}L_{s}^{2}([0,t])}, \\
		%For the high-frequency part, we can conclude
		% Under the same assumptions above
\label{eq:smoothingQh}
		\left\Vert \left\langle x\right\rangle ^{\beta}\int_\R \phi_2(k)
		\mathcal{Q}(x,k)e^{ik^{2}s}h(k)\,dk\right\Vert _{L_{x}^{\infty}L_{s}^{2}([0,t])} &\lesssim\left\Vert h\right\Vert _{L^{2}}, \\
		% By duality, the following inhomogeneous estimate holds:
\label{eq:smoothingQimh}
		\left\Vert \int_0^t \left[\phi_2(k) \int_\R  e^{-ik^{2}s}\overline{\mathcal{Q}(y,k)}F(s,y)\,dy\right]\,ds\right\Vert _{L_{k}^{2}}
		&\lesssim  {\big\| \left\langle x\right\rangle^{-\beta}F \big\|}_{L_{x}^{1}L_{s}^{2}([0,t])}.
		\end{align}
	\end{corollary}
	
	\begin{proof}
		These results follow directly from Lemma \ref{lem:smoothingsim} by choosing $\phi(k)=\phi_1(k) k$  
		and $\phi(k)=\phi_2(k)$.
		
	\end{proof}
    
\subsection{Local decay}\label{sseclocdec}
	We have the following improved local decay estimate for flows with improved low-frequency behavior:
	
	\begin{lemma}\label{lem:lowlocaldecay}
		Let $\mathcal{Q}:\R_x\times \R_k \mapsto \C$ be such that for some $\beta\in\mathbb{R}$
		\begin{equation}\label{eq:QboundP}
		\sup_{x\in\mathbb{R},\,k\in\mathbb{R}}\left|\left\langle x\right\rangle^{\beta}
		\partial_{k}^{j}\mathcal{Q}(x,k)\right|<\infty, \qquad j=0,1,
		\end{equation}
		and let  $\phi$ be a smooth function such that $|\phi(k)|\lesssim |k|$ for $|k|\leq1$
			and $|\phi(k)|$ is bounded for $|k|\geq1$. 
		Then, the following estimate holds 
		\[
		\left\Vert \left\langle x\right\rangle^{\beta}\int_\R \mathbf{1}_{\{\pm k\geq0\}}\phi(k)
		\mathcal{Q}(x,k)e^{ik^{2}t}h(k)\,dk\right\Vert_{L_{x}^{\infty}}
		\lesssim \frac{1}{\left\langle t\right\rangle }\left\Vert h\right\Vert _{H^{1}}.
		\]
	\end{lemma}
	
	\begin{proof}
		See (\cite{CPNLS2}, Lemma 3.8). 
	\end{proof}

Finally, we state a local decay estimate for the derivative of Schr\"odinger-type flows.

	\begin{lemma}\label{lem:localderivative}
		Let $\mathfrak{Q}:\R_x\times \R_k \mapsto \C$ be such that for some $\beta\in\mathbb{R}$
		\begin{equation}\label{eq:QboundP2}
	\sup_{x\in\mathbb{R},\,k\in\mathbb{R}}\left| \jx^\beta \partial_x^j \partial_k^\rho \mathfrak{Q}(x,k)\right|\lesssim 1,
	\qquad j,\rho=0,1,
		\end{equation}
		and assume that $\jx^\beta \partial_x^j \partial_k^\rho \mathfrak{Q}(x,k)$ 
		are symbols of $L^2$-bounded PDOs.
		%could be more precise with using $\beta-1$
		Then, the following estimate holds
		\begin{equation}\label{eq:QdecayL2}
		\left\Vert \left\langle x\right\rangle ^{\beta-1}\partial_{x}
		\int_\R \mathbf{1}_{\{\pm k\geq0\}}e^{ikx}\mathfrak{Q}(x,k)e^{ik^{2}t}h(k)\,dk\right\Vert_{L_{x}^{2}}
		\lesssim \frac{1}{\jt^{1/2}} \left\Vert h\right\Vert _{H^1}.
		\end{equation}
	\end{lemma}
	
	\begin{proof}
		See (\cite{CPNLS2}, Lemma 3.9).
	\end{proof}

\begin{corollary} \label{cor:locdecdiffflow}
    Using the notation \eqref{decompphi}, with $f=e^{-itH}u$ , we have  
		\begin{align}\label{dxuM}
		{\big\| \jx^{-1}\partial_x u_{M}(t) \big\|}_{L^2_x} \lesssim \jt^{-1/2} {\| f^\#(t) \|}_{H^1} 
		%\lesssim \js^{-1/2+\alpha} \varepsilon
		, \qquad M \in \{S,R\}.
		\end{align}
\end{corollary}
\begin{proof} For $M \in \{S,R\}$, we have 
\begin{align}
    u_M(t,x) = \int_\R \K^\sharp_M(x,k) e^{-ik^2t} f^\sharp(k)\frac{dk}{2\pi}.
\end{align}
We start with the case $M=S$. In view of the decomposition \eqref{KsharpSdecomp} of $\K^\sharp_S(x,k)$, it suffices to prove the estimate 
\begin{align}
    {\big\| \jx^{-1}\partial_x u_{\ast}(t) \big\|}_{L^2_x} \lesssim \jt^{-1/2} {\| f^\#(t) \|}_{H^1} 
		%\lesssim \js^{-1/2+\alpha} \varepsilon
		, \qquad \text{ for } \ast \in \{0,\pm \},
\end{align}
where 
\begin{align*}
    u_{\ast}(t,x) : = \int_\R \chi_*(x)\K^\sharp_{\ast}(x,k) e^{-ik^2t} f^\sharp(k)\frac{dk}{2\pi}.
\end{align*}

For $\ast =0$, recalling the expression \eqref{KsharpS0} of $\K^\sharp_0(x,k)$, we can write 
\begin{align*}
    u_0(t,x) &= (\chi_+(x)aA_- + \chi_-(x)B_+)m_-(x,0)\int_{\R} e^{-ikx}e^{-ik^2t}   f^\sharp(k)\, \frac{dk}{2\pi} \\
        & +(\chi_-(x)\frac{1}{a}A_+ + \chi_+(x)B_-)m_+(x,0)\int_{\R} e^{ikx}e^{-ik^2t} f^\sharp(k)\, \frac{dk}{2\pi}.
\end{align*}
We are going to treat the first term, since the analysis is similar for the second one. After taking derivatives with respect to $x$, we get :

\begin{align}\label{der1}
    &(\partial_x\chi_+(x)aA_- + \partial_x\chi_-(x)B_+)m_-(x,0)\int_{\R} e^{-ikx}e^{-ik^2t}   f^\sharp(k)\, \frac{dk}{2\pi}  \\
    \label{der2}
    &+(\chi_+(x)aA_- + \chi_-(x)B_+)\partial_xm_-(x,0)\int_{\R} e^{-ikx}e^{-ik^2t}   f^\sharp(k)\, \frac{dk}{2\pi} \\
    \label{der3}
    &+ (\chi_+(x)aA_- + \chi_-(x)B_+)m_-(x,0)\,\partial_x\int_{\R} e^{-ikx}e^{-ik^2t}   f^\sharp(k)\, \frac{dk}{2\pi}.
\end{align} 
 The $\jx^{-1}$ weighted $L^2_x$ norm of \eqref{der1} and \eqref{der2} is bounded, up to a multiplicative constant independent of $t$, by

\begin{align}\label{supintegral}
    \underset{x\in\R}{\sup}\Big| \int_\R e^{-ikx}e^{-ik^2t}   f^\sharp(k)\, \frac{dk}{2\pi}\Big| 
\end{align} 
This follows from the fact that $\partial_x\chi_\pm$ are compactly supported, the function $m_-(x,0)$ is bounded and  $\partial_xm_-(x,0) \le \jx^{-\gamma +1}$ from Lemma \ref{m0asymp}. Then \eqref{supintegral} is less than $\jt^{-1/2} {\| f^\#(t) \|}_{H^1}$ due to Lemma \ref{lemstat} and the one dimensional Sobolev embedding $\|f^\sharp\|_{L^\infty} \lesssim \| f^\#(t) \|_{H^1}$. The estimation of the third term \eqref{der3} follows directly from Lemma \ref{lem:localderivative} with $\beta =0$ and $\mathfrak{Q}(x,k)=1$.

For $* =\pm$, we will focus on the case $\ast =+$, since the other case is similar. From \eqref{KsharpSpm}, we have 
\begin{align*}
    u_+(t,x) &= \chi_+(x)\int_0^{+\infty} e^{ikx}(\gamma_+(k) - \gamma_+(0)) e^{-ik^2t} f^\sharp(k) \ dk \\
    &+  \chi_+(x)\int_{-\infty}^0 e^{-ikx}(\lambda_+(k) - \lambda_+(0)) e^{-ik^2t} f^\sharp(k) \ dk.
\end{align*}
We then focus on the first term, since the second can be treated similarly. After taking derivative in $x$, we get 
\begin{align*}
    &\partial_x\chi_+(x) \int_0^{+\infty} e^{ikx}(\gamma_+(k) - \gamma_+(0)) e^{-ik^2t} f^\sharp(k) \ dk \\
    & + \chi_+(x) \ \partial_x \int_0^{+\infty} e^{ikx}(\gamma_+(k) - \gamma_+(0)) e^{-ik^2t} f^\sharp(k) \ dk.
\end{align*}
The first term is treated as above, using the fact that $\partial_x\chi_+(x)$ is compactly supported and Lemma \ref{lemstat}. For the second one, we apply Lemma \ref{lem:localderivative} with $\beta =0$ and $\mathfrak{Q}(x,k)= \gamma_+(k) - \gamma_+(0)$, which satisfies the assumptions of Lemma \ref{lem:localderivative}. Indeed, $\gamma_+(k)$ and its derivative are bounded (Lemma \ref{estim TR}). This provides \eqref{eq:QboundP2} and $L^2$-PDOs boundedness assumption of Lemma \ref{lem:localderivative}.\\

We look at the case $M=R$. From the definition of the regular part \eqref{KRk>0} of the distorted kernel, we focus on the contribution of $e^{ikx}$, since the case of $e^{-ikx}$ is similar. 
Then, we are going to treat
\begin{align*}
&  \int_{k \ge 0} e^{ikx} \Big\{ \chi_+(x)\big[ \gamma_+(k)(m_+(x,k) - m_+(x,0)) \\
   & + (\gamma_+(k) - \gamma_+(0))(m_+(x,0) - 1) \big]+\chi_-(x) A_+(m_-(x,-k) - m_-(x,0)) \Big\} e^{-ik^2t}f^\sharp(k) \ dk .
\end{align*}
We then apply Lemma \ref{lem:localderivative} with $\beta =0$ and 
\begin{align*}
    \mathfrak{Q}(x,k) &= \chi_+(x)\big[ \gamma_+(k)(m_+(x,k) - m_+(x,0))  + (\gamma_+(k) - \gamma_+(0))(m_+(x,0) - 1) \big] \\
    & \qquad +\chi_-(x) A_+(m_-(x,-k) - m_-(x,0)).
\end{align*}
The function $\mathfrak{Q}(x,k)$ satisfies \eqref{eq:QboundP2}, as a consequence of the decay properties of Jost functions (Lemma \ref{estim m}) and the scattering data (Lemma \ref{estim TR}). It also satisfies the $L^2$-PDOs boundedness assumption of Lemma \ref{lem:localderivative} due to PDO bounds in Lemma \ref{lem:m-1}.

\end{proof}

%----------------------------------------------------
\def\eps{\epsilon}
	\def\epss{\eps_0,\eps_1
,\eps_1,\eps_2}
	\def\kk{k_0,k_1,k_2,k_3}

\bigskip
\section{The nonlinear spectral distribution for non-generic potentials}\label{secmu}
This section is devoted to the decomposition of the nonlinear spectral distribution (NSD)
\begin{equation}\label{musharpdef}
	\mu^{\#}(k,\ell,m,n):=\int_\R \overline{\mathcal{K}^{\#}(x,k)}\mathcal{K}^{\#}(x,\ell)
	\overline{\mathcal{K}^{\#}(x,m)}\mathcal{K}^{\#}(x,n)\,dx.
	\end{equation}
% \np{The rigorous way to understand $\mu^\sharp$ is as a limit in $\mathcal{S}'(\R^4)$. In fact, we consider a non-negative compactly supported function $\varphi$ on $\R$ such that $\varphi(0)=1$. The object $\mu^\sharp$ is defined as the limit in $\mathcal{S}'(\R^4)$ when $\eta \to 0$ of the following
% %$L^\infty$
% bounded functions 
% \begin{align*}
%     (k,\ell,m,n) \mapsto \int_\R \overline{\mathcal{K}^{\#}(x,k)}\mathcal{K}^{\#}(x,\ell)
% 	\overline{\mathcal{K}^{\#}(x,m)}\mathcal{K}^{\#}(x,n) \varphi(\eta x)\,dx.
% \end{align*}}

\subsection{Preliminaries} 
Following (\cite{CPNLS2}, Section 4), we consider the set of coefficients appearing in the decomposition of the singular part of the modified kernel \eqref{KsharpSdecomp}-\eqref{KsharpSpm}. We denote this set by
\begin{equation}
\label{defS}
\begin{split}
	 A := \Big\{ [\gamma_+(\cdot) - \gamma_+(0)) ]\mathbf{1}_+(\cdot),\ &[\lambda_-(\cdot) - \lambda_-(0)] \mathbf{1}_-(\cdot), \\
      &[\gamma_-(\cdot) - \gamma_-(0)] \mathbf{1}_-(\cdot),
\ [\lambda_+(\cdot) - \lambda_+(0)] \mathbf{1}_+(\cdot), \ 1,\ -1 \Big\}.
\end{split}
\end{equation}

We can notice that this set contains the coefficients $a^\epsilon_\iota$,
	$\epsilon,\iota \in \{+,-\}$ defined in \eqref{KsharpSpm}. These are all bounded continuous functions and their derivatives are also bounded (Lemma \ref{lem:Ksing}).
	
	To simplify the forthcoming writings, we use the following  notations:
	we denote signs $\{+,-\}$ by $\eps_j$ and let
	\begin{align}\label{convsum}
	\sum_\ast = \sum_{\epss \in \{+,-\}} \, .
	\end{align}
	%We will also use indexes $\iota \in \{0,+,-\}$.
	
	Given $(a_0,a_1,a_2,a_3) \in A^4$, we denote a fourfold tensor product with alternate conjugation as follows:
	\begin{align}\label{convprod}
	\prod_\ast %{j=1}^4 
	a_j(k_j) := \overline{a_0(k_0)} a_1(k_1) \overline{a_2(k_2)} a_3(k_3).
	\end{align}
	We will also use a similar notation for the product of four functions that also depend on the (same) variable $x$:
	\begin{align}\label{convprod'}
	\prod_\ast %{j=1}^4 
	b_j(x,k_j) := \overline{b_0(x,k_0)} b_1(x,k_1) \overline{b_2(x,k_2)} b_3(x,k_3).
	\end{align}
	
	For ${\bf \eps} = (\epss)$ and ${\bf k} = (\kk)$, we let
	\begin{align}\label{dot}
	{\bf \eps} \cdot {\bf k} := \eps_0k_0 + \eps_1k_1 + \eps_2k_2 + \eps_3k_3
	\end{align}
	be the standard dot product.
	
	We define $\mathcal{C}$ to be the set of functions on $\R^4$ that are a tensor product of $4$ elements in $A$:
	\begin{align}\label{defSet}
	\mathcal{C} := \big\{ f: {\bf k}:=(k_0,k_1,k_2,k_3) \in \R^4 \rightarrow \C, \, \, 
	\mbox{s.t.} \,  f({\bf k}) = \prod_\ast a_j(k_j), \,\, a_j \in A \big\},
	\end{align}
	%\begin{align*}
    
	$\mathcal{C} =  \bar{A}\otimes  A \otimes  \bar{A}\otimes A.$ %\approx A^4.
	%\end{align*}
We will also consider the above set minus the constant functions $1$ and $-1$:
	\begin{align}\label{defSet'}
	\mathcal{C}_L = \mathcal{C} \smallsetminus \{1,-1\}.
	\end{align}

\subsection{Decomposition of the NSD}
First of all, we provide the following Fourier transform formulas. We fix $\zeta$ an even smooth function with compact support and integral $1$.
\begin{lemma} \label{fouchi}
    We consider $\chi_+$ and $\chi_-$ as defined in \eqref{defchi+-}. For $r=1,2,3,4$, the following formulas hold
	\begin{align}\label{chi+-}
	\begin{split}
	\int_\R e^{-i\xi x}\chi_{+}^{r}(x)\,dx = 
	\pi\delta_{0}(\xi)+\pv\frac{\what{\zeta}(\xi)}{i\xi}+\what{\varpi_r}(\xi),
	\\
	\int_\R e^{-i\xi x}\chi_{-}^{r}(x)\,dx =
	\pi\delta_{0}(\xi)-\pv\frac{\what{\zeta}(\xi)}{i\xi}+\what{\varpi_r}(\xi),
	\end{split}
	\end{align}
	where $\varpi_r$ is  a  smooth function with compact support.
\end{lemma}
\begin{proof}These formulas follow from the fact that 
\begin{align}\label{fou1+}
    \widehat{\mathbf{1}_{+}}(\xi) = \pi \delta_0(\xi) + \pv\frac{1}{i\xi} \ .
\end{align}
First, one can observe that it suffices to treat the case of $\chi_+$. We notice that $\varpi_r : =\chi^r_+ - \zeta \ast \mathbf{1}_{+}$ is a smooth function with compact support. Thus, 
\begin{align*}
    \int_\R e^{-i\xi x}\chi_{+}^{r}(x)\,dx = 
\what{\zeta \ast \mathbf{1}_{+}}(\xi) + \what{\varpi_r}(\xi).
\end{align*}
The formula follows from the convolution property of the (flat) Fourier transform, the relation \eqref{fou1+} and the fact that $\what{\zeta}(0) = \int_\R \zeta(x) \ dx =1$.
\end{proof}

\begin{theorem}\label{theomu}
		% The NSD $\mu^\#$ defined in \eqref{musharpdef} can be decomposed into the sum of  delta functions $\mu_0^\sharp$,
		% a singular part $\mu_S^\sharp$ with (a priori) dangerous p.v. terms and an improved low frequency behavior (subscript $L$)
		% and a regular part (subscript $R$) as follows:
		Let ${\bf k}=(k_0,k_1,k_2,k_3) \in \R^4$. The following decomposition holds
		\begin{align}\label{mudecomp}
		  \mu^\#({\bf k}) = \mu_0^\sharp({\bf k}) + \mu_{\pv}^\#({\bf k}) + \mu_L^\#({\bf k}) \ + \mu^\#_R({\bf k}).
		\end{align}
		
		\setlength{\leftmargini}{1.5em}
		\begin{itemize}
			
			\medskip
            \item  The  distribution $\mu_0^\sharp$ is given by
            \begin{equation}\label{mu0}
            \begin{aligned}
		\mu_0^\#({\bf k}) := 2\pi(A_-^4 + A_+^4) \delta_0(k_0-k_1+k_2-k_3) + 4\pi A_-^2A_+^2 \Big[ \delta_0(k_0-k_1-k_2+k_3) \\+\delta_0(k_0+k_1+k_2+k_3) + \delta_0(k_0+k_1-k_2-k_3)\Big] .
		\end{aligned}
        \end{equation}
			\item  The distribution $\mu_\pv^\sharp$ is given by
        \begin{align}\label{mupv}
		 \mu_{\pv}^\#({\bf k}) := 2(A_-^2 - A_+^2)B_+A_+ \sum_{\eps \in E} \pv\frac{\what{\zeta}(\eps\cdot {\bf k})}{i\eps \cdot {\bf k}} ,
		\end{align}
        where $E : = \{ (1,-1,1,1) ; (1,-1,-1,-1) ;  (1,1,1,-1) ; (-1,-1,1,-1) \}$.\\
        \item
            Denoting ${\bf \eps} = (\epss) \in \{+,-\}^4$, the part  $\mu^\#_L({\bf k})$ is a finite $\C$-linear combination of terms of the form
			\begin{align}\label{muL}
			\begin{split}
			% \mu^\#_L({\bf k}) = \sum_\ast %{\epss \in \{+,-\}} 
			a_{\epss}({\bf k}) \Big[ \pi \delta({\bf \eps} \cdot {\bf k} ) %\eps_0k_0+\eps_1\ell+\eps_2m+\eps_3n) 
			\pm \pv \frac{\what{\zeta}({\bf \eps} \cdot {\bf k} )}{i{\bf \eps} \cdot {\bf k} } \Big] &,
			\quad \, \mbox{with} \quad a_{\epss} \in \mathcal{C}_L,
			\end{split}
			\end{align}
			where $\zeta$ is the function from Lemma \ref{fouchi}.
			
			\medskip
			\item The  distribution $\mu_R^\sharp$ is given by
			\begin{align}\label{muR0}
			\begin{split}
			\mu^\#_R({\bf k}) :=  \mu^\#_{R,1}({\bf k}) + \mu^\#_{R,2}({\bf k}),  
			\end{split}
			\end{align}
			where
			\begin{align}\label{muR1}
			\begin{split}
			\mu^\#_{R,1}({\bf k}) := \sum_{(A_0,A_1,A_2,A_3) \in \mathcal{X}_R} \int_{\R} \prod_\ast \mathcal{K}^\#_{A_j}(x,k_j) \, dx,
			\,\quad \mathcal{X}_R := \{ S,R \}^4 \smallsetminus \{(S,S,S,S)\},
			\end{split}
			\end{align}
			and $\mu^\#_{R,2}({\bf k})$ is a finite $\C$-linear combination of terms of the form
			\begin{align}\label{muR2}
			%\sum_\ast
			\begin{split}
			& a_{\epss}({\bf k}) \, \what{\Fou}({\varphi_{\epss}})({\bf \eps} \cdot {\bf k}),
			%\qquad \varphi_{\epss} \in \mathcal{S}.
			\\
			& \mbox{with} \quad a_{\epss} \in \mathcal{C},
			\quad \mbox{and} \quad \big|\partial_x^\alpha \varphi_{\epss} \big| \lesssim \jx^{-\gamma+1},\quad \alpha=0,1.
			\end{split}
			\end{align}
			
		\end{itemize}
		
	\end{theorem}

\bigskip

\begin{remark}  \label{remarkdecompmu}We make some remarks on the decomposition above and its relevance for the nonlinear analysis:
\begin{itemize}
    \smallskip
    \item  The piece $\mu^\#_R$ is called the regular part. In fact, it is a contribution of a local part $\mu^\#_{R,2}$ which is a linear combination
		of localized functions of ${\bf \eps} \cdot {\bf k}$, and a pseudodifferential part $\mu^\#_{R,1}$
		which involves at least one localized $\mathcal{K}^\#_R(x,k_j)$ kernel.
        
    \smallskip
    \item The pieces $\mu_L^\sharp$ and $\mu_\pv^\sharp$ are the singular parts of the NSD.

    \smallskip
    \item The $\mu^\#_0$ contribution is the easiest to handle in comparison to the other part of $\mu^\sharp$, since its analysis could be done as in the flat case ($V=0$).

    \smallskip 
    \item The part $\mu^\#_L$ is called the "improved low-frequency" part. In fact, the contributions appearing in \eqref{muL} come with Lipschitz coefficients vanishing at zero frequencies. This part can be handled using smoothing estimates as it was done in the Section 5 of \cite{CPNLS2}.

    \smallskip
    \item The part $\mu^\#_\pv$ is called "dangerous" $\pv$ part, due to the absence of low-frequency improvement as in $\mu^\#_L$. Nevertheless, one can see from \eqref{mupv} that if the coefficient $(A_-^2 - A_+^2)A_+B_+ = 0$, there are no "dangerous" $\pv$ part anymore. This null condition is equivalent to the situation where $a^2 =1$. Indeed, this is due to the relation 
     \begin{align}
        (A_-^2 - A_+^2)A_+B_+ = |a|\frac{a -a^3}{(a^2 +1)^2} 
    \end{align}
 where $a$ is the limit at $-\infty$ of the zero-energy resonance \eqref{a}. In this case, modified scattering has been essentially established in \cite{CPNLS2}. In fact, we can say that the following was established in \cite{CPNLS2}: in absence of "dangerous" $\pv$ part, for any $\alpha \in (0,1/4]$, there exists $\varepsilon >0$ small enough such that globally in time, the bound $\| \partial_kf^\sharp(t)\|_{L^2_k} \lesssim \varepsilon\jt^\alpha$ holds. This bound is of crucial importance in their argument for the pointwise estimate of $f^\sharp(t)$. We can observe that the pointwise analysis depends on the value of the scattering data $T(0), R_\pm(0)$ that are respectively $\pm 1,0$ (see Section $7$ in \cite{CPNLS2}). In view of Lemma \ref{lemngTR}, we understand that the result in \cite{CPNLS2} follows not necessarily because of the parity of the zero-energy resonance, but rather the parity of its limit at $\pm \infty$ i.e. $a^2 =1$.

 \smallskip
 \item The "dangerous" $\pv$ part are troublesome, since we could not use the improved low-frequency smoothing effects as in \cite{CPNLS2}. But, even if we could not close the bootstrap scheme to get global-in-time bounds, we obtain an "optimal" growth of $\varepsilon^3\jt^{1/4}$ when estimating $\| \partial_kf^\sharp(t)\|_{L^2_k}$ (see Section \ref{propgoalproof} for a discussion on this optimality).

    \smallskip
    \item  The refined decomposition performed in Theorem \ref{theomu} is of importance through its isolation of the "improved low-frequency" structure in the singular part of the NSD $\mu^\sharp.$ It helps to understand the subtlety of non-generic potentials without symmetry-type assumptions. Indeed, this formulation of Theorem \ref{theomu} shows the appearance of "dangerous" $\pv$ part. Nevertheless, for the nonlinear estimate of the singular part, we only need to know that the singular part is a finite $\C$-linear combination of tempered distribution on $\R^4$ of a certain form. This form is given by 
\begin{align}\label{gensingpart}
    (k,\ell,m,n) \mapsto \overline{a_{\e_0}(k)}a_{\e_1}(\ell)\overline{a_{\e_2}(m)}a_{\e_3}(n) \ \mathrm{b}(\e_0k + \e_1\ell + \e_2m + \e_3n)
\end{align}
where $(\e_0,\e_1,\e_2,\e_3) \in \{+,-\}^4$, the functions $a_{\e_j}$ and their derivatives are bounded, and $\mathrm{b}$ is a tempered distribution. Moreover, the distribution $\mathrm{b}$ takes the form 
\begin{align} \label{bform}
\xi \mapsto c\delta_0(\xi) + \pv\frac{\what{\zeta}(\xi)}{i\xi}
\end{align}
where $c$ is a constant and $\zeta$ a smooth function with compact support.
    \end{itemize}
\end{remark}

Theorem \ref{theomu} induces, in distorted frequency space, a corresponding decomposition 
of the nonlinear interactions through the Duhamel formulation \eqref{introD}. We recall the profile $f$ of the solution $u$, defined in \eqref{main1prof} as $f(t,x) := e^{-itH}u(t,x)$.

\begin{corollary}\label{duhamelfdecomp}
The following decomposition holds
\begin{align}\label{eq:expandf}
	f^{\#}(t,k)  = f^{\#}(0,k)  + i
	\int_0^{t} \big(\mathcal{N}_{0} + \mathcal{N}_{\pv} +\mathcal{N}_{L}
	+ \mathcal{N}_{R,1} + \mathcal{N}_{R,2} \big)\,ds,
\end{align}
where, for $\ast\in\{0; \pv; L; R,\!1; R,\!2\}$, 
\begin{align}\label{Nast}
	\begin{split} 
	\mathcal{N}_{\ast}(s,k) := \underset{\R^3}{\iiint} e^{is(-k^2+\ell^2-m^2+n^2)}
	f^{\#}(s,\ell)\overline{f^{\#}(s,m)}
	f^{\#}(s,n) \ \mu^{\#}_{\ast}(k,\ell,m,n)\,\frac{dn}{2\pi}\frac{dm}{2\pi}\frac{d\ell}{2\pi} \ .
	\end{split}
	\end{align}   
\end{corollary}
\begin{proof}
    From the Duhamel formulation in frequency space \eqref{introD}, we have 
    \begin{align*}
         f^\sharp(t,k) = f^\sharp(0,k) + i\int_0^t \!\! \underset{\R^3}{\iiint} e^{is(-k^2+\ell^2-m^2+n^2)}
f^{\#}(s,\ell)\overline{f^{\#}(s,m)}f^{\#}(s,n) \, \mu^\sharp(k,\ell,m,n) \,\frac{dn}{2\pi}\frac{dm}{2\pi}\frac{d\ell}{2\pi} .
    \end{align*}
    The result follows directly from decomposition of the NSD $\mu^\sharp$ provided by theorem \ref{theomu}.
\end{proof}

\subsection{Proof of theorem \ref{theomu}}
\label{subsec:DecM}
	We proceed in a few steps.
	% Recall the definitions \eqref{Ksharpdecomp} and \eqref{KsharpSdecomp}-\eqref{KsharpSpm} and Lemma \ref{lem:Ksing}.
	% Note that $\mathcal{K}_{\pm}^{\#}(x,k)$ vanish at $k=0$, and are a linear combination of regular exponentials with 
	% coefficients in the set $A \smallsetminus \{1,-1\}$.
	
	\medskip
	\noindent
	{\it Step 1: Removing $\mu^\#_{R,1}$ and the delta functions and dangerous $\pv$ terms}.
	First, we remove $\mu^\#_{R,1}$ 
	as defined in \eqref{muR1}, from $\mu^\#$. We are left with
	\begin{align}\label{prmuS1}
	\begin{split}
	 \mu^\#({\bf k}) - \mu^\#_{R,1}({\bf k}) & = \int_{\R} \prod_\ast \mathcal{K}^\#_S(x,k_j) \, dx
	\\
	& = \sum_{(\iota_0,\iota_1,\iota_2,\iota_3) \in \{0,+,-\}^4} \int_{\R} \prod_\ast \chi_{\iota_j}(x)\mathcal{K}^\#_{\iota_j}(x,k_j) \, dx.
	\end{split}
	\end{align}
	To prove the result, it will suffice to show that the terms in \eqref{prmuS1} are 
	either of the form \eqref{mu0} , \eqref{mupv}, \eqref{muL} or of the form \eqref{muR2}.
	
	Now, we consider the leading order term in  \eqref{prmuS1}, 
	that is the one with $(\iota_0,\iota_1,\iota_2,\iota_3)=(0,0,0,0)$. Recalling the convention \eqref{convprod} and the convention $\chi_0(x) \equiv 1$, this leading order term
	is given by 
	\begin{align}\label{prmuS0}
	\begin{split}
	 \int_{\R} \prod_\ast \mathcal{K}^\#_{0}(x,k_j) \, dx.
	\end{split}
	\end{align}
    We use the decomposition of $\K^{\#}_0(x,k)$ from \eqref{K0closedecomp} as follows    
\begin{align}\label{Ksharp0dec}
	\begin{split}
    \K^{\#}_0(x,k) &= \chi_+(x)A_- (a m_-(x,0)-1)e^{-ikx} + \chi_-(x)B_+ (m_-(x,0)-1)e^{-ikx} \\
    &  + \chi_-(x)A_+ (a^{-1} m_+(x,0)-1)e^{ikx}  + \chi_+(x)B_-(m_+(x,0)-1)e^{ikx}\\
    & + \chi_+(x)A_- e^{-ikx} + \chi_-(x)B_+ e^{-ikx} +  \chi_-(x)A_+ e^{ikx}  + \chi_+(x)B_-e^{ikx}.
    %\\
    %&= \chi_+(x)A_- (a m_-(x,0)-1)e^{-ikx} + \chi_-(x)B_+ (m_-(x,0)-1)e^{-ikx} \\
    %&  + \chi_-(x)A_+ (a^{-1} m_+(x,0)-1)e^{ikx}  + \chi_+(x)B_-(m_+(x,0)-1)e^{ikx}
	\end{split}
	\end{align}
     From Lemma \ref{m0asymp}, $m_+(x,0)$ converges rapidly to $ 1$ and $a$ as $x \rightarrow +\infty$ and $x\rightarrow-\infty$ respectively, and  $m_-(x,0)$ converges rapidly to $ 1$ and $a^{-1}$ as $x \rightarrow - \infty$ and $x+\rightarrow \infty$ respectively. \\
     Therefore, up to terms that go in \eqref{muR2}, we have the contribution
     \begin{align*}
        \int_\R\prod_\ast ( g(x)e^{-ik_jx} + f(x)e^{ik_jx} ) \ dx
     \end{align*}
     where $g(x) : = \chi_+(x)A_- + \chi_-(x)B_+$ and $f(x) : = \chi_-(x)A_+ + \chi_+(x)B_-$. And after computations, we get :
     \begin{align*}
       \int_\R \Big\{ & g(x)^4e^{ix(k_0 - k_1 +k_2 -k_3)} + f(x)^4e^{-ix(k_0 - k_1 +k_2 -k_3)} \\
        & + g(x)^2f(x)^2 \Big[ e^{ix(k_0 - k_1 -k_2 +k_3)} + e^{ix(k_0 + k_1 +k_2 +k_3)} + e^{ix(k_0 + k_1 -k_2 -k_3)} \\
        &\qquad \qquad \qquad \qquad + e^{ix(-k_0 - k_1 -k_2 -k_3)} + e^{ix(-k_0 + k_1 + k_2 -k_3)} + e^{ix(-k_0 - k_1 +k_2 +k_3)}\Big] \\
        & + g(x)^3f(x) \Big[ e^{ix(k_0 - k_1 +k_2 +k_3)} + e^{ix(k_0 - k_1 -k_2 -k_3)} \\
        & \qquad \qquad \qquad \qquad+ e^{ix(k_0 + k_1 +k_2 -k_3)}+ e^{ix(-k_0 - k_1 +k_2 -k_3)}\Big] \\
        & +g(x)f(x)^3 \Big[ e^{-ix(k_0 - k_1 +k_2 +k_3)} + e^{-ix(k_0 - k_1 -k_2 -k_3)} \\
        & \qquad \qquad \qquad \qquad+ e^{-ix(k_0 + k_1 +k_2 -k_3)}+ e^{-ix(-k_0 - k_1 +k_2 -k_3)}\Big] \Big\} \ dx
     \end{align*}

     We perform the analysis of the terms $g(x)^mf(x)^n$ 
     % for $n,m$ even number, we get up to localized terms (that will contribute immediately to \eqref{muR2}) that the potential "dangerous terms" are :
     \medskip
     \begin{itemize}
         \item For $(m,n) = (4,0)$, we get, up to localized terms that go in \eqref{muR2},  
         \begin{align*}
             \int_\R (A_-^4\chi_+(x)^4 +B_+^4 \chi_-(x)^4)e^{ix(k_0 - k_1 +k_2 -k_3)} \ dx.
         \end{align*}
     We use the Fourier transform formulas \eqref{chi+-} and we have cancellation of the $\pv$ terms due to the fact that $A_-^2 = B_+^2$ \eqref{eq:ABrela3}. Thus, up to terms in \eqref{muR2}, we get
         \begin{align} \label{delta40}
             2\pi A_-^4 \delta_0(k_0 - k_1 + k_2 -k_3) .
         \end{align} 
         
         \item For $(m,n) = (0,4)$, we get, up to localized terms that go in \eqref{muR2}, 
          \begin{align*}
             \int_\R (A_+^4\chi_-(x)^4 +B_-^4 \chi_+(x)^4)e^{-ix(k_0 - k_1 +k_2 -k_3)} \ dx
         \end{align*}
         We use the Fourier transform formulas \eqref{chi+-} and we have cancellation of the $\pv$ terms due to the fact that $A_+^2 = B_-^2$ \eqref{eq:ABrela3}. Thus, up to terms in \eqref{muR2}, we get
         \begin{align} \label{delta04}
              2\pi A_+^4 \delta_0(k_0 - k_1 + k_2 -k_3)
         \end{align}
         since $\delta_0$ is even.
         \medskip
         \item For $(m,n) = (2,2)$, we get, up to localized terms that go in \eqref{muR2},
         \begin{align*}
             \int_\R (A_-^2B_-^2\chi_+(x)^4 +A_+^2B_+^2 \chi_-(x)^4)e^{ix \eps \cdot {\bf k }} \ dx.
         \end{align*}
         We proceed as above by observing that $A_-^2B_-^2 = A_+^2B_+^2$. By using the evenness of $\delta_0$, we get, up to terms in \eqref{muR2} 
         \begin{align} \label{delta22}
             4\pi A_-^2A_+^2 \Big[ \delta_0(k_0-k_1-k_2+k_3) +\delta_0(k_0+k_1+k_2+k_3) + \delta_0(k_0+k_1-k_2-k_3)\Big].
         \end{align}

         \medskip
         \item For $(m,n) = (3,1)$ and $(1,3)$, we get, up to  terms \eqref{muR2},
          \begin{align*}
             \int_\R (A_-^3B_-\chi_+(x)^4  + A_+B_+^3 \chi_-(x)^4)e^{ix \eps \cdot {\bf k }} \ dx + \int_\R (A_-B_-^3\chi_+(x)^4  + A_+^3B_+ \chi_-(x)^4)e^{-ix \eps \cdot {\bf k }} \ dx.
         \end{align*}
         We can write 
         \begin{align*}
            & A_-B_- \Big[ A_-^2\int_\R \chi_+(x)^4 e^{ix\eps \cdot {\bf k}}dx + B_-^2\int_\R \chi_+(x)^4 e^{-ix\eps \cdot {\bf k}}dx\Big]  \\
            &+ A_+B_+ \Big[ B_+^2\int_\R \chi_-(x)^4 e^{ix\eps \cdot {\bf k}}dx + A_+^2\int_\R \chi_-(x)^4 e^{-ix\eps \cdot {\bf k}}dx\Big].
         \end{align*}
         Since $A_+B_+ + A_-B_- = 0 $ \eqref{eq:ABrela1}, it suffices to consider, up to the multiplicative constant $A_+B_+$, 
         \begin{align*}
             &- A_-^2\int_\R \chi_+(x)^4 e^{ix\eps \cdot {\bf k}}dx - B_-^2\int_\R \chi_+(x)^4 e^{-ix\eps \cdot {\bf k}}dx \\
             & + B_+^2\int_\R \chi_-(x)^4 e^{ix\eps \cdot {\bf k}}dx + A_+^2\int_\R \chi_-(x)^4 e^{-ix\eps \cdot {\bf k}}dx.
         \end{align*}
     We use the Fourier transform formulas \eqref{chi+-} to get, up to terms in \eqref{muR2}, 
     \begin{align}\label{pv31-13}
         -2(A_+^2 - A_-^2)\pv \frac{\what{\zeta}(\eps \cdot {\bf k})}{i\eps \cdot { \bf k}}.
     \end{align}
     \end{itemize}
    At the end, we obtain \ref{mu0} as the contribution of \eqref{delta40}, \eqref{delta04} and \eqref{delta22}. And \eqref{pv31-13} \eqref{mupv} as announced in the theorem.

	% Next, we look at the contributions from \eqref{prmuS1} that give rise 
	% to the improved low frequency singular distribution $\mu^\#_L$.
	
	\medskip
	\noindent
	{\it Step 2: The singular part $\mu^\#_L$}.
    This part of the distribution arises from the contributions to \eqref{prmuS1} 
	that have no terms with both $+$ and $-$ in the sum,
	that is, the term
	\begin{align}
	\label{prmu+a}
	& \sum_{(\iota_0,\iota_1,\iota_2,\iota_3) \in \{0,+\}^4 \smallsetminus \{0\}^4} 
	\int_{\R} \prod_\ast \chi_{\iota_j}(x)\mathcal{K}^\#_{\iota_j}(x,k_j) \, dx
	\\
	\label{prmu+b}
	& + \sum_{(\iota_0,\iota_1,\iota_2,\iota_3) \in \{0,-\}^4 \smallsetminus \{0\}^4} 
	\int_{\R} \prod_\ast \chi_{\iota_j}(x)\mathcal{K}^\#_{\iota_j}(x,k_j) \, dx.
	\end{align}

	We look at the first sum above \eqref{prmu+a}. 
	By symmetry in the $k_j$ variables (up to irrelevant conjugations),
	it suffices to analyze four types of terms in the sum in \eqref{prmu+a}, 
	depending on how many indexes $\iota_j = 0$ appear; these terms are (recall the convention $\chi_0\equiv1$)
	\begin{subequations}\label{prmu+1}
		\begin{align}
		\label{prmu+11}
		& \int_{\R} \chi_+(x)  \overline{\mathcal{K}^\#_0(x,k_0)} \mathcal{K}^\#_0(x,k_1) 
		\overline{\mathcal{K}^\#_0(x,k_2)} \mathcal{K}^\#_+(x,k_3) \, dx,
		\\
		\label{prmu+12}
		& \int_{\R} \chi_+^2(x)  \overline{\mathcal{K}^\#_0(x,k_0)} \mathcal{K}^\#_0(x,k_1) 
		\overline{\mathcal{K}^\#_+(x,k_2)} \mathcal{K}^\#_+(x,k_3) \, dx,
		\\
		\label{prmu+13}
		& \int_{\R} \chi_+^3(x) \overline{\mathcal{K}^\#_0(x,k_0)} \mathcal{K}^\#_+(x,k_1) 
		\overline{\mathcal{K}^\#_+(x,k_2)} \mathcal{K}^\#_+(x,k_3) \, dx,
		\\
		\label{prmu+14}
		& \int_{\R} \chi_+^4(x) \prod_\ast \mathcal{K}^\#_{+}(x,k_j) \, dx.
		\end{align}
	\end{subequations}
 Recalling the refined decomposition \eqref{K0closedecomp} of $\mathcal{K}^\#_0(x,k)$, we perform the computation of \begin{align}
	    \overline{\mathcal{K}^\#_0(x,k_0)} \mathcal{K}^\#_0(x,k_1) 
		\overline{\mathcal{K}^\#_0(x,k_2)} \quad \text{ , } \quad \overline{\mathcal{K}^\#_0(x,k_0)} \mathcal{K}^\#_0(x,k_1) \quad \text{ and } \quad \overline{\mathcal{K}^\#_0(x,k_0)}.
	\end{align}
In fact, these terms are $\C-$linear combination of  elements of the form
\begin{align} \label{contribut}
        \varphi_{\epsilon}(x)e^{i\epsilon \cdot\bf k} + C_{\eta_+,r}\chi_+^r(x)e^{i\eta_+\cdot\bf k} +  C_{\eta_-,r}\chi_-^r(x)e^{i\eta_-\cdot\bf k}
    \end{align}
    where $\epsilon = (\epsilon_0, \epsilon_1, \epsilon_2, \epsilon_3) \in \{0,+,-\}^4 $, $\eta_\pm = (\eta_{0 \pm}, \eta_{1 \pm}, \eta_{2 \pm}, \eta_{3 \pm}) \in \{0,+,-\}^4$, $r= 1,2,3$, $C_{\eta_\pm,r}$ are constants and $\varphi_\epsilon(x)$ are localized smooth functions. Furthermore, the functions $\varphi_\epsilon(x)$ either contain the factor $\chi_+(x)\chi_-(x)$, which is smooth and compactly supported, or they involve products of the form $\chi_{\pm}(x)\big(m_\pm(x,0) - 1\big)$ and $\chi_{\pm}(x)\big(a^{\pm 1}m_\mp(x,0) - 1\big)$. The decay properties of the latter are provided by Lemma \ref{m0asymp}. 
    Now, we recall from the definition \eqref{KsharpSpm} that $\mathcal{K}^\#_+(x,k)$ comes with a coefficient that vanishes at $k=0$. Thus, applying Fourier formulas \eqref{chi+-} combined with the decomposition \eqref{contribut}. We get that the tempered distributions \eqref{prmu+11}, \eqref{prmu+12}, \eqref{prmu+13} and \eqref{prmu+14} are sums of terms in \eqref{muL} and \eqref{muR2}.

	\medskip
	\noindent
	{\it Step 3: The regular part $\mu_{R,2}^\sharp$}. From the analysis above, up to terms in \eqref{muR2}, we are left with the terms in \eqref{prmuS1}, where the sum is taken over quadruples 
	$(\iota_0,\iota_1,\iota_2,\iota_3)$
	that contain at least one $+$ and one $-$ sign. Thus,
 %    , which we denote by 
	% \begin{align}
	% \mathcal{I} := \big\{ (\iota_0,\iota_1,\iota_2,\iota_3) \in \{+,-,0\}^4 \, \mbox{s.t.} \, \exists \, a,b \in \{0,1,2,3\} \,\, \mbox{with} 
	% \,\, \iota_a = +, \iota_b = - \big\}.
	% \end{align}
	up to permuting variables (and conjugating), we can reduce matters to the terms where the $+$ and $-$ indexes correspond to the first two 
	kernels. %$\mathcal{K}^\#_{\iota_0}(k_0,x)$ and  $\mathcal{K}^\#_{\iota_1}$, 
	That is, the sum
	\begin{align}\label{prmuR0}
	\begin{split}
	& \sum_{(\iota_2,\iota_3) \in \{+,-,0\} } \int_{\R} \overline{\chi_{+}(x)\mathcal{K}^\#_{+}(x,k_0)}  
	\, \chi_{-}(x)\mathcal{K}^\#_{-}(x,k_1) 
	\, \overline{\chi_{\iota_2}(x)\mathcal{K}^\#_{\iota_2}(x,k_2)}  \, \chi_{\iota_3}(x)\mathcal{K}^\#_{\iota_3}(x,k_3) \, dx.
	\end{split}
	\end{align}
	Using \eqref{KsharpSpm}, we can write
	\begin{align}\label{prmuR1}
	\begin{split}
	\eqref{prmuR0} & = \sum_{ \substack{\iota_2,\iota_3 \in \{+,-,0\} \\ \eps_0,\eps_1 \in \{+,-\} } }   \overline{a_+^{\eps_0}(k_0)}  a_-^{\eps_1}(k_1)
	%\overline{a_{\iota_2}^{\eps_2}(k_2)}  \, a_{\iota_3}^{\eps_3}(k_3)
	\\ 
	& \times \int_{\R} (\chi_+\chi_-)(x) \, 
	%\overline{\mathcal{K}^\#_{+}(k_0,x)}  \mathcal{K}^\#_{-}(k_1,x) 
	%e^{ix(-\eps_0k_0+\eps_1k_1-\eps_2k_2+\eps_3k_3)} 
	e^{ix(-\eps_0k_0+\eps_1k_1)} 
	\overline{\chi_{\iota_2}(x)\mathcal{K}^\#_{\iota_2}(x,k_2)}  
	\, \chi_{\iota_3}(x)\mathcal{K}^\#_{\iota_3}(x,k_3) \, dx.
	\end{split}
	\end{align}
	Since $\chi_+\chi_-$ is compactly supported,
	the expression above is of the form \eqref{muR2} in view of the definitions 
	\eqref{KsharpS0}, \eqref{KsharpSpm}, and the definition \eqref{defSet}. 
 $\hfill \Box$
%-------------------------------------------------------
\bigskip
\section{\texorpdfstring{$L^2$}{L2} estimates for the singular part}\label{secnonlinl2sing} 
In this section, we prove the $L^2$ estimates involving the terms $\mathcal{N}_\ast$ for $\ast \in \{ 0, L, \pv\}$ \eqref{Nast}. The estimates follow from multilinear estimates (see Corollary \ref{cor:estimTbeps}), the inhomogeneous smoothing estimates of Lemma \ref{lem:smoothingsim} and a Fourier restriction-type inequality (see Lemma \ref{k^2lemma}). The latter is the main novelty in the analysis of singular parts especially in the absence of low-frequency improvement. We start by defining the following trilinear form.

% We start with $L^2$ estimates of the singular part. Here are lemmas that will be of use for the estimates.

% This section is devoted to the $L^2$ estimates, specifically to proving the bound \eqref{dkf bound} in Proposition \ref{propgoal}. For that sake, we first focus on the singular part.

\begin{definition}
    Let $\mathrm{b}$ be a tempered distribution on $\R$ and $ \eps = (\e_0,\e_1,\e_2,\e_3) \in \{ +,-\}^4$. For $\mathcal{S}(\R)$ functions $f_1,f_2,f_3$, we define the following trilinear form 
		\begin{align}\label{eq:tril}
		\begin{split}
		\mathcal{T}_{\mathrm{b},\eps}(f_1,f_2,f_3)(t,k) :=
		\underset{\R^3}{\iiint} & e^{it(-k^2+\ell^2-m^2+n^2)}f_{1}(\ell) \overline{f_{2}(m)} f_{3}(n)
		 \\
         & \times \mathrm{b}(\e_0k+\e_1\ell+\e_2m+\e_3n)\,\frac{dn}{2\pi}\frac{dm}{2\pi}\frac{d\ell}{2\pi} .
		\end{split} 
		\end{align}
\end{definition}
This trilinear form appears naturally in the expression of $\mathcal{N}_{\ast}$ due to the form of the $\mu_\ast^\sharp$  for $\ast \in \{ 0, L, \pv\}$.

\medskip
\begin{lemma}\label{lem:inverseFD}
		Let $\mathrm{b}$ be a tempered distribution on $\R$ and $ \eps = (\e_0,\e_1,\e_2,\e_3) \in \{ +,-\}^4$. Let $f_1,f_2,f_3$ be functions in $\mathcal{S}(\R)$.   Then, the following identity holds
		\begin{align}\label{eq:inverseFD}
		\begin{split}
		 \widehat{\mathcal{F}}^{-1} &\big[k \mapsto e^{itk^2} {\mathcal{T}}_{\mathrm{b},\e_0,\e_1,\e_2,\e_3}(f_1,f_2,f_3)(t,k)\big](x) \\
		&=  \e_0u_{1}(t,-\e_1x)\overline{u_{2}(t,\e_2x)}u_{3}(t,-\e_3x)\widehat{\mathcal{F}}^{-1}\left[\mathrm{b}\right](\eps_0x),
		\end{split}
		\end{align}
         where $u_{j}:=e^{-it\partial_{xx}}\widehat{\mathcal{F}}^{-1}(f_j)$.
	\end{lemma}
    \begin{proof}
    First, we can write 
    \begin{align} \label{invFDcomput}
    \begin{split}
        {\mathcal{T}}_{\mathrm{b},\e_0,\e_1,\e_2,\e_3}(f_1,f_2,f_3)(t,k) = e^{-itk^2} \underset{\R^3}{\iiint} &\big(e^{it\ell^2} f_1(\ell) \big) \overline{\big(e^{itm^2} f_2(m)\big)} \big(e^{itn^2} f_3(n)\big) \\
        &\times \ \mathrm{b}(\e_0k+\e_1\ell+\e_2m+\e_3n)\,\frac{dn}{2\pi}\frac{dm}{2\pi}\frac{d\ell}{2\pi} .
         \end{split}
    \end{align}
    Then by observing that $e^{it\xi^2}f_j(\xi) = \what{u_j}(\xi)$, we can write \eqref{invFDcomput} as the following convolution
    \begin{align}
        e^{itk^2}{\mathcal{T}}_{\mathrm{b},\e_0,\e_1,\e_2,\e_3}(f_1,f_2,f_3)(t,k) = \e_1\e_2\e_3 \big(\what{u_1}(-\e_1\cdot) \ast \overline{\what{u_2}(-\e_2\cdot)} \ast \what{u_3}(-\e_3\cdot) \ast \mathrm{b}\big)(\e_0k). \nonumber
    \end{align}
       The identity \eqref{eq:inverseFD} follows from the fact that the (flat) Fourier transform changes convolution in pointwise product. 
    \end{proof}
\medskip
\begin{corollary} \label{cor:estimTbeps}
         Let $\mathrm{b}$ be a tempered distribution on $\R$ and $ \eps = (\e_0,\e_1,\e_2,\e_3) \in \{ +,-\}^4$. Let $f_1,f_2,f_3$ be functions in $\mathcal{S}(\R)$. If $\widehat{\mathcal{F}}^{-1}\left[\mathrm{b}\right] \in L^\infty$, then the following estimates hold
\begin{align*}
\begin{split}
     \| {\mathcal{T}}_{\mathrm{b},\eps}(f_1,f_2,f_3)(t,\cdot)\|_{L^2} &\lesssim \|u_1(t,\cdot)\|_{L^\infty}\|u_2(t,\cdot)\|_{L^\infty} \|f_3\|_{L^2},
\end{split}
    \\
    \begin{split}
         \| {\mathcal{T}}_{\mathrm{b},\eps}(f_1,f_2,f_3)(t,\cdot)\|_{L^2} &\lesssim \|u_1(t,\cdot)\|_{L^\infty}\|f_2\|_{L^2} \|u_3(t,\cdot)\|_{L^\infty}, 
    \end{split}
   \\
   \begin{split}
        \| {\mathcal{T}}_{\mathrm{b},\eps}(f_1,f_2,f_3)(t,\cdot)\|_{L^2} &\lesssim \|f_1\|_{L^2}\|u_2(t,\cdot)\|_{L^\infty}\|u_3(t,\cdot)\|_{L^\infty} ,
   \end{split}
\end{align*}
where $u_{j}:=e^{-it\partial_{xx}}\widehat{\mathcal{F}}^{-1}(f_j)$.
    \end{corollary}
    \begin{proof} We focus on the first estimate, as the remaining ones follow by similar arguments. By (flat) Plancherel, we have 
        \begin{align*}
         \| {\mathcal{T}}_{\mathrm{b},\eps}(f_1,f_2,f_3)(t,\cdot)\|_{L^2} =   \| \widehat{\mathcal{F}}^{-1} \big[k \mapsto e^{itk^2} {\mathcal{T}}_{\mathrm{b},\e_0,\e_1,\e_2,\e_3}(f_1,f_2,f_3)(t,k)\big]\|_{L^2} .
        \end{align*}
        By using Lemma \ref{lem:inverseFD}, we can write  
        \begin{align*}
           \| {\mathcal{T}}_{\mathrm{b},\eps}(f_1,f_2,f_3)(t,\cdot)\|_{L^2} \lesssim \|u_1(t,\cdot)\|_{L^\infty}\|u_2(t,\cdot)\|_{L^\infty} \|u_3(t,\cdot)\|_{L^2} \| \widehat{\mathcal{F}}^{-1}\left[\mathrm{b}\right] \|_{L^\infty}.
        \end{align*}
        Using the fact that the (flat) Schrödinger propagator $e^{-it\partial_{xx}}$ is unitary on $L^2$ and (flat) Plancherel, we get $\|u_3\|_{L^2} = \|f_3\|_{L^2}$. Invoking the $L^\infty$ boundedness of $\widehat{\mathcal{F}}^{-1}\left[\mathrm{b}\right]$, the estimate follows.
    \end{proof}

\medskip
\begin{lemma}\label{lem:algetri} Let $\mathrm{b}$ be a tempered distribution on $\R$ and $ \eps = (\e_0,\e_1,\e_2,\e_3) \in \{ +,-\}^4$. Let $f_1,f_2,f_3$ be functions in $\mathcal{S}(\R)$. Then, the following identity holds
		\begin{align}\label{eq:kdiffT}
		\begin{split}
		 \e_0 \partial_{k}\mathcal{T}_{\mathrm{b},\e_0,\e_1,\e_2,\e_3}&(f_1,f_2,f_3)(t,k) 
		   \\
          =&-\e_1\mathcal{T}_{\mathrm{b},\e_0,\e_1,\e_2,\e_3}
		\left(\partial_{\ell}f_{1},f_{2},f_{3}\right)(t,k)+\e_2\mathcal{T}_{\mathrm{b},\e_0,\e_1,\e_2,\e_3}\left(f_{1},\partial_{m}f_{2},f_{3}\right)(t,k)
		\\
		&  -\e_3\mathcal{T}_{\mathrm{b},\e_0,\e_1,\e_2,\e_3}\left(f_{1},f_{2},\partial_{n}f_{3}\right)(t,k)
		 -2it \, \mathcal{T}_{\mathrm{\xi b},\e_0,\e_1,\e_2,\e_3}(f_1,f_2,f_3)(t,k).
		\end{split}
		\end{align}
	\end{lemma}
	
	\begin{proof}
		See Chen-Pusateri \cite[Lemma  4.3]{NLSV}.
	\end{proof}
    \medskip
    In addition to the above lemmas, the following Fourier restriction-type lemma will be of use.
    \begin{lemma} \label{k^2lemma}
    Let $F$ be a function in $\mathcal{S}(\R)$. Then, for $1 <p,q,r < \infty$ such that $\frac{1}{q} + \frac{1}{p} =1$ and $\frac{1}{p} = \frac{1}{r} - \frac{1}{2}$, the following estimate holds 
    \begin{align}
        \int_\R |\widehat{F}(k^2)|^2 \ dk \underset{p}{\lesssim} ||F||_{L^q} ||F||_{L^r} \ .
    \end{align}
    In particular, 
    \begin{align}\label{k^2lem43}
        \Big( \int_\R |\widehat{F}(k^2)|^2 \ dk \Big)^{1/2} \lesssim ||F||_{L^{4/3}} \ .
    \end{align}
    \end{lemma}
    \begin{proof} By the change of variables $\lambda = k^2$, we obtain
    \begin{align}
        I:=\int_\R |\widehat{F}(k^2)|^2\,dk
        = \frac12 \int_0^{+\infty} |\widehat{F}(\lambda)|^2\,\lambda^{-1/2}\,d\lambda
        = \frac12 \int_\R |G(\lambda)|^2 |\lambda|^{-1/2}\,d\lambda,
    \end{align}
    where $G(\lambda):=\mathbf 1_{[0,+\infty)}(\lambda)\widehat F(\lambda)$. 
    
     We know the following Fourier transform formula in $\R^n$ :
    \begin{align}
        \widehat{\frac{1}{|k|^\alpha}}(\xi) = c_{n,\alpha} \frac{1}{|\xi|^{n-\alpha}}, \quad \quad \quad \text{ for } \alpha \in (0,n).
    \end{align}
    In our case, we take $n=1$ and $\alpha =1/2$.

    Since the Fourier transform of $|\lambda|^{-1/2}$ is a constant multiple of $|\xi|^{-1/2}$ in dimension one, Plancherel gives
    \begin{align}
        I
        &\lesssim \int_\R \widehat{|G|^2}(\xi)|\xi|^{-1/2}\,d\xi
        = \int_\R (\widehat G * \widehat{\overline G})(\xi)|\xi|^{-1/2}\,d\xi.
    \end{align}
    Therefore, by H\"older's inequality with $\frac1q+\frac1p=1$,
    \begin{align}
        I
        &\lesssim \|\widehat G\|_{L^q}\,\big\||\xi|^{-1/2} * \widehat{\overline G}\big\|_{L^p}.
    \end{align}
    Since $\frac1p=\frac1r-\frac12$, the Hardy--Littlewood--Sobolev inequality yields
    \begin{align}
        I \lesssim \|\widehat G\|_{L^q}\,\|\widehat G\|_{L^r}.
    \end{align}
    It remains to control $\widehat G$. Using
    \begin{align}
        \widehat{\mathbf 1_{[0,+\infty)}}(\xi)=\pi\delta_0(\xi)+\pv\frac{1}{i\xi},
    \end{align}
    we have, up to harmless constants,
    \begin{align}
        \widehat G
        = \widehat{\mathbf 1_{[0,+\infty)}} * \widehat{\widehat F}
        = \widehat{\mathbf 1_{[0,+\infty)}} * F(-\cdot).
    \end{align}
    Hence for every $1<s<\infty$,
    \begin{align}
        \|\widehat G\|_{L^s}
        &\lesssim \|F\|_{L^s}+\Big\|\pv\frac{1}{\xi}*F\Big\|_{L^s}
        \lesssim \|F\|_{L^s},
    \end{align}
    by the $L^s$ boundedness of the Hilbert transform. Taking $s=q$ and $s=r$ proves
    \begin{align}
        \int_\R |\widehat F(k^2)|^2\,dk \lesssim \|F\|_{L^q}\|F\|_{L^r}.
    \end{align}
We get \eqref{k^2lem43}, by taking $q=r=4/3$.
    \end{proof}
\medskip
The sequel of this section will be devoted to prove the following statements.
\begin{proposition} \label{l2boundsing}
Under the bootstrap hypothesis \eqref{boothyp}, it holds that for every $|t| \le T$, 
    \begin{align}\begin{split}
		\label{weightmain0}
		 {\Big\| \int_0^t \partial_k \mathcal{N}_0(s,k) \, ds \Big\|}_{L^2_k} &\lesssim C_0^3\varepsilon^3 \jt^{1/4} ,
        \end{split}
		\\
        \begin{split}
		\label{weightmain1}
		 {\Big\| \int_0^t \partial_k \mathcal{N}_L(s,k) \, ds  \Big\|}_{L^2_k} &\lesssim C_0^3\varepsilon^3 \jt^{1/4} ,
         \end{split}
		\\
        \begin{split}
        \label{weightmainpv}
		 {\Big\| \int_0^t \partial_k \mathcal{N}_{\pv}(s,k) \, ds  \Big\|}_{L^2_k} &\lesssim C_0^3\varepsilon^3 \jt^{1/4} .
        \end{split}
        \end{align}
\end{proposition} 

\begin{proof}For $\ast \in \{0,L,\pv\}$, We invoke the form of $\mathcal{N}_{\ast}$ in \eqref{Nast}.  We invoke the form of $\mu_\ast$ given in \eqref{gensingpart}. 
% \begin{align*}
%     \mathcal{N}_{\ast}(s,k) := \underset{\R^3}{\iiint} e^{is(-k^2+\ell^2-m^2+n^2)}
% 	f^{\#}(s,\ell)\overline{f^{\#}(s,m)}
% 	f^{\#}(s,n) \ \mu^{\#}_{\ast}(k,\ell,m,n)\,\frac{dn}{2\pi}\frac{dm}{2\pi}\frac{d\ell}{2\pi}.
% \end{align*}
% Theorem \ref{theomu} tells us that $\mu^{\#}_{\ast}$ are finite $\C$-linear combination of tempered distribution on $\R^4$ of a certain type. It takes the form
% \begin{align*}
%     (k,\ell,m,n) \mapsto \overline{a_{\e_0}(k)}a_{\e_1}(\ell)\overline{a_{\e_2}(m)}a_{\e_3}(n) \ \mathrm{b}(\e_0k + \e_1\ell + \e_2m + \e_3n)
% \end{align*}
% where $\e = (\e_0,\e_1,\e_2,\e_3) \in \{+,-\}^4$, the functions $a_{\e_j}$ and their derivatives are bounded in $L^\infty$, and $\mathrm{b}$ is a tempered distribution. Moreover, the distribution $\mathrm{b}$ takes the form 
% \begin{align} \label{bform}
% \xi \mapsto c\delta_0(\xi) + \pv\frac{\what{\zeta}(\xi)}{i\xi}
% \end{align}
% where $c$ is a constant and $\zeta$ a smooth function with compact support. 
Thus, to prove the proposition, it suffices to estimate
\begin{align*}
    \Big\| \int_0^t \partial_k \Big(\overline{a_{\e_0}(k)}{\mathcal{T}}_{\mathrm{b},\eps}(a_{\e_1}f^\sharp,a_{\e_2}f^\sharp,a_{\e_3}f^\sharp)(s,k)\Big) \ ds \Big\|_{L^2_k} 
\end{align*}
where $\mathrm{b}$ takes the form \eqref{bform}. Since $a_{\e_0}$ and $\partial_ka_{\e_0}$ are bounded, it suffices to estimate 
\begin{align} \label{noderiv}
    &\Big\| \int_0^t {\mathcal{T}}_{\mathrm{b},\eps}(a_{\e_1}f^\sharp,a_{\e_2}f^\sharp,a_{\e_3}f^\sharp)(s,k) \ ds \Big\|_{L^2_k}, \\
    \label{withderiv}
     &\Big\| \int_0^t \partial_k {\mathcal{T}}_{\mathrm{b},\eps}(a_{\e_1}f^\sharp,a_{\e_2}f^\sharp,a_{\e_3}f^\sharp)(s,k) \ ds \Big\|_{L^2_k}.
\end{align}
For $j\in\{1,2,3\}$, we set $f_j := a_{\e_j}f^\sharp$ and $u_j(t):= e^{-it\partial_{xx}}\what{\Fou}^{-1}(f_j)$. The tempered distribution $\mathrm{b}$ satisfies $\what{\Fou}^{-1}[\mathrm{b}] \in L^\infty$. Indeed, 
$$ \what{\Fou}^{-1}[\mathrm{b}](x) = c + \frac{1}{2}(\sgn \ast \zeta)(x) $$
which is bounded since $\zeta \in L^1$. Then by Corollary \ref{cor:estimTbeps}, we have 
\begin{align}\label{Tbf1f2f3 1}
    \Big\|{\mathcal{T}}_{\mathrm{b},\eps}(f_1,f_2,f_3)(s,k) \Big\|_{L^2_k} \lesssim \| u_1(s)\|_{L^\infty_x}\| u_2(s)\|_{L^\infty_x}\| f_3(s)\|_{L^2_k}.
\end{align}
We observe that 
\begin{align*}
    u_j(s,x)= e^{-is\partial_{xx}}\what{\Fou}^{-1}(f_j)(x) = \what{\Fou}^{-1}\Big[ k \mapsto e^{isk^2} a_{\e_j}(k)f^\sharp(s,k)\Big](x) 
\end{align*}
Thus, from Lemma \ref{lemstat}, we get 
\begin{align*}
    \|  u_j(s) \|_{L^\infty} &\lesssim \frac{1}{\js^{1/2}} \Big(  \| a_{\e_j}f^\sharp(s)\|_{L^\infty_k} + \frac{1}{\js^{1/4}} \| \partial_k(a_{\e_j}f^\sharp(s))\|_{L^2_k} \Big).
\end{align*}
Since $a_{\e_j}$ and $\partial_ka_{\e_j}$ are bounded,
\begin{align*}
    \|  u_j(s) \|_{L^\infty}& \lesssim \frac{1}{\js^{1/2}} \Big(  \| f^\sharp(s)\|_{L^\infty_k} + \frac{1}{\js^{1/4}} (\| f^\sharp(s)\|_{L^2_k} + \| \partial_kf^\sharp(s)\|_{L^2_k} )\Big)\ .
\end{align*}
 And using the bootstrap hypothesis  \eqref{boothyp}, and the fact that $\| f^\sharp(s)\|_{L^2_k} = \| u^\sharp(s)\|_{L^2_k} = \| u(s)\|_{L^2_x} = \| u_0\|_{L^2_x} \le \varepsilon$, we have 
\begin{align} \label{dispdecayauxill}
    \|  u_j(s) \|_{L^\infty_x} \lesssim\frac{C_0\varepsilon}{\js^{1/2}} \qquad \text{ for all } s \in [-T,T].
\end{align}
We combine \eqref{Tbf1f2f3 1} and \eqref{dispdecayauxill} to get 
\begin{align*}
    \Big\|{\mathcal{T}}_{\mathrm{b},\eps}(f_1,f_2,f_3)(s,k) \Big\|_{L^2_k} \lesssim \frac{C_0^3\varepsilon^3}{\js}.
\end{align*}
We perform triangle inequality and time integration to obtain
\begin{align*}
    \eqref{noderiv} \lesssim C_0^3\varepsilon^3\int_0^t \frac{ds}{\js} \lesssim C_0^3\varepsilon^3 \log \jt \lesssim C_0^3\varepsilon^3 \jt^{1/4}.
\end{align*}

\medskip
Next, we estimate \eqref{withderiv}. We start by invoking Lemma \ref{lem:algetri}. Thus, triangle inequality leads to 
\begin{align} \label{usealgetri}
\begin{split}
    \eqref{withderiv} &\lesssim \Big\| \int_0^t {\mathcal{T}}_{\mathrm{b},\eps}(\partial_\ell f_1,f_2,f_3)(s,k) \ ds \Big\|_{L^2_k}  + \Big\| \int_0^t {\mathcal{T}}_{\mathrm{b},\eps}(f_1,\partial_m f_2,f_3)(s,k) \ ds \Big\|_{L^2_k} \\
    & + \Big\| \int_0^t {\mathcal{T}}_{\mathrm{b},\eps}(f_1,f_2,\partial_n f_3)(s,k) \ ds \Big\|_{L^2_k}  + \Big\| \int_0^t s{\mathcal{T}}_{\xi \mathrm{b},\eps}(f_1,f_2,f_3)(s,k) \ ds \Big\|_{L^2_k} .
\end{split}
\end{align}
The first 3 terms in \eqref{usealgetri} can be treated similarly. Then, we focus on the first one. Invoking corollary \ref{cor:estimTbeps}, we get for $s \in [-T,T]$, 
\begin{align*}
     \Big\|{\mathcal{T}}_{\mathrm{b},\eps}(\partial_\ell f_1,f_2,f_3)(s,k) \Big\|_{L^2_k} \lesssim \| \partial_\ell f_1(s)\|_{L^2} \| u_2(s)\|_{L^\infty}\| u_3(s)\|_{L^\infty}.
\end{align*}
We use \eqref{dispdecayauxill} to estimate $\| u_2(s)\|_{L^\infty}$ and $\| u_3(s)\|_{L^\infty}$. Using the $L^\infty$ boundedness of $a_{\e_1}$ and $\partial_\ell a_{\e_1}$, and the bootstrap hypothesis, we have $\| \partial_\ell f_1(s)\|_{L^2} \lesssim C_0\varepsilon \js^{1/4}$. Thus, we obtain
\begin{align*}
     \Big\|{\mathcal{T}}_{\mathrm{b},\eps}(\partial_\ell f_1,f_2,f_3)(s,k) \Big\|_{L^2_k} \lesssim \frac{C_0^3\varepsilon^3}{\js^{3/4}} \ .
\end{align*}
We get the desired estimate by performing integration in time.

\medskip
We are left with the last term of \eqref{usealgetri}. We observe that $\xi\mathrm{b}(\xi) = -i\what{\zeta}(\xi)$. So, we are led to estimate the $L^2_k$ norm of 
\begin{align*}
    I(t,k) &:= \int_0^t s{\mathcal{T}}_{\what{\zeta},\eps}(f_1,f_2,f_3)(s,k) \ ds \\
     & = \int_0^t e^{-isk^2} e^{isk^2} s{\mathcal{T}}_{\what{\zeta},\eps}(f_1,f_2,f_3)(s,k) \ ds \\
     & = \int_0^t e^{-isk^2} \what{\Fou}\what{\Fou}^{-1}\Big[k \mapsto e^{isk^2} s{\mathcal{T}}_{\what{\zeta},\eps}(f_1,f_2,f_3)(s,k)\Big] \ ds \\
     & = \int_0^t e^{-isk^2}\int_\R e^{-iky} F(s,y) \ dy\ ds 
\end{align*}
where $F(s,y) :=  \what{\Fou}^{-1}\Big[k \mapsto e^{isk^2} s{\mathcal{T}}_{\what{\zeta},\eps}(f_1,f_2,f_3)(s,k)\Big](y)$. We can write
\begin{align*}
I(t,k) = \int_0^t e^{-isk^2}\int_\R (e^{-iky}-1) F(s,y) \ dy\ ds + \int_0^t e^{-isk^2}\int_\R F(s,y) \ dy\ ds. 
\end{align*}
By observing that 
$$\int_\R F(s,y) \ dy = \int_\R \what{\Fou}^{-1}\Big[k \mapsto e^{isk^2} s{\mathcal{T}}_{\what{\zeta},\eps}(f_1,f_2,f_3)(s,k)\Big](y)\  dy \ ,$$
 (flat) Fourier inversion formula gives us 
$$ \int_\R F(s,y) \ dy  = s{\mathcal{T}}_{\what{\zeta},\eps}(f_1,f_2,f_3)(s,0) \ . $$
Thus, we get 
\begin{align*}
I(t,k) & = \int_0^t e^{-isk^2}\int_\R (e^{-iky}-1) F(s,y) \ dy\ ds + \int_0^t e^{-isk^2}s{\mathcal{T}}_{\what{\zeta},\eps}(f_1,f_2,f_3)(s,0)\ ds \\
& =: I_1(t,k) \ + I_2(t,k)\ . 
\end{align*}\\
\medskip
\emph{Estimates for $I_1(t,k)$.} We consider the functions
\begin{align*}
    \phi(k)&:= \begin{cases}
        k \qquad\text{ if } |k| \le 1, \\
        1 \qquad\text{ if } |k| > 1, 
    \end{cases} \qquad , \qquad
    Q(y,k):= \begin{cases}
        \frac{e^{-iky} -1}{k} \qquad\text{ if } |k| \le 1, \\
        e^{-iky} -1 \qquad\text{ if } |k| > 1. 
    \end{cases}
\end{align*}
We can check that $e^{-iky} -1 = \phi(k)Q(y,k)$. Thus, 
$$ I_1(t,k) = \int_0^t \int_\R e^{-isk^2}\phi(k)Q(y,k)F(s,y) \ dy\ ds \ .$$
Moreover, 
$$ | \phi(k) | \le \sqrt{|k|} \quad , \quad \underset{(y,k) \in \R_y \times \R_k}{\sup} \big| \jy^{-1} Q(y,k)\big| < \infty \ .$$
Then, invoking the smoothing estimates of Lemma \ref{lem:smoothingsim}, we get 
\begin{align}
    \label{usesmoothestimI1}
    \| I_1(t,k) \|_{L^2_k} \lesssim \| \jy F(s,y) \|_{L^1_y L^2_s([0,t])} \ .
\end{align}
Using Lemma \ref{lem:inverseFD}, we write 
$F(s,y) = \e_0 su_1(s,-\e_1y)\overline{u_2(s,\e_2y)}u_3(s,-\e_3y)\zeta(\e_0y).$
Then we obtain the following estimate
\begin{align*}
    \| \jy F(s,y) \|_{L^1_y L^2_s([0,t])} \lesssim \| \jy \zeta(y) \|_{L^1_y} \Big\| s\|u_1(s,y)\|_{L^\infty_y}\|u_2(s,y)\|_{L^\infty_y}\|u_3(s,y)\|_{L^\infty_y} \Big\|_{L^2_s([0,t])} 
\end{align*}
Since $\zeta$ is smooth with compact support, we have $\| \jy \zeta(y) \|_{L^1_y} < \infty$. Thus, from the decay estimate \eqref{dispdecayauxill}, we get 
\begin{align}\label{estimF(s,y)}
    \| \jy F(s,y) \|_{L^1_y L^2_s([0,t])} \lesssim  \Big\| \frac{C_0^3\varepsilon^3}{\js^{1/2}} \Big\|_{L^2_s([0,t])} \lesssim C_0^3\varepsilon^3  \sqrt{\log \jt}.
\end{align}
Combining \eqref{usesmoothestimI1} and \eqref{estimF(s,y)} leads to the desired estimate.\\\\
\emph{Estimates for $I_2(t,k)$.} We set $G_t(s) := s{\mathcal{T}}_{\what{\zeta},\eps}(f_1,f_2,f_3)(s,0) \mathbf{1}_{[0,t]}(s) $. Thus, we can write 
$$I_2(t,k) = \what{G_t}(k^2)\ .$$
Then, Lemma \ref{k^2lemma} gives 
\begin{align}
    \label{usek^2lemI2}
    \|I_2(t,k)\|_{L^2_k} \lesssim \|G_t(s)\|_{L^{4/3}_s} \ .
\end{align}
Now, from Lemma \ref{lem:inverseFD}, we can observe that 
$$ {\mathcal{T}}_{\what{\zeta},\eps}(f_1,f_2,f_3)(s,0) = \int_\R u_1(s,-\e_1x)\overline{u_2(s,\e_2x)}u_3(s,-\e_3x)\zeta(x) \ dx.$$
This implies the following pointwise bound
\begin{align}
    \label{ptwise Tzetae(s,0)}
    |{\mathcal{T}}_{\what{\zeta},\eps}(f_1,f_2,f_3)(s,0)| &\lesssim \| u_1(s,x)\|_{L^\infty_x} \| u_2(s,x)\|_{L^\infty_x} \| u_3(s,x)\|_{L^\infty_x} \|\zeta(x)\|_{L^1_x} \nonumber\\
    |{\mathcal{T}}_{\what{\zeta},\eps}(f_1,f_2,f_3)(s,0)| & \lesssim \frac{C_0^3\varepsilon^3}{\js^{3/2}}, 
\end{align}
due to the decay \eqref{dispdecayauxill} and $\zeta \in L^1$. The pointwise bound \eqref{ptwise Tzetae(s,0)} implies $ |G_t(s)| \lesssim C_0^3\varepsilon^3\js^{-1/2}$. This yields  
$$ \|G_t(s)\|_{L^{4/3}_s} \lesssim C_0^3\varepsilon^3 \jt^{1/4}.$$
Therefore, the desired estimate is obtained due to \eqref{usek^2lemI2}.
\end{proof}

\bigskip
\section{\texorpdfstring{$L^2$}{L2} estimates for the regular part}\label{secnonlinl2reg} 

In this section, we perform the $L^2$ estimates involving $\mathcal{N}_\ast$ for $\ast \in \{R,1 \ ;\ R,2\}$. The arguments follow closely the Section $6$ of \cite{CPNLS2}.

\begin{proposition}\label{l2reg1}
    Under the bootstrap hypothesis \eqref{boothyp}, it holds that for every $|t| \le T$,
   \begin{align}\begin{split}
		\label{weightmainr1}
		& {\Big\| \int_0^t \partial_k \mathcal{N}_{R,1}(s,k) \, ds \Big\|}_{L^2_k} \lesssim C_0^3\varepsilon^3 \jt^{1/4} \ .
        \end{split}
        \end{align}
\end{proposition}
\begin{proof}
    Theorem \ref{theomu} tells us that $\mu_{R,1}^\sharp(k,\ell,m,n)$ is a finite sum of terms of the form 
    \begin{align}\label{mur1(1)}
       \mu_{R,1}^{\sharp,(1)}(k,\ell,m,n):= \int_\R\overline{\mathcal{K}^{\#}_{R}(x,k)}\mathcal{K}^{\#}_{A_{1}}(x,\ell)
	\overline{\mathcal{K}^{\#}_{A_{2}}(x,m)}\mathcal{K}^{\#}_{A_{3}}(x,n)\,dx
    \end{align}
    or 
    \begin{align}\label{mur1(2)}
       \mu_{R,1}^{\sharp,(2)}(k,\ell,m,n):= \int_\R\overline{\mathcal{K}^{\#}_{S}(x,k)}\mathcal{K}^{\#}_{A'_{1}}(x,\ell)
	\overline{\mathcal{K}^{\#}_{A'_{2}}(x,m)}\mathcal{K}^{\#}_{A'_{3}}(x,n)\,dx,
    \end{align}
    where the indices $A_j$ and $A'_j$ are in $\{S,R\}$. Moreover, there exist $j \in \{1,2,3\}$ such that $A'_j =R$. Without loss of generality, we suppose that $A'_3 = R$. Thus, for $\bullet \in \{1,2\}$, we define 
    \begin{align} \label{NR1bull}
        \mathcal{N}_{R,1}^{(\bullet)}(s,k) &: =  \underset{\R^3}{\iiint} e^{is(-k^2+ \ell^2 - m^2 +n^2)}f^{\#}(\ell)\overline{f^{\#}(m)}f^{\#}(n)  
	\,\mu^{\#,(\bullet)}_{R,1}(k,\ell, m,n)
	\,\frac{dn}{2\pi}\frac{dm}{2\pi}\frac{d\ell}{2\pi}\ ,\\
    & = e^{-isk^2}\underset{\R^3}{\iiint}u^{\#}(\ell)\overline{u^{\#}(m)}u^{\#}(n)  
	\,\mu^{\#,(\bullet)}_{R,1}(k,\ell, m,n)
	\,\frac{dn}{2\pi}\frac{dm}{2\pi}\frac{d\ell}{2\pi}\ .
    \end{align}
    % Then, it suffices to show that  
    % $$ {\Big\| \int_0^t \partial_k \mathcal{N}_{R,1}^{(\bullet)}(s,k) \, ds \Big\|}_{L^2_k} \lesssim C_0^3\varepsilon^3 \jt^{1/4} \ .$$
    We have 
    \begin{align}
        \partial_k \mathcal{N}_{R,1}^{(\bullet)}(s,k) & = -2iske^{-isk^2}\underset{\R^3}{\iiint} u^{\#}(\ell)\overline{u^{\#}(m)}u^{\#}(n)  
	\,\mu^{\#,(\bullet)}_{R,1}(k,\ell, m,n)
	\,\frac{dn}{2\pi}\frac{dm}{2\pi}\frac{d\ell}{2\pi}\ \\
    & + e^{-isk^2}\underset{\R^3}{\iiint} u^{\#}(\ell)\overline{u^{\#}(m)}u^{\#}(n)  
	\,\partial_k\mu^{\#,(\bullet)}_{R,1}(k,\ell, m,n)
	\,\frac{dn}{2\pi}\frac{dm}{2\pi}\frac{d\ell}{2\pi}\ . 
    \end{align}
Moreover, 
\begin{align}
     \partial_k\mu_{R,1}^{\sharp,(1)}(k,\ell,m,n)= \int_\R\partial_k\overline{\mathcal{K}^{\#}_{R}(x,k)}\mathcal{K}^{\#}_{A_{1}}(x,\ell)
	\overline{\mathcal{K}^{\#}_{A_{2}}(x,m)}\mathcal{K}^{\#}_{A_{3}}(x,n)\,dx \ , \\
       \partial_k\mu_{R,1}^{\sharp,(2)}(k,\ell,m,n)= \int_\R\partial_k \overline{\mathcal{K}^{\#}_{S}(x,k)}\mathcal{K}^{\#}_{A'_{1}}(x,\ell)
	\overline{\mathcal{K}^{\#}_{A'_{2}}(x,m)}\mathcal{K}^{\#}_{R}(x,n)\,dx \ . 
\end{align}
We consider 
\begin{align}\label{defI(.)}
    I^{(\bullet)}(t,k) &: = \int_0^t iske^{-isk^2}\underset{\R^3}{\iiint} u^{\#}(\ell)\overline{u^{\#}(m)}u^{\#}(n)  
	\,\mu^{\#,(\bullet)}_{R,1}(k,\ell, m,n)
	\,\frac{dn}{2\pi}\frac{dm}{2\pi}\frac{d\ell}{2\pi}\ ds \ ,\\
    \label{defJ(.)}
    J^{(\bullet)}(t,k) &: = \int_0^t  e^{-isk^2}\underset{\R^3}{\iiint} u^{\#}(\ell)\overline{u^{\#}(m)}u^{\#}(n)  
	\,\partial_k\mu^{\#,(\bullet)}_{R,1}(k,\ell, m,n) 
	\,\frac{dn}{2\pi}\frac{dm}{2\pi}\frac{d\ell}{2\pi}\ ds \ .
\end{align}
It suffices to prove that 
\begin{align}
    \|I^{(\bullet)}(t,k)\|_{L^2_k} + \|J^{(\bullet)}(t,k)\|_{L^2_k} \lesssim C_0^3\varepsilon^3 \jt^{1/4} \ . 
\end{align}
Due to the growing factor $k$ in $I^{(\bullet)}$, we perform a low$\backslash$high frequency analysis for this term. Thus, we consider non-negative cut-off functions $\phi_1$ and $\phi_2$ such that $\phi_1$ is supported in $[-2,2]$, $\phi_1(k)= 1$ for $|k| \le 1$ and $\phi_1 + \phi_2 =1$. For $\ast \in \{1,2\}$, we define 
\begin{align}
    I^{(\bullet)}_{\phi_\ast}(t,k) &: = \int_0^t is\phi_{\ast}(k)ke^{-isk^2}\underset{\R^3}{\iiint} u^{\#}(\ell)\overline{u^{\#}(m)}u^{\#}(n)  
	\,\mu^{\#,(\bullet)}_{R,1}(k,\ell, m,n)
	\,\frac{dn}{2\pi}\frac{dm}{2\pi}\frac{d\ell}{2\pi}\ ds \ .
 %    J^{(\bullet)}_{\phi_\ast}(t,k) &: = \int_0^t  \phi_{\ast}(k)e^{-isk^2}\underset{\R^3}{\iiint} u^{\#}(\ell)\overline{u^{\#}(m)}u^{\#}(n)  
	% \,\partial_k\mu^{\#,(\bullet)}_{R,1}(k,\ell, m,n) 
	% \,\frac{dn}{2\pi}\frac{dm}{2\pi}\frac{d\ell}{2\pi}\ ds \ .
\end{align}
We proceed in two steps.
\subsubsection*{Step I: Estimates for $\mu^{\#,(1)}_{R,1}$} 
First, from the expression \eqref{mur1(1)} of $\mu^{\#,(1)}_{R,1}$, we can write 
\begin{align}
    I^{(1)}_{\phi_1}(t,k) =  \int_0^t i\phi_{1}(k)ke^{-isk^2}\underset{\R}{\int} \overline{\K_R^\sharp(x,k)} su_{A_1}(s,x)\overline{u_{A_2}(s,x)}u_{A_3}(s,x) \ dx \ ds \ ,
\end{align}
where $u_{A_j}$ are defined via definition \ref{singregproj}. Then, we use the inhomogeneous smoothing estimate \eqref{eq:smoothingQim} by taking $\mathcal{Q}(x,k) = \K_R^\sharp(x,k)$ with $\beta = \gamma-1$ (due to Lemma \ref{estimkR}), and $F(s,x) = s u_{A_1}(s,x)\overline{u_{A_2}(s,x)}u_{A_3}(s,x)$. We get 
\begin{align}
    \| I^{(1)}_{\phi_1}(t,k)\|_{L^2_k} \lesssim \Big\| \jx^{1-\gamma} s u_{A_1}(s,x)\overline{u_{A_2}(s,x)}u_{A_3}(s,x) \Big\|_{L^1_x L^2_s([0,t])} \ .
\end{align}
Corollaries \ref{cor:pointwisesingular} and \ref{cor:regularLinfty} combined with the bootstrap assumption \eqref{boothyp} give the decay 
\begin{align} \label{decayuAj}
 \|u_{A_j}(s,x)\|_{L^\infty_x} \lesssim \frac{C_0\varepsilon}{\js^{1/2}} \ .
 \end{align}
Since $x \mapsto \jx^{1-\gamma} \in L^1$, we obtain
$$ \| I^{(1)}_{\phi_1}(t,k)\|_{L^2_k} \lesssim \Big\| \frac{C_0^3\varepsilon^3}{\js^{1/2}} \Big\|_{ L^2_s([0,t])} \lesssim C_0^3\varepsilon^3\sqrt{\log \jt} \lesssim C_0^3\varepsilon^3 \jt^{1/4} \ .$$

Second, we can write 
\begin{align}
    I^{(1)}_{\phi_2}(t,k) =  -\int_0^t \phi_{2}(k)e^{-isk^2}\underset{\R}{\int} \overline{ik\K_R^\sharp(x,k)} su_{A_1}(s,x)\overline{u_{A_2}(s,x)}u_{A_3}(s,x) \ dx \ ds \ .
\end{align}
Now, we recall the expression \eqref{kR}
 of $\K_R^\sharp(x,k)$. We can observe that 
 \begin{align} \label{KRid}
     ik\K_R^\sharp(x,k) = \partial_x \mathcal{Q}^{\sharp,R}_1(x,k) - \mathcal{Q}^{\sharp,R}_2(x,k)
 \end{align}
 where 
 \begin{align}
     \mathcal{Q}^{\sharp,R}_1(x,k) &: = \begin{cases}
    \mathcal{G}_+^+(x,k)e^{ikx} - \mathcal{G}_+^-(x,k)e^{-ikx} \quad \quad \text{ for } k\ge 0 \\
    \mathcal{G}_-^+(x,k)e^{ikx} - \mathcal{G}_-^-(x,k)e^{-ikx} \quad \quad \text{ for } k<0
    \end{cases}\ ,\\
     \mathcal{Q}^{\sharp,R}_2(x,k) &:= \begin{cases}
    \partial_x\mathcal{G}_+^+(x,k)e^{ikx} - \partial_x\mathcal{G}_+^-(x,k)e^{-ikx} \quad \quad \text{ for } k\ge 0 \\
    \partial_x\mathcal{G}_-^+(x,k)e^{ikx} - \partial_x\mathcal{G}_-^-(x,k)e^{-ikx} \quad \quad \text{ for } k<0
    \end{cases} \ .
 \end{align}
 In particular, due to Lemma \ref{estimG+-}, we have 
 \begin{align}
\label{decayQR}|\mathcal{Q}^{\sharp,R}_1(x,k)|+|\mathcal{Q}^{\sharp,R}_2(x,k)| \le \jx^{-\gamma +1} \ .
\end{align}

We define
\begin{align}\label{I1phi21}
    I^{(1)}_{\phi_2,1}(t,k) &:=  -\int_0^t \phi_{2}(k)e^{-isk^2}\underset{\R}{\int} \partial_x\overline{\mathcal{Q}_1^{\sharp,R}(x,k)} su_{A_1}(s,x)\overline{u_{A_2}(s,x)}u_{A_3}(s,x) \ dx \ ds \ ,\\
    \label{I1phi22}
    I^{(1)}_{\phi_2,2}(t,k) &:=  \int_0^t \phi_{2}(k)e^{-isk^2}\underset{\R}{\int} \overline{\mathcal{Q}_2^{\sharp,R}(x,k)} su_{A_1}(s,x)\overline{u_{A_2}(s,x)}u_{A_3}(s,x) \ dx \ ds \ 
\end{align}
such that $I^{(1)}_{\phi_2}(t,k) = I^{(1)}_{\phi_2,1}(t,k) + I^{(1)}_{\phi_2,2}(t,k)$. For the term $I^{(1)}_{\phi_2,1}$, we perform an integration by part in the variable $x$ to get 
\begin{align}
    I^{(1)}_{\phi_2,1}(t,k) = \int_0^t \phi_{2}(k)e^{-isk^2}\underset{\R}{\int} \overline{\mathcal{Q}_1^{\sharp,R}(x,k)} \partial_x[su_{A_1}(s,x)\overline{u_{A_2}(s,x)}u_{A_3}(s,x)] \ dx \ ds \ .
\end{align}
We use the inhomogeneous smoothing estimate \eqref{eq:smoothingQimh} by taking $\mathcal{Q}(x,k) = \mathcal{Q}_1^{\sharp,R}(x,k)$ with $\beta = \gamma -1$ (due to \eqref{decayQR}), and $F(s,x) = \partial_x[su_{A_1}(s,x)\overline{u_{A_2}(s,x)}u_{A_3}(s,x)]$. We obtain
$$\| I^{(1)}_{\phi_2,1}(t,k) \|_{L^2_k} \lesssim \Big\| \jx^{1-\gamma} s \partial_x[u_{A_1}(s,x)\overline{u_{A_2}(s,x)}u_{A_3}(s,x)] \Big\|_{L^1_xL^2_s([0,t])} \ .$$ 
Without loss of generality, it suffices to estimate
\begin{align}
    \label{dertargetA1}
\Big\| \jx^{1-\gamma} s \partial_xu_{A_1}(s,x)\ \overline{u_{A_2}(s,x)}u_{A_3}(s,x)\Big\|_{L^1_xL^2_s([0,t])},
\end{align} 
since the case where $\partial_x$ acts on $\overline{u_{A_2}}$ or $u_{A_3}$ is similar. We have 
\begin{align}
    \eqref{dertargetA1} \lesssim \Big\|s \|\jx^{1-\gamma}  \partial_xu_{A_1}(s,x)\|_{L^1_x}\ \|u_{A_2}(s,x)\|_{L^\infty_x}\|u_{A_3}(s,x)\|_{L^\infty_x} \Big\|_{L^2_s([0,t])} \ .
\end{align}
By Cauchy-Schwarz, we have 
$$\|\jx^{1-\gamma}\partial_xu_{A_1}(s,x)\|_{L^1_x} \le \|\jx^{-1}\partial_xu_{A_1}(s,x)\|_{L^2_x}\|\jx^{2-\gamma}\|_{L^2_x} \lesssim \|\jx^{-1}\partial_xu_{A_1}(s,x)\|_{L^2_x} \ ,$$ 
since $\gamma > 5/2$. So, we have the estimate
\begin{align}
\eqref{dertargetA1} \lesssim \Big\|s \|\jx^{-1}  \partial_xu_{A_1}(s,x)\|_{L^2_x}\ \|u_{A_2}(s,x)\|_{L^\infty_x}\|u_{A_3}(s,x)\|_{L^\infty_x} \Big\|_{L^2_s([0,t])} \ .
\end{align}
The local decay from Corollary \ref{cor:locdecdiffflow}, combined with bootstrap assumption \eqref{boothyp} gives  \begin{align}
    \label{localdecayagain}\Big\|\jx^{-1}  \partial_xu_{A_1}(s,x)\Big\|_{L^2_x} \lesssim C_0\varepsilon \js^{-1/4} \ .
\end{align}
Then, due to the dispersive decay \eqref{decayuAj}, this yields
\begin{align}
    \eqref{dertargetA1} \lesssim \Big\|C_0^3\varepsilon^3 \js^{-1/4} \Big\|_{L^2_s([0,t])} \lesssim C_0^3\varepsilon^3 \jt^{1/4} \ .
\end{align}

For the term $I^{(1)}_{\phi_2,2}$ \eqref{I1phi22}, we use the inhomogeneous smoothing estimate \eqref{eq:smoothingQimh} by taking $\mathcal{Q}(x,k) = \mathcal{Q}_2^{\sharp,R}(x,k)$ with $\beta = \gamma -1$ (due to \eqref{decayQR}), and $F(s,x) = su_{A_1}(s,x)\overline{u_{A_2}(s,x)}u_{A_3}(s,x)$. We get
$$\| I^{(1)}_{\phi_2,1}(t,k) \|_{L^2_k} \lesssim \Big\| \jx^{1-\gamma} s u_{A_1}(s,x)\overline{u_{A_2}(s,x)}u_{A_3}(s,x)\Big\|_{L^1_xL^2_s([0,t])} \ .$$
Due to the dispersive decay \eqref{decayuAj} and the fact that $x \mapsto \jx^{1-\gamma} \in L^1$, we get 
$$\| I^{(1)}_{\phi_2,1}(t,k) \|_{L^2_k} \lesssim \Big\| C_0^3\varepsilon^3 \js^{-1/2} \Big\|_{L^2_s([0,t])} \lesssim C_0^3\varepsilon^3 \sqrt{\log \jt} \ . $$

Next, we treat $J^{(1)}(t,k)$ defined at \eqref{defJ(.)}. Using triangle inequality, we can write 
\begin{align}
    \left\Vert J^{(1)}(t,k) \right\Vert _{L^{2}_k} 
	& \lesssim \int_0^t \left\Vert  \int \overline{\partial_k \mathcal{K}^{\#}_R(x,k)}
	 u_{A_{1}}(s,x) \overline{u_{A_{2}}(s,x)} u_{A_{3}}(s,x)  \, dx \right\Vert_{L^2_k}\,ds   \ . 
\end{align}
The pseudo-differential inequality on $L^2$ \eqref{eq:weiFR} gives 
	\begin{align*}
	 \left\Vert J^{(1)}(t,k) \right\Vert _{L^{2}_k} 
	& \lesssim \int_0^t \left\Vert  
	 u_{A_{1}}(s,x) \overline{u_{A_{2}}(s,x)} u_{A_{3}}(s,x)   \right\Vert_{L^2_x}\,ds \\
     & \lesssim \int_0^t   
	 \|u_{A_{1}}(s,x)\|_{L^\infty_x} \|u_{A_{2}}(s,x)\|_{L^\infty_x} \|u_{A_{3}}(s,x)   \|_{L^2_x}\,ds \\
     &	\lesssim  C_0^3\varepsilon^3 \log \jt.
	\end{align*}
    
\subsubsection*{Step II: Estimates for $\mu^{\sharp,(2)}_{R,1}$} First, from the expression \eqref{mur1(2)} of $\mu^{\sharp,(2)}_{R,1}$, we can write 
\begin{align}
    I_{\phi_1}^{(2)}(t,k) =  \int_0^t i\phi_{1}(k)ke^{-isk^2}\underset{\R}{\int} \overline{\K_S^\sharp(x,k)} su_{A_1'}(s,x)\overline{u_{A_2'}(s,x)}u_{R}(s,x) \ dx \ ds \ .
\end{align}
We use the inhomogeneous smoothing estimate \eqref{eq:smoothingQim} by taking $\mathcal{Q}(x,k) = \K_S^\sharp(x,k)$ with $\beta = 0$, and $F(s,x) = su_{A_1'}(s,x)\overline{u_{A_2'}(s,x)}u_{R}(s,x)$, to get  
\begin{align}\label{estimuA1uA2uR}
    \| I_{\phi_1}^{(2)}(t,k)\|_{L^2_k} &\lesssim \Big\| su_{A_1'}(s,x)u_{A_2'}(s,x)u_{R}(s,x) \Big\|_{L^1_x L^2_s([0,t])}.
\end{align}
Thus, we can write
\begin{align}
   \| I_{\phi_1}^{(2)}(t,k)\|_{L^2_k} \lesssim \Big\| s\|u_{A_1'}(s,x)\|_{L^\infty_x}\|u_{A_2'}(s,x)\|_{L^\infty_x}\| \jx^{\gamma-1}u_{R}(s,x)\|_{L^\infty_x} \| \jx^{1-\gamma}\|_{L^1_x} \Big\|_{ L^2_s([0,t])} \ .
\end{align}
We note that $x\mapsto \jx^{1-\gamma} \in L^1$. Corollary \ref{cor:regularLinfty} combined with the bootstrap assumption \eqref{boothyp} gives 
\begin{align}\label{decayuRagain}
    \|\jx^{\gamma-1}u_{R}(s,x)\|_{L^\infty_x} \lesssim C_0\varepsilon \js^{-1/2} \ .
\end{align}
Using this and the dispersion decay \eqref{decayuAj}, we obtain 
\begin{align}
    \| I_{\phi_1}^{(2)}(t,k)\|_{L^2_k} \lesssim \Big\| C_0^3\varepsilon^3 \js^{-1/2} \Big\|_{ L^2_s([0,t])} \lesssim C_0^3\varepsilon^3 \sqrt{\log \jt } \ .
\end{align}

Second, we have 
\begin{align}
    I_{\phi_2}^{(2)}(t,k) =  -\int_0^t \phi_{2}(k)e^{-isk^2}\underset{\R}{\int} \overline{ik\K_S^\sharp(x,k)} su_{A_1'}(s,x)\overline{u_{A_2'}(s,x)}u_{R}(s,x) \ dx \ ds \ .
\end{align}
Recalling the formulas \eqref{KsharpSdecomp}-\eqref{KsharpSpm} of $K_S^\sharp(x,k)$, we can observe that 
\begin{align}
    ikK_S^\sharp(x,k) = \partial_x \mathcal{Q}^{\sharp,S}_1(x,k) - \mathcal{Q}^{\sharp,S}_2(x,k) 
\end{align}
where 
\begin{align}
    \mathcal{Q}^{\sharp,S}_1(x,k) & := \chi_0(x)\mathcal{Q}^{\sharp,S,0}_1(x,k) + \chi_+(x) \mathcal{Q}^{\sharp,S,+}_1(x,k) 
	+ \chi_-(x) \mathcal{Q}^{\sharp,S,-}_1(x,k) \ , \\
    \mathcal{Q}^{\sharp,S}_2(x,k) & := \chi_0(x)\mathcal{Q}^{\sharp,S,0}_2(x,k) + \partial_x\chi_+(x) \mathcal{Q}^{\sharp,S,+}_1(x,k) 
	+ \partial_x\chi_-(x) \mathcal{Q}^{\sharp,S,-}_1(x,k) \,
\end{align}
with 
	\begin{align}\label{KS1}
	\begin{split}
    \mathcal{Q}^{\sharp,S,+}_1(x,k)  &:=  a_+^+(k)e^{ikx} - a_+^-(k)e^{-ikx} \ ,
	\\
	 \mathcal{Q}^{\sharp,S,-}_1(x,k) & := a_-^+(k)e^{ikx} - a_-^-(k)e^{-ikx} \ , \\
     \mathcal{Q}^{\sharp,S,0}_1(x,k) &: = -(\chi_+(x)aA_- + \chi_-(x)B_+)m_-(x,0)e^{-ikx}  \\
     & + (\chi_-(x)\frac{1}{a}A_+ + \chi_+(x)B_-)m_+(x,0)e^{ikx} \ , \\
    \mathcal{Q}^{\sharp,S,0}_2(x,k) & := - \partial_x\Big[(\chi_+(x)aA_-+ \chi_-(x)B_+)m_-(x,0)\Big]e^{-ikx}  \\
     & + \partial_x\Big[(\chi_-(x)\frac{1}{a}A_+ + \chi_+(x)B_-)m_+(x,0) \Big]e^{ikx} \ . 
	\end{split}
	\end{align}
Moreover, we have 
\begin{align}
    \label{decayQS}
|\mathcal{Q}^{\sharp,S}_1(x,k)| + |\mathcal{Q}^{\sharp,S}_2(x,k)| \lesssim 1 \ . 
\end{align}
We define
\begin{align}\label{I2phi21}
    I^{(2)}_{\phi_2,1}(t,k) &:=  -\int_0^t \phi_{2}(k)e^{-isk^2}\underset{\R}{\int} \partial_x\overline{\mathcal{Q}_1^{\sharp,S}(x,k)} su_{A_1'}(s,x)\overline{u_{A_2'}(s,x)}u_{R}(s,x) \ dx \ ds \ ,\\
    \label{I2phi22}
    I^{(2)}_{\phi_2,2}(t,k) &:=  \int_0^t \phi_{2}(k)e^{-isk^2}\underset{\R}{\int} \overline{\mathcal{Q}_2^{\sharp,S}(x,k)} su_{A_1'}(s,x)\overline{u_{A_2'}(s,x)}u_{R}(s,x) \ dx \ ds \ 
\end{align}
such that $I^{(2)}_{\phi_2}(t,k) = I^{(2)}_{\phi_2,1}(t,k) + I^{(2)}_{\phi_2,2}(t,k)$. For the term $I^{(2)}_{\phi_2,1}$, we perform an integration by part in the variable $x$ to get 
\begin{align}
    I^{(2)}_{\phi_2,1}(t,k) = \int_0^t \phi_{2}(k)e^{-isk^2}\underset{\R}{\int} \overline{\mathcal{Q}_1^{\sharp,S}(x,k)} \partial_x[su_{A_1'}(s,x)\overline{u_{A_2'}(s,x)}u_{R}(s,x)] \ dx \ ds \ .
\end{align}
We use the inhomogeneous smoothing estimate \eqref{eq:smoothingQimh} by taking $\mathcal{Q}(x,k) = \mathcal{Q}_1^{\sharp,S}(x,k)$ with $\beta = 0$ (due to \eqref{decayQS}), and $F(s,x) = \partial_x[su_{A_1'}(s,x)\overline{u_{A_2'}(s,x)}u_{R}(s,x)]$. We obtain
\begin{align}
    \|  I^{(2)}_{\phi_2,1}(t,k)\|_{L^2_k} \lesssim \Big\|  \partial_x[su_{A_1'}(s,x)\overline{u_{A_2'}(s,x)}u_{R}(s,x)]\Big\|_{L^1_xL^2_s([0,t])} \ .
\end{align}
It suffices to estimate 
\begin{align}
    \label{deruA1'}
    \Big\|  s\partial_xu_{A_1'}(s,x) \ u_{A_2'}(s,x) u_{R}(s,x)\Big\|_{L^1_x L^2_s([0,t])} \ ,\\
    \label{deruA2'}
    \Big\|  su_{A_1'}(s,x) \partial_x u_{A_2'}(s,x) \  u_{R}(s,x)\Big\|_{L^1_x L^2_s([0,t])} \ , \\
    \label{deruA3'}
    \Big\|  su_{A_1'}(s,x) u_{A_2'}(s,x) \ \partial_xu_{R}(s,x)\Big\|_{L^1_x L^2_s([0,t])} \ .
\end{align}
Since the estimation of \eqref{deruA1'} and \eqref{deruA2'} is similar, we focus on the first one. So we can write 
\begin{align}
    \eqref{deruA1'} \lesssim \Big\|  s\|\jx^{-1}\partial_xu_{A_1'}(s,x)\jx^{2-\gamma}\|_{L^1_x} \ \|u_{A_2'}(s,x)\|_{L^\infty_x} \|\jx^{\gamma -1}u_{R}(s,x)\|_{L^\infty_x}\Big\|_{ L^2_s([0,t])} 
\end{align}
and by Cauchy-Schwarz, we have 
\begin{align}
    \|\jx^{-1}\partial_xu_{A_1'}(s,x)\jx^{2- \gamma }\|_{L^1_x} \le \|\jx^{-1}\partial_xu_{A_1'}(s,x)\|_{L^2_x} \|\jx^{2-\gamma}\|_{L^2_x} . 
\end{align}
We know that $\|\jx^{2- \gamma}\|_{L^2_x} < \infty$, since $\gamma > 5/2$. Thus, we can write  
\begin{align}
    \eqref{deruA1'} \lesssim \Big\|  s\|\jx^{-1}\partial_xu_{A_1'}(s,x)\|_{L^2_x} \ \|u_{A_2'}(s,x)\|_{L^\infty_x} \|\jx^{\gamma -1}u_{R}(s,x)\|_{L^\infty_x}\Big\|_{ L^2_s([0,t])} 
\end{align}
We combine the decays \eqref{localdecayagain}, \eqref{decayuAj} and \eqref{decayuRagain} to get 
\begin{align}
    \eqref{deruA1'} \lesssim \Big\|  C_0^3\varepsilon^3 \js^{-1/4}\Big\|_{ L^2_s([0,t])} \lesssim  C_0^3\varepsilon^3 \jt^{1/4} \ .
\end{align}
For the estimate for \eqref{deruA3'}, we write
\begin{align}
    \eqref{deruA3'} \lesssim \Big\|  s\|u_{A_1'}(s,x)\|_{L^\infty_x} \|u_{A_2'}(s,x)\|_{L^\infty_x} \ \|\jx^{\gamma -1}\partial_xu_{R}(s,x)\|_{L^2_x} \| \jx^{1-\gamma}\|_{L^2_x}\Big\|_{ L^2_s([0,t])}  \ .
\end{align}
Using the fact that $x \mapsto \jx^{1-\gamma} \in L^2$, the dispersive decays \eqref{decayuAj} and \eqref{decayuRagain}, we get 
\begin{align}
    \eqref{deruA3'} \lesssim \Big\|  C_0^3\varepsilon^3 \js^{-1/4}\Big\|_{ L^2_s([0,t])} \lesssim  C_0^3\varepsilon^3 \jt^{1/4} \ . 
\end{align}

We treat the term $I^{(2)}_{\phi_2, 2}$ \eqref{I2phi22}. We use the inhomogeneous smoothing estimate \eqref{eq:smoothingQimh} again by taking $\mathcal{Q}(x,k) = \mathcal{Q}_2^{\sharp,S}(x,k)$ with $\beta =0$ (due to \eqref{decayQS}), and $F(s,x) =su_{A_1'}(s,x)\overline{u_{A_2'}(s,x)}u_{R}(s,x) $ to get 
\begin{align}
    \| I^{(2)}_{\phi_2, 2}(t,k) \|_{L^2_k} \lesssim \Big\| su_{A_1'}(s,x)\overline{u_{A_2'}(s,x)}u_{R}(s,x) \Big\|_{L^1_x L^2_s([0,t])} \ .
\end{align}
The estimation of \eqref{estimuA1uA2uR} gives the desired estimate. 
\smallskip
Next, we estimate $J^{(2)}(t,k)$ defined at \eqref{defJ(.)}. Using triangle inequality, we can write 
\begin{align}
    \left\Vert J^{(2)}(t,k) \right\Vert _{L^{2}_k} 
	& \lesssim \int_0^t \left\Vert  \int \overline{\partial_k \mathcal{K}^{\#}_S(x,k)}
	 u_{A_{1}'}(s,x) \overline{u_{A_{2}'}(s,x)} u_{R}(s,x)  \, dx \right\Vert_{L^2_k}\,ds  \ .
\end{align}
The pseudo-differential inequality on $L^2$ \eqref{eq:weiF}-\eqref{eq:weiFR} gives 
{\setlength{\belowdisplayskip}{-10pt}
\setlength{\belowdisplayshortskip}{-10pt}	\begin{align*}
	 \left\Vert J^{(2)}(t,k) \right\Vert _{L^{2}_k} 
	& \lesssim \int_0^t \left\Vert  
	 u_{A_{1}'}(s,x) \overline{u_{A_{2}'}(s,x)} u_{R}(s,x)   \right\Vert_{L^2_x}\,ds \\
     & \lesssim \int_0^t   
	 \|u_{A_{1}'}(s,x)\|_{L^\infty_x} \|u_{A_{2}'}(s,x)\|_{L^\infty_x} \|u_R(s,x)   \|_{L^2_x}\,ds \\
     &	\lesssim  C_0^3\varepsilon^3 \log \jt \ .
	\end{align*}
    }
\end{proof}

	\medskip
    \begin{proposition} \label{l2boundreg2}
Under the bootstrap hypothesis \eqref{boothyp}, it holds that for every $|t| \le T$,
    \begin{align}\begin{split}
		\label{weightmainr2}
		& {\Big\| \int_0^t \partial_k \mathcal{N}_{R,2}(s,k) \, ds \Big\|}_{L^2_k} \lesssim C_0^3\varepsilon^3 \jt^{1/4} \ . 
        \end{split}
        \end{align}
    \end{proposition}

 \begin{proof}
We recall the definition of $\mu^\sharp_{R,2}$ at \eqref{muR2}. Thus, for $\e = (\e_0,\e_1,\e_2,\e_3) \in \{+,-\}^4$, we define 
\begin{align}\label{mur2eps}
    \mu^{\sharp,\e}_{R,2} (k,\ell,m,n) : = a_{\e_0,\e_1,\e_2,\e_3}(k,\ell,m,n)\  \what{\Fou}(\varphi_\e)(\e_0k + \e_1\ell +\e_2m + \e_3n) 
\end{align}
where $\varphi_\e$ and $a_{\e_0,\e_1,\e_2,\e_3}(k,\ell,m,n) = \overline{a_{\e_0}(k)}a_{\e_1}(\ell)\overline{a_{\e_2}(m)}a_{\e_3}(n)$ are given in \eqref{muR2}. We recall that the functions $a_{\e_j}$ and their first derivatives are bounded. We define 
\begin{align} \label{nr2eps}
    \mathcal{N}_{R,2}^\e(s,k) & := \underset{\R^3}{\iiint} e^{is(-k^2+\ell^2-m^2+n^2)}
	f^{\#}(\ell)\overline{f^{\#}(m)} f^{\#}(n)\mu_{R,2}^{\#,\e}(k,\ell,m,n)\,\frac{dn}{2\pi}\frac{dm}{2\pi}\frac{d\ell}{2\pi} \ .
\end{align}
We have
% \begin{align}
%     \e_0\partial_k \mathcal{N}_{R,2}^\e(s,k) & = \e_0\partial_k \overline{a_{\e_0}(k)} \mathcal{T}_{\what{\varphi_\e}, \e}\big( a_{\e_1}f^{\#} , a_{\e_2}f^{\#} , a_{\e_3}f^{\#}\big) (s,k) \\
%     & -\e_1\mathcal{T}_{\what{\varphi_\e},\e}
% 		\left(\partial_{\ell}f_{1},f_{2},f_{3}\right)(t,k)+\e_2\mathcal{T}_{\what{\varphi_\e},\e}\left(f_{1},\partial_{m}f_{2},f_{3}\right)(t,k)
% 		\\
% 		&  -\e_3\mathcal{T}_{\what{\varphi_\e},\e}\left(f_{1},f_{2},\partial_{n}f_{3}\right)(t,k)
% 		 -2it \, \mathcal{T}_{\xi\what{\varphi_\e},\e}(f_1,f_2,f_3)(t,k).
% \end{align}
\begin{align}
    \partial_k \mathcal{N}_{R,2}^\e(s,k) &= -2isk\underset{\R^3}{\iiint} e^{is(-k^2+\ell^2-m^2+n^2)}
	f^{\#}(\ell)\overline{f^{\#}(m)} f^{\#}(n)\mu_{R,2}^{\#,\e}(k,\ell,m,n)\,\frac{dn}{2\pi}\frac{dm}{2\pi}\frac{d\ell}{2\pi} \\ 
    & + \underset{\R^3}{\iiint} e^{is(-k^2+\ell^2-m^2+n^2)}
	f^{\#}(\ell)\overline{f^{\#}(m)} f^{\#}(n)\ \partial_k\mu_{R,2}^{\#,\e}(k,\ell,m,n)\,\frac{dn}{2\pi}\frac{dm}{2\pi}\frac{d\ell}{2\pi} \ .
\end{align}
Moreover, 
\begin{align}
    \partial_k\mu_{R,2}^{\#,\e}(k,\ell,m,n) &= \partial_k \overline{a_{\e_0}(k)}a_{\e_1}(\ell)\overline{a_{\e_2}(m)}a_{\e_3}(n)\  \what{\Fou}(\varphi_\e)(\e_0k + \e_1\ell +\e_2m + \e_3n) \\
    & -i\e_0 \overline{a_{\e_0}(k)}a_{\e_1}(\ell)\overline{a_{\e_2}(m)}a_{\e_3}(n) \what{\Fou}(x\varphi_\e)(\e_0k + \e_1\ell +\e_2m + \e_3n) \ .
\end{align}
We define  
\begin{align} \label{I1e}
    I^\e(t,k) &: = \int_0^t -2isk\underset{\R^3}{\iiint} e^{is(-k^2+\ell^2-m^2+n^2)}
	f^{\#}(\ell)\overline{f^{\#}(m)} f^{\#}(n)\mu_{R,2}^{\#,\e}(k,\ell,m,n)\,\frac{dn}{2\pi}\frac{dm}{2\pi}\frac{d\ell}{2\pi} \ .
\end{align}
Recalling the definition \eqref{eq:tril}, we define
\begin{align}
    \begin{split}
    \label{I2e}
     II^\e(t,k) &: = \int_0^t \partial_k \overline{a_{\e_0}(k)} \mathcal{T}_{\what{\Fou}(\varphi_\e), \e}\big( a_{\e_1}f^{\#} , a_{\e_2}f^{\#} , a_{\e_3}f^{\#}\big) (k) \ ds  \\
     & + \int_0^t -i\e_0 \overline{a_{\e_0}(k)} \mathcal{T}_{\what{\Fou}(x\varphi_\e), \e}\big( a_{\e_1}f^{\#} , a_{\e_2}f^{\#} , a_{\e_3}f^{\#}\big) (k) \ ds \ ,
     \end{split}
\end{align}
such that 
$$ \partial_k \mathcal{N}_{R,2}^\e(s,k) = I^\e(t,k) + II^\e(t,k) \ .$$
Thus, it suffices to show that  
$$ \| I^\e(t,k) \|_{L^2_k}  + \| II^\e(t,k) \|_{L^2_k}  \lesssim C_0^3\varepsilon^3 \jt^{1/4} \ . $$

We set  $f_j := a_{\e_j}f^{\#}$ for $j \in \{1,2,3\}$. We proceed in 2 steps. 
\subsubsection*{Step 1: Estimation of $I^\e$}
We can write 
\begin{align}
    I^\e(t,k) = \int_0^t -2ik\overline{a_{\e_0}(k)}e^{-isk^2}{\int}_\R e^{-i\e_0kx}\varphi_\e(x) 
	su_1(s,x) \overline{u_2(s,-x)} u_3(s,x) \ dx \ ,
\end{align}
where 
\begin{align}\label{defu_j}
u_j(s,x) := \int_\R e^{-\e_jikx + isk^2} f_j(k) \frac{dk}{2\pi} \ . 
\end{align}

We perform a low$\backslash$high frequency analysis. Thus, we consider non-negative cut-off functions $\phi_1$ and $\phi_2$ such that $\phi_1$ is supported in $[-2,2]$, $\phi_1(k)= 1$ for $|k| \le 1$ and $\phi_1 + \phi_2 =1$. For $\ast \in \{1,2\}$, we define
\begin{align}
    I_{\phi_\ast}^\e(t,k) = \int_0^t -2ik\phi_\ast(k)\overline{a_{\e_0}(k)}e^{-isk^2}{\int}_\R e^{-i\e_0kx}\varphi_\e(x) 
	su_1(s,x) \overline{u_2(s,-x)} u_3(s,x) \ dx \ . 
\end{align}
For low frequencies, we use the inhomogeneous smoothing estimate \eqref{eq:smoothingQim} by taking $\mathcal{Q}(x,k) = e^{i\e_0kx}$ with $\beta = 0$, and  $F(s,x) =  \varphi_\e(x) 
	su_1(s,x) \overline{u_2(s,-x)} u_3(s,x)$. We get 
\begin{align}
    \| I_{\phi_1}^\e(t,k) \|_{L^2_k} &\lesssim \Big\| \varphi_\e(x) 
	su_1(s,x) \overline{u_2(s,-x)} u_3(s,x) \Big\|_{L^1_x L^2_s([0,t])} \\
     &\lesssim \Big\| \|\varphi_\e\|_{L^1}
	s\|u_1(s,x)\|_{L^\infty_x} \|u_2(s,x)\|_{L^\infty_x} \|u_3(s,x)\|_{L^\infty_x} \Big\|_{L^2_s([0,t])} .
\end{align}
Due to Lemma \ref{lemstat}, and the bootstrap hypothesis \eqref{boothyp}, we have the decay 
\begin{align}
    \label{decayujagain}
    \|u_j(s,x)\|_{L^\infty_x} \lesssim C_0\varepsilon \js^{-1/2} \ . 
\end{align}
Recalling that $\varphi_\e \in L^1$, this gives
\begin{align}
    \| I_{\phi_1}^\e(t,k) \|_{L^2_k} \lesssim \Big\| C_0^3\varepsilon^3\js^{-1/2} \Big\|_{L^2_s([0,t])}  \lesssim C_0^3\varepsilon^3 \jt^{1/4} \ . 
\end{align}

For high frequencies, we perform integration by parts in the $x$ variable to absorb the growing factor $k$. In fact, we can write
\begin{align}
    I_{\phi_2}^\e(t,k) = \int_0^t -2\e_0\phi_2(k)\overline{a_{\e_0}(k)}e^{-isk^2}{\int}_\R e^{-i\e_0kx} \partial_x [\varphi_\e(x) 
	su_1(s,x) \overline{u_2(s,-x)} u_3(s,x) ] \ dx \ . 
\end{align}
We use the inhomogeneous smoothing estimate \eqref{eq:smoothingQimh} by taking $\mathcal{Q}(x,k) = e^{i\e_0kx}$ with $\beta = 0$, and  $F(s,x) =  \partial_x [\varphi_\e(x) 
	su_1(s,x) \overline{u_2(s,-x)} u_3(s,x)]$. We obtain
\begin{align}
    \| I_{\phi_2}^\e(t,k)\|_{L^2_k} \lesssim \Big\| \partial_x [ \varphi_\e(x) 
	su_1(s,x) \overline{u_2(s,-x)} u_3(s,x)] \Big\|_{L^1_x L^2_s([0,t])} 
\end{align}
Then, it suffices to estimate 
\begin{align}
    \label{deru1}
    \Big\| \varphi_\e(x) 
	s \partial_x u_1(s,x)\  u_2(s,-x) u_3(s,x) \Big\|_{L^1_x L^2_s([0,t])} \ , \\
    \label{deru2}
     \Big\| \varphi_\e(x) 
	s  u_1(s,x)  \partial_x u_2(s,-x) \ u_3(s,x) \Big\|_{L^1_x L^2_s([0,t])} \ ,\\
    \label{deru3}
    \Big\| \varphi_\e(x) 
	s  u_1(s,x)   u_2(s,-x) \ \partial_x u_3(s,x) \Big\|_{L^1_x L^2_s([0,t])} \ ,\\
    \label{derphi}
    \Big\| \partial_x  \varphi_\e(x) \ 
	su_1(s,x) u_2(s,-x) u_3(s,x) \Big\|_{L^1_x L^2_s([0,t])} \ . 
\end{align}
The quantities \eqref{deru1}, \eqref{deru2}, \eqref{deru3} can be treated similarly. Then, we focus on \eqref{deru1}. By using Cauchy Schwarz and the fact that $\gamma > 5/2$, we have 
\begin{align}
    \eqref{deru1} \lesssim \Big\|  s\|\jx^{\gamma -1}\varphi_\e(x)\|_{L^\infty_x}\|\jx^{-1}\partial_xu_{1}(s,x)\|_{L^2_x} \ \|u_{2}(s,x)\|_{L^\infty_x} \|u_{3}(s,x)\|_{L^\infty_x}\Big\|_{ L^2_s([0,t])} 
\end{align}
Lemma \ref{lem:localderivative}, combined with the bootstrap hypothesis \eqref{boothyp}, gives the decay 
\begin{align}
    \label{localdecayderu1}
    \|\jx^{-1}\partial_xu_{1}(s,x)\|_{L^2_x} \lesssim C_0\varepsilon \js^{-1/4} \ . 
\end{align}
Since $\|\jx^{\gamma -1}\varphi_\e(x)\|_{L^\infty_x} < \infty $ \eqref{muR2}, we combined the decays \eqref{localdecayderu1} and \eqref{decayujagain} to get 
the desired estimate.

For \eqref{derphi}, since $\partial_x \varphi_\e \in L^1$ \eqref{muR2}, we have 
\begin{align}
    \eqref{derphi} \lesssim \Big\|  s\|u_{1}(s,x)\|_{L^\infty_x} \ \|u_{2}(s,x)\|_{L^\infty_x} \|u_{3}(s,x)\|_{L^\infty_x}\Big\|_{ L^2_s([0,t])} \ . 
\end{align}
The decay \eqref{decayujagain} gives the desired estimate.

\subsubsection*{Step II: Estimation of $II^\e$}
Recalling \eqref{I2e}, it suffices to estimate 
\begin{align}
    \label{IIb}
    \int_0^t \Big\| \mathcal{T}_{\mathrm{b}, \e}(f_1,f_2,f_3)(s,k)\Big\|_{L^2_k} \ ds
\end{align}
where $\mathrm{b}(\xi) = \what{\Fou}(x^\alpha\varphi_\e)$ for $\alpha =0,1$. In these cases, since $\what{\Fou}^{-1}(\mathrm{b)} =  x^\alpha\varphi_\e \in L^\infty$ , Corollary \ref{cor:estimTbeps} gives 
\begin{align}
\eqref{IIb} \lesssim \int_0^t \| \mathrm{u}_1(s,x)\|_{L^\infty_x}\| \mathrm{u}_2(s,x)\|_{L^\infty_x} \|f_3(s,k)\|_{L^2_k} \ ds
\end{align}
where $\mathrm{u}_j(s,x) := \what{\Fou}^{-1}[k \mapsto e^{isk^2}f_j(k) ]$. Moreover, we have the bounds 
\begin{align}
    \label{decayvj}
    \| \mathrm{u}_j(s,x) \|_{L^\infty_x} \lesssim C_0\varepsilon \js^{-1/2} \  \quad , \quad \|f_3(s,k)\|_{L^2_k} \lesssim C_0\varepsilon
\end{align}
Thus, we obtain 
\begin{align}
    \eqref{IIb} \lesssim \int_0^tC_0^3\varepsilon^3\frac{1}{\js} \ ds \lesssim C_0^3\varepsilon^3 \log \jt \ . 
\end{align}

\end{proof}

%-------------------------------------------------------
\section{\texorpdfstring{$L^\infty$}{L-infinity} estimates }\label{secnonlinlinf}

In this section, we perform the pointwise estimates for the different nonlinear interactions. We treat separately the regular and the singular part. 
 
\subsection{Estimates for the regular part}\label{ssecRasy}
We begin by showing that the regular part only contributes to the lower order reminder terms in $\partial_tf^\sharp$.

%This follows from the localization properties in the $x$ variable of $\mathcal{K}_R$ 
%and $\mathcal{K}_+(x,k_0) \mathcal{K}_+(x,k_1)$. 

%Given the localization from the regular part, we have the following estimate:
\begin{proposition}\label{prop:asyregular}
Under the bootstrap hypothesis \eqref{boothyp},  it holds that for every $|t| \le T$,
\begin{equation}\label{ODEregular}
{\| \mathcal{N}_{R,1}(t,k) \|}_{L^{\infty}_k } \lesssim C_0^3\varepsilon^3 \jt^{-3/2}.
\end{equation}
\end{proposition}

\begin{proof}
We recall the decomposition performed at the beginning of the proof of Proposition \ref{l2reg1}. Thus, it suffices to estimate the nonlinear terms $\mathcal{N}_{R,1}^{(\bullet)}(t,k)$ for $\bullet = 1,2$ (see \eqref{NR1bull}). 
% According to the structure of the regular part of the measure \eqref{muR0}, it suffices to estimate the nonlinear terms $\mathcal{N}_{R,1}$ and $\mathcal{N}_{R,2}$ 
% associated with $\mu_{R,1}^\sharp$ and $\mu_{R,2}^\sharp$ respectively, as in \eqref{muR1}-\eqref{muR2}.

% We first consider the contribution from $\mathcal{N}_{R,1}$.  
% By definition, $\mu_{R,1}^\sharp$, as in \eqref{eq:mur11}-\eqref{eq:mur12}, is a linear combination of the following two types of terms,
% up to similar ones obtained by permuting the variables:
% \begin{align}
% \mu_{R,1}^{\sharp,(1)}(k,\ell,m,n) & :=
%  \int\overline{\mathcal{K}^\sharp_{R}(x,k)}
%   \mathcal{K}^\sharp_{M_{1}}(x,\ell)\overline{\mathcal{K}^\sharp_{M_{2}}(x,m)}\mathcal{K}^\sharp_{M_{3}}(x,n)\,dx,
% \\
% \mu_{R,1}^{\sharp,(2)}(k,\ell,m,n) & := 
% \int\overline{\mathcal{K}^\sharp_{S}(x,k)}\mathcal{K}^\sharp_{M_{1}}(x,\ell)
%   \overline{\mathcal{K}^\sharp_{M_{2}}(x,m)}\mathcal{K}^\sharp_{R}(x,n)\,dx,
%   \qquad M_{i}\in\{S,R\}.
% \end{align}
% Consider the nonlinear terms associated to the measures above:
% \begin{align*}
%     \mathcal{N}^{(1)}_{R,1}\left(t,k\right)=\iiint e^{it\left(-k^{2}+\ell^{2}-m^{2}+n^{2}\right)}
% f^\sharp(t,\ell)\overline{f^\sharp(t,m)}f^\sharp(t,n)\,\mu_{R,1}^{\sharp,(1)}(k,\ell,m,n) \, \frac{dn}{2\pi}\frac{dm}{2\pi}\frac{d\ell}{2\pi},
%   \\
%    \mathcal{N}^{(2)}_{R,1}\left(t,k\right)=\iiint e^{it\left(-k^{2}+\ell^{2}-m^{2}+n^{2}\right)}
%   f^\sharp(t,\ell)\overline{f^\sharp(t,m)}f^\sharp(t,n)\,\mu_{R,1}^{\sharp,(2)}(k,\ell,m,n) \, \frac{dn}{2\pi}\frac{dm}{2\pi}\frac{d\ell}{2\pi}.
% \end{align*}
Using the notation of \eqref{mur1(1)}-\eqref{mur1(2)}, we have 
\begin{align*}
\mathcal{N}^{(1)}_{R,1}(t,k) & = e^{-itk^2} \int\overline{\mathcal{K}^\sharp_{R}(x,k)} u_{A_1}(t,x)
  \overline{u_{A_2}}(t,x) u_{A_3}(t,x)\,dx \ , \\
  \mathcal{N}^{(2)}_{R,1}(t,k) & = e^{-itk^2} \int\overline{\mathcal{K}^\sharp_{S}(x,k)} u_{A_1'}(t,x)
  \overline{u_{A_2'}}(t,x) u_{R}(t,x)\,dx.
\end{align*}
On one hand, due to Lemma \ref{estimkR}, we have 
$$ \underset{k\in \R}{\sup}\int |\mathcal{K}_{R}(x,k)| \ dx < \infty \ .  $$
Then, we have
\begin{align*}
 {\big\| \mathcal{N}^{(1)}_{R,1}\left(t, k \right) \big\|}_{L^{\infty}_k}
 \lesssim 
 \left\Vert  u_{A_1}(t,x) \overline{u_{A_2}}(t,x) u_{A_3}(t,x)\right\Vert_{L^{\infty}_x} 
 \ . 
\end{align*}
Invoking the dispersive decay \eqref{decayuAj}, we get the desired estimate.

On the other hand, since $\mathcal{K}_{S}(x,k)$ is bounded, we can write :
\begin{align*}
 \left\Vert  \mathcal{N}^{(2)}_{R,1}\left(t,k \right)\right\Vert _{L^{\infty}_k}
 \lesssim \left\Vert  u_{M_1}(t,x) \overline{u_{M_2}}(t,x){\langle x \rangle}^{\gamma-1}
 u_{R}(t,x)\right\Vert _{L^{\infty}_x} \ . 
\end{align*}
Due to the decays \eqref{decayuAj} and  \eqref{decayuRagain}, we get the desired estimate.
% \begin{align}\label{decayuRagain}
%     \|\jx^{\gamma-1}u_{R}(s,x)\|_{L^\infty_x} \lesssim C_0\varepsilon \js^{-1/2} \ .
% \end{align}
% the last inequality follows from \eqref{eq:singulardecay} and \eqref{regularinfty}.
% %and the same  notations as \eqref{eq:notaasy}.
\end{proof}

\begin{proposition}\label{prop:asyregular2}
Under the bootstrap hypothesis \eqref{boothyp},  it holds that for every $|t| \le T$,
\begin{equation}\label{ODEregular2}
{\| \mathcal{N}_{R,2}(t,k) \|}_{L^{\infty}_k } \lesssim C_0^3\varepsilon^3 \jt^{-3/2}.
\end{equation}
\end{proposition}
\begin{proof}
We recall the beginning of the proof of Proposition \ref{l2boundreg2}. Thus, it suffices to estimate $\mathcal{N}_{R,2}^{\eps}(t,k)$ defined at \eqref{nr2eps}. Due to \eqref{mur2eps}, we can write 
% Finally, we consider the contribution of the nonlinear terms associated with $\mu^{\#}_{R,2}(k,\ell,m,n)$ given by a linear combination of terms
% 	%, see \eqref{muR2}, we know a generic term in from $\mu_{R,2}$ is
% 	of the form (we omit the dependence on $\eps_0,\eps_1,\dots$ in the notation)
% 	\begin{align}
% 	\begin{split} 
% 		%In the proof of the estimate \eqref{weightmain1} we will use the commutation Lemma \ref{lem:algetri} with	
% 	\mu_{R,2}^{\#,(0)}(k,\ell,m,n) := 
%  \overline{a_{\e_0}(k)} a_{\e_1}(\ell) \overline{a_{\e_2}(m)} a_{\e_3}(n)
% 	\int\varPsi(x)e^{i(-\e_0 k+\e_1\ell-\e_2 m+\e_3n)x}\,dx,
% 	\\
% 	|\partial_x^\alpha \varPsi(x)| \lesssim \jx^{-\gamma+1}, \quad \alpha =0,1.
% 	\end{split}
% 	\end{align}  
% Therefore the terms we are going to estimate to conclude the proof of the proposition take the form
% \begin{align*}
%     \mathcal{N}_{R,2}^{(0)}(t,k)& = e^{-itk^2}\int\int\int e^{it(\ell^2-m^2+n^2)}f^\sharp(\ell)\overline{f^\sharp(m)}f^\sharp(n) \mu_{R,2}^{\#,(0)}(k,\ell,m,n) \ \frac{dn}{2\pi}\frac{dm}{2\pi}\frac{d\ell}{2\pi} \\
%     \mathcal{N}_{R,2}^{(0)}(t,k) & =  \overline{a_{\e_0}(k)} e^{-itk^2} %(\sqrt{2\pi})^3
% 	\int_{\R_x}\varPsi(x)e^{-\e_0 ikx} u_{\e_1}(x) 
% 	\overline{u_{\e_2}(x)} u_{\e_3}(x)\, dx
% \end{align*}
\begin{align}
    \mathcal{N}_{R,2}^{\eps}(t,k) & =  \overline{a_{\e_0}(k)} e^{-itk^2} 
	\int_{\R_x}e^{-\e_0 ikx}\varphi_\eps(x) u_{1}(t,x) 
 	\overline{u_{2}(t,-x)} u_{3}(t,x)\, dx
\end{align}
where $u_{j}(t,x)$ is defined at \eqref{defu_j}.  
Due to the boundedness of the coefficients $a_{\e_0}(\cdot)$ and the $L^1$ integrability of $\varphi_{\eps}$, we get 
\begin{align}
    \| \mathcal{N}_{R,2}^{\eps}
(t,k) \|_{L_k^{\infty}} \lesssim \|u_{\e_1}(t,x) 
 	\overline{u_{\e_2}(t,-x)} u_{\e_3}(t,x) \|_{L^\infty_x} \ . 
\end{align}
Then, due to the decay \eqref{decayujagain}, we get the desired estimate. 
\end{proof}

%%%%%%%%%%%%%%%%%%%%%%%%%%%%
%%%%%%%%%%%%%%%%%%%%%%%%%%%%
%%%%%%%%%%%%%%%%%%%%%%%%%%%%

\smallskip
\subsection{Asymptotics of singular part}\label{ssecSasy}
It remains to bound $\mathcal{N}_{S}\left(t,k\right)$.
%which is the same as the singular part in \cite{NLSV}. 
We start with the following standard stationary phase-type lemma adapted from \cite[Section 5]{GPR}.

\begin{lemma}%[Lemma ... \cite{GPR}]
\label{AsLem1}
For $k,t \in \R$ and $\eps = (\e_0,\e_1,\e_2,\e_3) \in \{+,-\}^4$, we consider the integral expression
\begin{align}\label{I1}
\begin{split}
I_{\eps}[g_1,g_2,g_3](t,k) & = \underset{\R^3}{\iiint} e^{it \Phi_\e(k,p,m,n)}
	g_1(\epsilon_1(\epsilon_0 k - p + \epsilon_2 m - \epsilon_3 n)) \overline{g_2(m)} 
	g_3(n) \pv \frac{\widehat{\zeta}(p)}{p} \, dm dn dp
\\
\Phi_\e(k,p,m,n) & = -k^2 + (\epsilon_0 k - p + \epsilon_2 m - \epsilon_3 n)^2 - m^2 + n^2,
\end{split}
\end{align}
for $\zeta$ an even smooth function with compact support and integral $1$, 
and with $g:=(g_1,g_2,g_3)$ satisfying
\begin{align}
\label{AsLem1as}
{\| g(k) \|}_{L^\infty_k} + {\| \langle k \rangle g(k) \|}_{L^2_k} + \langle t \rangle^{-1/4} {\| \partial_kg(k) \|}_{L^2_k} \leq 1 \ . 
\end{align} 
Then, for any $k,t \in \R$,
\begin{align}\label{I2}
\begin{split}
I_{\eps}[g_1,g_2,g_3](t,k) = \frac{\pi}{|t|} e^{-itk^2} \int_\R e^{it(-p+\epsilon_0 k)^2} g_1(\epsilon_1(-p+\epsilon_0 k))
	\overline{g_2(\epsilon_2(-p+\epsilon_0 k))}
	\\ \times  g_3(\epsilon_3 (-p+\epsilon_0 k)) \mathrm{p.v.} \frac{\widehat{\zeta}(p)}{p} \,dp
	+ \mathcal{O}(|t|^{-1}) \ . 
\end{split}
\end{align}

%For %$|k\sqrt{t}|\gtrsim t^{1/4}$, 
%$|k| \geq |t|^{-3\alpha}$ we can further simplify \eqref{I2} to
%\begin{align}\label{I3}
%\begin{split}
%I[g_1,g_2,g_3](t,k) & =-i\frac{\pi}{|t|}
%  \sqrt{\frac{\pi}{2}} \cdot \mathrm{sign}(\e_0k t) \cdot g_1(\epsilon_1\epsilon_0 k) 
%  \overline{g_2(\epsilon_2\epsilon_0 k))} g_3(\epsilon_3\epsilon_0 k)+  \mathcal{O}(|t|^{-1-\rho}).
%\end{split}
%\end{align}
\end{lemma}

\begin{proof}
See Lemma 5.1 in \cite{GPR}. %For the proof of \eqref{I3}, see \cite[Lemma 6.3]{NLSV}.
One should note that the proof in \cite{GPR} still works for $\alpha =0 $, leading to the above lemma.
\end{proof}

The following lemma will be of use for the $\delta_0$ terms. 

\begin{lemma} \label{Asdirac}
    For $k, t \in \mathbb{R}$ and $\eps = (\e_0,\e_1,\e_2,\e_3) \in \{+,-\}^4$, we consider the integral expression
\begin{equation}\label{eq:L}
L_{\eps}[g_1, g_2, g_3](t, k) = 
\underset{\R^2}{\iint} e^{i t \Phi_\e(k,p,m,n)} 
\, g_1\big(\e_1(\e_0 k + \e_2 m - \epsilon_3 n)\big) 
\, g_2(m) \, g_3(n) \, dm \, dn,
\end{equation}
with
\[
\Phi_\e(k, p, m, n) = -k^2 + (\e_0 k + \e_2 m - \epsilon_3 n)^2 - m^2 + n^2,
\]
and $g := (g_1, g_2, g_3)$ satisfying
\begin{equation}\label{eq:g-norm}
{\| g(k) \|}_{L^\infty_k} + {\| \langle k \rangle g(k) \|}_{L^2_k} + \langle t \rangle^{-1/4} {\| \partial_kg(k) \|}_{L^2_k} \leq 1 \ . 
\end{equation}
Then, for any $k,t \in \mathbb{R}$,
\begin{equation}\label{eq:asymptotic}
L_{\eps}[g_1, g_2, g_3](t, k) = 
\frac{\pi}{|t|} g_1(\e_1 \e_0 k)\, g_2(\e_2 \e_0 k)\, g_3(\epsilon_3 \e_0 k) 
+ \mathcal{O}\big(|t|^{-1}\big) \ .
\end{equation}
\end{lemma}

\begin{proof}
See Lemma 5.2 in \cite{GPR}. %For the proof of \eqref{I3}, see \cite[Lemma 6.3]{NLSV}.
One  should note that the proof in \cite{GPR} still works for $\alpha =0 $, leading to the above lemma.
\end{proof}

\begin{remark}\label{absorbmainterm}
    The reader could observe that in \eqref{eq:asymptotic}, $L_{\eps}[g_1, g_2, g_3](t, k) = \mathcal{O}\big(|t|^{-1}\big)$. Indeed, according to assumptions \eqref{eq:g-norm}, the main term in \eqref{eq:asymptotic} can be absorbed in the $\mathcal{O}\big(|t|^{-1}\big)$ remainder. Nevertheless, we stated Lemma \ref{Asdirac} as above to stay close to the formulation in \cite[Lemma 5.2]{GPR}. In fact, the asymptotic in \cite{GPR} gives 
    $$L_{\eps}[g_1, g_2, g_3](t, k) = 
\frac{\pi}{|t|} g_1(\e_1 \e_0 k)\, g_2(\e_2 \e_0 k)\, g_3(\epsilon_3 \e_0 k) 
+ \mathcal{O}\big(|t|^{-1-\alpha/3}\big) \ ,$$
under the assumptions
$$ {\| g(k) \|}_{L^\infty_k} + {\| \langle k \rangle g(k) \|}_{L^2_k} + \langle t \rangle^{-1/4+\alpha} {\| \partial_kg(k) \|}_{L^2_k} \leq 1 \ . $$
\end{remark}

%Unlike in Chen-Pusateri \cite{NLSV}, in the non-genereic setting, 
%we cannot use $\wt{f}(t,0)=0$ for all $t$.
To obtain bounds for expressions like \eqref{I2}, 
we shall use the following lemma: %in order to obtain its leading order behavior.

\begin{lemma}\label{lemStat}
For $K\in \mathbb{R}$, $t > 0$, consider the principal value integral  expression
\begin{equation}\label{eq:I}
I(t, K) = \pv \int_\R 
e^{i t x^{2}} \, G(x) \, \frac{\psi(x - K)}{x - K} \, dx,
\end{equation}
for $\psi \in \mathcal{S}$, and $G$ satisfying
\begin{equation}\label{eq:g-cond}
\|G(x)\|_{L^\infty_x} + \langle t \rangle^{-1/4} \|\partial_xG(x)\|_{L^2_x} \leq 1 \ . 
\end{equation}
Then, for $K \in \R$, $t > 0$, 
% \begin{align}\label{sgnpvasymp}
%     I(t,K) = i\pi G(K)\psi(0)e^{itK^2} \sgn(2tK) + O(1) \ . 
% \end{align}
$I(t,K)$ is uniformly bounded. 
\end{lemma}

\begin{proof}In what follows, we will often omit the p.v. notation where it is understood. We consider $\varphi$ a smooth even cutoff function with support in $(-1,1)$. We define the quantities 
    \begin{align}
        A &= \int_\R e^{i t x^{2}} \, G(x) \, \frac{\psi(x - K)}{x - K}\varphi(|x-K|t^{1/2}) \, dx \ , \\
        B &=  \int_\R e^{i t x^{2}} \, G(x) \, \frac{\psi(x - K)}{x - K}\Big[1 - \varphi(|x-K|t^{1/2}) \Big] \, dx  \ ,  
    \end{align}
such that $ I = A + B $. 
For the first term, we have
\begin{align*}
&\left| A - G(K) 
   \int_\R e^{i t x^{2}}\frac{ \psi(x - K)}{x - K} 
   \, \varphi\!\left(|x - K| t^{1/2}\right) dx 
 \right| \\
&\qquad \lesssim \int_\R
   \frac{|G(x) - G(K)| \, |\psi(x - K)|}{|x - K|} 
   \, \varphi\!\left(|x - K| t^{1/2 }\right) dx \\
&\qquad \lesssim \int_\R
   \int_x^K |G'(s)| \ ds \, \frac{|\psi(x - K)|}{|x - K|} 
   \, \varphi\!\left(|x - K| t^{1/2 }\right) dx  \\
&\qquad \lesssim \|G'\|_{L^2} 
   \int_\R\frac{\varphi\!\left(|x - K| t^{1/2}\right)}{\sqrt{|x - K|}} \, dx \quad \quad (\text{by Cauchy-Schwarz})\\
&\qquad \lesssim t^{1/4}
   \int_\R\frac{\varphi\!\left(|x|\right)}{\sqrt{|x|t^{-1/2}}} \, \frac{dx}{t^{1/2}} \\
&\qquad \lesssim  1 \ . 
\end{align*}

For the second term, we define the quantities
\begin{align}
B_1 &= \int_\R e^{i t x^{2}} 
G(x)\frac{\psi(x-K)}{x-K}
\Big[1 - \varphi\!\left(|x-K|t^{1/2}\right)\Big]
\, \varphi\!\left(|x| t^{1/2}\right) dx \ , \nonumber \\
B_2 &= \int_\R e^{i t x^{2}} 
G(x)\frac{\psi(x-K)}{x-K}
\Big[1 - \varphi\!\left(|x-K|t^{1/2}\right)\Big]
\Big[1 - \varphi\!\left(|x| t^{1/2 }\right)\Big] dx \ ,  \label{eq:B-decomp}
\end{align}
such that $B = B_1 + B_2$. We can see directly that $B_1$ is an acceptable remainder. Indeed, 
\begin{align*}
|B_1| &\lesssim \|g\|_{L^\infty} 
\int_\R\frac{1}{|x-K|} 
\Big[1 - \varphi\!\left(|x-K|t^{1/2 }\right)\Big] 
\varphi\!\left(|x| t^{1/2 }\right) dx \\
&\lesssim \|g\|_{L^\infty} 
\int_{|x|<t^{-1/2} ; |x-K|>t^{-1/2}} \frac{1}{|x-K|} 
\Big[1 - \varphi\!\left(|x-K|t^{1/2 }\right)\Big] 
\varphi\!\left(|x| t^{1/2}\right) dx \\
&\lesssim t^{1/2} \, t^{-1/2} 
\lesssim 1 \ . 
\end{align*}

For $B_2$, notice that we are away from the singularity of the integrand as well as from the stationary point $x=0$. We can then integrate by parts in $x$ to show this is also a remainder. This yields
\begin{align*}
|B_2| &= \left| \int_\R\frac{1}{t} e^{i t x^{2}} 
\partial_x \left( \frac{1}{2x} \, G(x) \frac{\psi(x-K)}{x-K} 
\Big[1 - \varphi\!\left(|x-K|t^{1/2 }\right)\Big] 
\Big[1 - \varphi\!\left(|x| t^{1/2}\right)\Big] \right) dx \right| \nonumber \\
&\lesssim \frac{1}{t} \big(C_1 + C_2 + C_3\big) \ , 
\end{align*}

where
\begin{align}
|C_1| &\lesssim \int_\R\frac{1}{|x|} |G'(x)| \frac{1}{|x-K|}
\Big[1 - \varphi\!\left(|x-K|t^{1/2}\right)\Big]
\Big[1 - \varphi\!\left(|x| t^{1/2 }\right)\Big] dx \ , \label{eq:C1} \\
|C_2| &\lesssim \int_\R\frac{|G(x)|}{|x|} 
\left|\partial_x\!\left( \frac{\psi(x-K)}{x-K} 
\Big[1 - \varphi\!\left(|x-K|t^{1/2}\right)\Big]\right)\right|
\Big[1 - \varphi\!\left(|x| t^{1/2}\right)\Big] dx \ , \label{eq:C2} \\
|C_3| &\lesssim \int_\R\frac{|G(x)|}{|x-K|} 
\Big[1 - \varphi\!\left(|x-K|t^{1/2}\right)\Big] 
\left|\partial_x\!\left( \frac{1}{x}\big(1 - \varphi(|x| t^{1/2})\big)\right)\right| dx \ . \label{eq:C3}
\end{align}

We can bound the first term by
\begin{align*}
|C_1| &\lesssim \|G'\|_{L^2} \, t^{1/2} 
\left(\int_\R\frac{1}{|x|^2} 
\Big[1 - \varphi\!\left(|x| t^{1/2}\right)\Big] dx \right)^{1/2} \\
&\lesssim t^{1/4} \cdot t^{1/2} \cdot t^{1/4} \left(\int_\R\frac{1 - \varphi(|x|)}{|x|^2} 
dx \right)^{1/2} \\
&\lesssim t \ . 
\end{align*}

We can estimate the second term by
\begin{align*}
|C_2| &\lesssim \|G\|_{L^\infty} \, t^{1/2} 
\int_\R\left| \partial_x\!\left( \frac{\psi(x-K)}{x-K} 
\Big[1 - \varphi\!\left(|x-K|t^{1/2 }\right)\Big]\right)\right| dx \\
&\lesssim \|G\|_{L^\infty} \, t^{1/2} \int_\R\Bigg\{\frac{|\partial_x \psi(x-K)|}{|x-K|} \Big[1 - \varphi\!\left(|x-K|t^{1/2 }\right)\Big] \  + \frac{|\psi(x-K)|}{|x-K|^2} \Big[1 - \varphi\!\left(|x-K|t^{1/2 }\right)\Big] \\ 
&  \hspace{6cm} + \frac{|\psi(x-K)|}{|x-K|}t^{1/2} \Big|\partial_x\varphi\!\left(|x-K|t^{1/2}\right)\Big| \Bigg\} \ dx  \ . 
\end{align*}
The first term of the integral is controlled using the cut-off function and the scale $|x-K| > t^{-1/2}$, while the other terms get controlled using changes of variables. It yields 
\begin{align*}
|C_2| \lesssim t^{1/2}\cdot ( t^{1/2} + t^{1/2} + t^{1/2}) \lesssim t
\end{align*}
\\

Finally, $C_3$ can be bounded similarly. Indeed, 
\begin{align*}
    |C_3| &\lesssim \|G\|_{L^\infty} \, t^{1/2} 
\int_\R\left| \partial_x\!\Big( \frac{1}{x}\big[1 - \varphi\big(|x| t^{1/2})\big]\Big) \right| dx \\
 &\lesssim \|G\|_{L^\infty} \, t^{1/2} 
\int_\R\Bigg\{\frac{1}{|x|^2}\big[1 - \varphi\big(|x| t^{1/2})\big] + \frac{t^{1/2}}{|x|}\big|\partial_x\varphi(|x|t^{1/2}) \big| \Bigg\} dx \\
& \lesssim t
\end{align*}
we get the desired estimate by performing changes of variables.\\

Thus, at this stage we can write 
\begin{align*}
    I(t,K) = G(K)  
   \int_\R e^{i t x^{2}} \varphi\!\left(|x - K| t^{1/2}\right) \ \pv \frac{ \psi(x - K)}{x - K} 
   \,  dx + \mathcal{O}(1) \ . 
\end{align*}
We notice that we can replace $\psi(x-K)$ by $\psi(0)$. We get 
\begin{align*}
    I(t,K) = G(K)\psi(0)  
   \int_\R e^{i t x^{2}} \varphi\!\left(|x - K| t^{1/2}\right) \ \pv \frac{ 1}{x - K} 
   \,  dx + \mathcal{O}(1) \ . 
\end{align*}
Then, performing the change of variables $x \mapsto x + K$, we get 
\begin{align}
    I(t,K) = G(K)\psi(0)e^{itK^2} 
   \int_\R e^{i t x^{2}} e^{2itxK} \varphi\!\left(|x| t^{1/2}\right) \ \pv \frac{ 1}{x} 
   \,  dx + \mathcal{O}(1) \ . 
   \label{I x^2+xK vp}
\end{align}
Using the fact that 
\begin{align*}
    e^{itx^2} - 1 = itx^2\int_0^1 e^{itx^2s} \ ds \ ,  
\end{align*}
we can write 
\begin{align*}
     \int_\R e^{i t x^{2}} e^{2itxK} \varphi\!\left(|x| t^{1/2}\right) \ \pv \frac{ 1}{x} 
   \,  dx  &=  \int_\R e^{2itxK} \varphi\!\left(|x| t^{1/2}\right) \ \pv \frac{ 1}{x} 
   \,  dx   \\
   & \hspace{2cm}+ it \int_\R e^{2itxK} \varphi\!\left(|x| t^{1/2}\right) x\int_0^1 e^{itx^2s} \ ds  
   \,  dx \ . 
\end{align*}
We observe that the second term has the following growth
$$t\int_0^{t^{-1/2}} x \ dx \lesssim 1 \ . $$
Therefore, from \eqref{I x^2+xK vp}, we can write 
\begin{align}
    I(t,K) & = G(K)\psi(0)e^{itK^2}  \int_\R e^{2itxK} \varphi\!\left(|x| t^{1/2}\right) \ \pv \frac{ 1}{x} 
   \,  dx + \mathcal{O}(1)\\
     & = G(K)\psi(0)e^{itK^2} \Fou{ \Big(\varphi(|\cdot|t^{1/2})\pv \frac{1}{x} \Big)}(-2tK) + \mathcal{O}(1) \\
     & = -\frac{i}{2} G(K)\psi(0)e^{itK^2} \Big(\Fou \big(\varphi(|\cdot|t^{1/2}) \big)  \ast \sgn \Big)(-2tK) + \mathcal{O}(1) \ . 
    \label{I conv}
\end{align}
By observing that $\Fou \big(\varphi(|\cdot|t^{1/2}) \big)(\xi) = t^{-1/2}\widehat{\varphi}(\xi t^{-1/2})$, we can write 
\begin{align*}
    \Big(\Fou \big(\varphi(|\cdot|t^{1/2}) \big)  \ast \sgn \Big)(-2tK)  & = \int_\R t^{-1/2}\widehat{\varphi}(\xi t^{-1/2}) \ \sgn(-2tK - \xi) \ d\xi \\
    & = \sgn(-2tK)\int_\R t^{-1/2}\widehat{\varphi}(\xi t^{-1/2}) \ d\xi \\
    & \quad \quad \quad + \int_\R t^{-1/2}\widehat{\varphi}(\xi t^{-1/2}) \Big[\sgn(-2tK - \xi) - \sgn(-2tK) \Big] \ d\xi \\
    & = -2\pi \ \sgn(2tK)  + \mathcal{O}( \|\widehat{\varphi}\|_{L^1}) \ ,
\end{align*}
where for the first term we perform a change of variables and use the fact that $\int_\R \widehat{\varphi}(\xi) \ d\xi = 2\pi \varphi(0) = 2\pi$. We get 
\begin{align}\label{sgnpvasymp}
    I(t,K) = i\pi G(K)\psi(0)e^{itK^2} \sgn(2tK) + \mathcal{O}(1) \ ,
\end{align}
leading to the desired bound.
% Finally, by plugging the last estimate in \eqref{I conv}, we get the desired bound. 
\end{proof}
We state the bounds for the nonlinear singular interactions.
\medskip
% \begin{proposition}\label{prop:asydirac}
% Under the bootstrap hypothesis \eqref{boothyp},  it holds that for every $|t| \le T$,
% \begin{equation}\label{ODEdirac}
% {\| \mathcal{N}_{0}\left(t,k \right) \|}_{L^{\infty}} \lesssim C_0^3\varepsilon^3 \jt^{-1}.
% \end{equation}
% \end{proposition}

% \begin{proof}
%     We recall the expression of $\mu_0^\#$ in \eqref{mu0}. Using the notations of Lemma \ref{Asdirac}, we see that $\mathcal{N}_0(t,k)$ is a finite linear combination of terms of the form $L_\eps[f^\sharp,f^\sharp,f^\sharp](t,k)$. Thus, it suffices to estimate such quantities. We apply Lemma \ref{Asdirac} with $g_j = \frac{1}{C_0\varepsilon}f^\sharp$. The assumption \eqref{eq:g-norm} is satisfied due to the bootstrap assumption \eqref{boothyp}. Then, we get 
%     $$ \Big\|L_\eps\Big[\frac{1}{C_0\varepsilon}f^\sharp,\frac{1}{C_0\varepsilon}f^\sharp,\frac{1}{C_0\varepsilon}f^\sharp\Big](t,k)\Big\|_{L^\infty_k} \lesssim \jt^{-1} \ . $$
%     The desired estimate follows from the fact that
%     $$ L_\eps[f^\sharp,f^\sharp,f^\sharp](t,k) = C_0^3\varepsilon^3 L_\eps\Big[\frac{1}{C_0\varepsilon}f^\sharp,\frac{1}{C_0\varepsilon}f^\sharp,\frac{1}{C_0\varepsilon}f^\sharp\Big](t,k) $$
%     since $L_\eps$ is a trilinear form.
% \end{proof}

\begin{proposition}\label{prop:asypv}
Under the bootstrap hypothesis \eqref{boothyp},  it holds that for every $|t| \le T$,
\begin{equation}\label{ODEpv}
\| \mathcal{N}_{0}\left(t,k \right) \|_{L^{\infty}_k} + \| \mathcal{N}_{\pv}\left(t,k \right) \|_{L^{\infty}_k} + \| \mathcal{N}_{L}\left(t,k \right) \|_{L^{\infty}_k} \lesssim C_0^3\varepsilon^3 \jt^{-1}.
\end{equation}
\end{proposition}

\begin{proof} We recall the expression of $\mu^\sharp_0$, $\mu^\sharp_\pv$, $\mu^\sharp_L$ given at \eqref{gensingpart}-\eqref{bform}. We use the respective notations of Lemmas \ref{Asdirac} and \ref{AsLem1}. Then, it suffices to bound expressions of the form  
\begin{align}
    L_\eps[f_1,f_2,f_3](t,k) \quad \text{ and } \quad I_\eps[f_1,f_2,f_3](t,k) \ ,
\end{align}
where $f_j = a_{\e_j}f^\sharp$. 

For $L_\e$, we apply Lemma \ref{Asdirac} with $g_j = \frac{1}{C_0\varepsilon}f_j$. The assumption \eqref{eq:g-norm} is satisfied due to the bootstrap assumption \eqref{boothyp}. Then, we get 
    $$ \|L_\eps[g_1,g_2,g_3](t,k)\|_{L^\infty_k} \lesssim \jt^{-1} \ . $$
    The desired estimate follows from the fact that
    $$ L_\eps[f_1,f_2,f_3](t,k) = C_0^3\varepsilon^3 L_\eps[g_1,g_2,g_3](t,k) \ , $$
    since $L_\eps$ is a trilinear form. 

For $I_\e$, we apply Lemma \ref{AsLem1} with $g_j = \frac{1}{C_0\varepsilon}f_j$. We get 
\begin{align}
    I_\eps[g_1,g_2,g_3](t,k) = \frac{\pi}{|t|} e^{-itk^2} \int_\R e^{it(-p+\epsilon_0 k)^2} g_1(\epsilon_1(-p+\epsilon_0 k))
	\overline{g_2(\epsilon_2(-p+\epsilon_0 k))}
	\\ \times  g_3(\epsilon_3 (-p+\epsilon_0 k)) \mathrm{p.v.} \frac{\widehat{\zeta}(p)}{p} \,dp
	+ \mathcal{O}(|t|^{-1}) \ . 
\end{align}
Then, we apply Lemma \ref{lemStat} with $\psi = \widehat{\zeta}$ and 
\begin{align*}
    G(x) = g_1(\e_1x)\overline{g_2(\e_2x)}g_3(\e_3x) \ . 
\end{align*}
It yields
$$ I_\eps[g_1,g_2,g_3](t,k) \lesssim 1 \ . $$
The desired estimate follows from the fact that
    $$ I_\eps[f_1,f_2,f_3](t,k) = C_0^3\varepsilon^3 I_\eps[g_1,g_2,g_3](t,k) \ , $$
    since $I_\eps$ is a trilinear form.

    % First of all, recalling the expression of $\mu_L^\#$ in \eqref{muL}, one can notice that the $\delta_0$ terms that appear can be handled as above, by using lemma \ref{Asdirac} and setting $g_i(\ell) = a_{\eps_{i+1}}f^\#(\ell)$. It remains to treat the improved low frequency $\pv$ terms which take the form \eqref{I1} with $g_i(\ell) = a_{\eps_{i+1}}f^\#(\ell)$, up to bounded factor $a_1(k)$. Thus, by invoking lemma \ref{AsLem1}, we get asymptotics like \ref{I2}. then we use lemma \ref{lemStat} to get the desired estimate of $\mathcal{N}_L(t,k)$. 
    
    % We use the same strategy for  $\mathcal{N}_{\pv}(t,k)$, but by setting $g_i(\ell) = f^\#(\ell)$ while using the pointwise estimate lemmas \ref{AsLem1} and \ref{lemStat}. 
\end{proof}

%-------------------------------------------------------
\bigskip
\section{Bootstrap argument and concluding remarks}\label{propgoalproof}
In this section, we provide the bootstrap arguments leading to the proof of the main Theorem \ref{mainthm}. Next, we conclude with some remarks about the choice of bootstrap norms. 

\begin{proposition}\label{propgoal}
    Under the bootstrap hypothesis \eqref{boothyp}, for all $t \in [-T,T]$, the following estimates hold
    \begin{align} \label{dkf bound}
        {\big\| \partial_{k} (\mathcal{F}^\sharp f) (t) \big\|}_{L_k^2} & \le \frac{C_0}{4}\varepsilon +CC_0^3\varepsilon^3\jt^{1/4} \ , \\
        \label{f bound}
        {\big\| (\mathcal{F}^\sharp f) (t) \big\|}_{L_k^\infty} & \le  \frac{C_0}{4}\varepsilon +CC_0^3\varepsilon^3\log \jt \ . 
    \end{align}
    with $C$ an absolute constant. 
\end{proposition}
\begin{proof}
We start with the estimate \eqref{dkf bound}. Indeed, from corollary \ref{duhamelfdecomp}, we can write 
\begin{align}\label{eq:expandfl2}
	{\big\| \partial_{k}f^{\#}(t) \big\|}_{L_{k}^{2}} \leq {\big\| \partial_{k}f^{\#}(0) \big\|}_{L_k^2} + 
	{\Big\|\int_0^{t} \partial_{k}\big(\mathcal{N}_{0} + \mathcal{N}_{\pv} +\mathcal{N}_{L}
	+ \mathcal{N}_{R,1} + \mathcal{N}_{R,2} \big)\,ds \Big\|}_{L_k^2} \ . 
\end{align}
We know that $f(0) = u_0$. Then, using \eqref{eq:weiF}, we have 
 $$ {\big\| \partial_{k}f^{\#}(0) \big\|}_{L_k^2} \le \frac{C_0}{4}{\big\| \jx u_0 \big\|}_{L^2} \le \frac{C_0}{4}\varepsilon \ ,$$
where $C_0$ is a constant depending only on $V$. Therefore, the $L^2_k$ estimate \eqref{dkf bound} follows from \eqref{eq:expandfl2}, and the estimates of Propositions \ref{l2boundsing}, \ref{l2reg1} and \ref{l2boundreg2}. 

For the $L^\infty$ estimate \eqref{f bound}, we start with 
\begin{align}\label{eq:expandflinf}
	{\big\| f^{\#}(t) \big\|}_{L_{k}^{\infty}} &\leq {\big\| f^{\#}(0) \big\|}_{L_k^\infty} \\
    &+ 
	\int_0^{t} \big({\|\mathcal{N}_{0}\|}_{L_k^\infty} + {\|\mathcal{N}_{\pv}\|}_{L_k^\infty}+{\|\mathcal{N}_{L}\|}_{L_k^\infty}
	+ {\|\mathcal{N}_{R,1}\|}_{L_k^\infty} + {\|\mathcal{N}_{R,2}\|}_{L_k^\infty} \big)\,ds \ .  
\end{align}
Using \eqref{fsharp l1}, we see that  
$$ {\big\| f^{\#}(0) \big\|}_{L_k^\infty} \lesssim {\big\| u_0 \big\|}_{L^1} \le \frac{C_0}{4}{\big\| \jx u_0 \big\|}_{L^2} $$
with $C_0$ a constant depending only on $V$. Thus, the estimate \eqref{f bound} follows from \eqref{eq:expandflinf} and Propositions \ref{prop:asyregular} and \ref{prop:asypv}. 
\end{proof}
\begin{proof}[Proof of main theorem \ref{mainthm}] The goal is to prove that for all $t \in  [-T_\varepsilon, T_\varepsilon]$:
\begin{align}
{\big\| (\mathcal{F}^\sharp f) (t) \big\|}_{L_k^\infty} + \jt^{-1/4} 
  {\big\| \partial_{k} (\mathcal{F}^\sharp f) (t) \big\|}_{L_k^2} \lesssim \varepsilon \ . 
  \nonumber
\end{align}
 The continuity at $t = 0$ combined with the small initial data hypothesis \eqref{smalldata}, gives the existence of a time interval $[-T^* , T^*]$ on which the bound  is satisfied. Let $C_0$ be a constant depending only on $V$ to be determined. Then, we consider  $T > 0 $ such that 
\begin{align} \label{boothyp-again}
    \underset{t \in [-T,T]}{\sup}\Big({\big\| (\mathcal{F}^\sharp f) (t) \big\|}_{L_k^\infty} + \jt^{-1/4} 
  {\big\| \partial_{k} (\mathcal{F}^\sharp f) (t) \big\|}_{L_k^2} \Big) = 2C_0\varepsilon  \ . 
\end{align}
If, $T > T_\varepsilon$, Theorem \ref{mainthm} follows immediately from the linear dispersion estimate \eqref{eq:linearpoinwiseH}
\begin{align}
    \|u(t)\|_{L^\infty_x} \lesssim \jt^{-1/2}\Big({\big\| (\mathcal{F}^\sharp f) (t) \big\|}_{L_k^\infty} + \jt^{-1/4} 
  {\big\| \partial_{k} (\mathcal{F}^\sharp f) (t) \big\|}_{L_k^2} \Big) \ . 
\end{align}
Thus, we assume that $T \le T_\varepsilon$. Now, we take  $\varepsilon_0$ small enough to get $CC_0^3\varepsilon^2 \le C_0/4$. Moreover, the choice of $T_\varepsilon$ is such that $CC_0^3\varepsilon^2\log \jt \le C_0/4$. Thus, Proposition \ref{propgoal} gives the bound 
$$\underset{t \in [-T,T]}{\sup}\Big({\big\| (\mathcal{F}^\sharp f) (t) \big\|}_{L_k^\infty} + \jt^{-1/4} 
  {\big\| \partial_{k} (\mathcal{F}^\sharp f) (t) \big\|}_{L_k^2} \Big) \le C_0\varepsilon $$ 
  that contradicts \eqref{boothyp-again}. Therefore, $T > T_\varepsilon$. 
\end{proof}

\medskip
We are going to discuss some issues in the choice of bootstrap norms in order to get global bounds for non-generic potentials. For $\alpha, T > 0$, we consider the following bootstrap norms 
\begin{align}
    \|u(T)\|_{\mathcal{B}_\alpha} := \underset{t\in [-T,T]}{\sup} \ ( \|f^\sharp(t)\|_{L^\infty} + \jt^{-\alpha}  \|\partial_kf^\sharp(t)\|_{L^2} ) \ . 
\end{align}
In previous works such as \cite{GPR} and \cite{CPNLS2}, the authors used it with  $\alpha < 1/4$.

The starting point of the discussion is the technique we used (see Section \ref{secnonlinl2sing}) for the $L^2$ estimate of the singular parts involving $\pv$ terms, especially those without low-frequency improvements. This technique gives a threshold growth of $t^{1/4}$ no matter what the growth of $\| \partial_k f^\sharp \|_{L^2}$ is in the bootstrap hypothesis. Indeed,  let us consider the bootstrap hypothesis 
\begin{align}
     \|u(T)\|_{\mathcal{B}_\alpha}= 2C_0\varepsilon 
\end{align}
with $\alpha \le 1/4$. Using the machinery of Chen and Pusateri in \cite{CPNLS2}, due to the low-frequency structure, we succeed in controlling the $L^2$ norm of 
$$\int_0^t \partial_k \mathcal{N}_0(s,k) \ ds \quad \text{ and } \quad \int_0^t \partial_k \mathcal{N}_L(s,k) \ ds \ . $$
It gives a $t^\alpha$ growth (see \cite[Section 5]{CPNLS2}). On the other hand, for the dangerous $\pv$ part, we have the following asymptotic in $L^2_k$ 
        \begin{align}
            \Big\|\int_0^t \partial_k \mathcal{N_\pv}(s,k) ds -8(A_+^2 - A_-^2)A_+B_+ \int_0^t e^{-isk^2} s|\underline{u}(s,0)|^2\underline{u}(s,0) \ ds \Big\|_{L^2_k} = \mathcal{O}(C_0^3\varepsilon^3 t^\alpha) \ , 
        \end{align}   
        where $\underline{u} = \Fou^{-1}u^\sharp$  . The main term of this $L^2_k$ asymptotic is the reason why we cannot close the bootstrap scheme for $\alpha < 1/4$. In fact, after a stationary phase lemma, this is equivalent to estimating
        \begin{align} \label{refinedasymppv}
            \int_1^t \int_1^t \frac{ss'}{\sqrt{|s-s'|}} |\underline{u}(s,0)\underline{u}(s',0)|^2 \underline{u}(s,0)\overline{\underline{u}(s',0)}  \ dsds' \ . 
        \end{align}

    In addition to the almost global sharp decay, one can prove that the bootstrap norms used for the generic and symmetric non-generic case can be unbounded in the case of asymmetric non-generic potentials. In fact, this possibility is linked to the long time asymptotics at $x = 0$ of the inverse of Wave Operators applied to the solution of \eqref{NLS}. This statement is made clear through the following proposition.
    Since this result is included only for explanatory purposes and is not needed for the proof of our main theorem, we omit some details of the argument for the sake of brevity.
\begin{proposition}
    Let $V$ be an asymmetric non-generic potential. i.e. $V$ is non-generic with a zero energy resonance which has a limit different from $1$ or $-1$ at $-\infty$. If  
    \begin{align}\label{liminfass}
        \liminf\limits_{t \to +\infty} t^{1/2} |\underline{u}(t,0)| > 0 \ , 
    \end{align} 
    then, for any $\alpha < 1/4$, the bootstrap norm $\|u\|_{\mathcal{B}_\alpha}$ is not bounded.
\end{proposition}
\begin{proof}[Sketch of a proof]
    We argue by contradiction, by assuming that $\|u\|_{\mathcal{B}_\alpha}$ is bounded. This bound allows us to use the stationary phase lemma used in the paper \cite{GPR} (especially Lemmas 5.1, 5.2 and 5.3) to derive an asymptotic ODE for the profile at frequency $k= 0 $. We recall the decomposition of the NSD from Theorem \ref{theomu}
    $$ \mu^\sharp = \mu^\sharp_0 + \mu^\sharp_\pv + \mu^\sharp_L + \mu^\sharp_R \ . $$
    The dynamic of the profile in the distorted Fourier setting is given by
    \begin{align}
	f^{\#}(t,k)  = f^{\#}(0,k)  + i
	\int_0^{t} \big(\mathcal{N}_{0} + \mathcal{N}_{\pv} +\mathcal{N}_{L}
	+ \mathcal{N}_{R} \big)\,ds,
    \end{align}
where, for $\ast\in\{0; \pv; L; R\}$, 
\begin{align}\label{Nast-2}
	\begin{split} 
	\mathcal{N}_{\ast}(s,k)= \iiint e^{is(-k^2+\ell^2-m^2+n^2)}
	f^{\#}(s,\ell)\overline{f^{\#}(s,m)}
	f^{\#}(s,n) \ \mu^{\#}_{\ast}(k,\ell,m,n)\,\frac{dn}{2\pi}\frac{dm}{2\pi}\frac{d\ell}{2\pi} \ . 
	\end{split}
	\end{align}  

By performing the pointwise multilinear estimates, the regular part is well estimated, and the part of the NSD with low frequency improvement has vanishing coefficients at $k = 0$. Those parts will not contribute to the main terms and we get 
\begin{align}
    -i\partial_t f^{\#}(t,0) =  \mathcal{N}_{0}(t,0) +  \mathcal{N}_{\pv}(t,0) + \mathcal{O}(t^{-1-\rho})
\end{align}
with a certain $\rho > 0$. In addtion, the dangerous $\pv$ parts go to the remainder too. Indeed, we apply Lemmas 5.1 and 5.3 of \cite{GPR}. We notice that at $k =0$, the main contribution of the $\pv$ terms is zero due to a cancellation caused by a symmetry in the formula. For the Dirac part, we use Lemma 5.2 of \cite{GPR} to end up with the following asymptotic ODE at frequency $k =0$, 
\begin{align}
     -i\partial_t f^{\#}(t,0) = \frac{A_-^4 + A_+^4 + 6 A_-^2A_+^2}{4\pi t} |f^{\#}(t,0)|^2f^{\#}(t,0) + \mathcal{O}(t^{-1-\rho}) \ . 
\end{align}
This yields the existence of a complex number $f_\infty$ such that  
\begin{align}
    \Big| \exp \Big({-}i c\int_0^t |f^\sharp(s,0)|^2 \frac{ds}{s+1}\Big) f^\sharp(t,0) - f_\infty \Big| \lesssim t^{-\rho}
\end{align}
with $c =\frac{1}{4\pi}(A_-^4 + A_+^4 + 6 A_-^2A_+^2) > 0 $.

This asymptotic combined with the linear asymptotics formula of $\underline{u}(t,x)$ gives, for $t\ge 1$,
\begin{align}
    \underline{u}(t,0) = \frac{1}{\sqrt{-4\pi it}}f_\infty e^{ic|f_\infty|^2\log t} + \mathcal{O}(t^{-1/2 -\rho}) \ . 
\end{align}
We put this asymptotic in the integral \eqref{refinedasymppv}. This implies that the main term of \eqref{refinedasymppv} grows as $t^{1/4}$ that contradicts the boundedness of the bootstrap norm $\|u\|_{\mathcal{B}_\alpha} $ for $\alpha < 1/4$, if $f_\infty \neq 0$. This last condition is ensured by the assumption \eqref{liminfass}.
\end{proof}

%\newpage

\end{document}